\documentclass[12pt]{article}
\usepackage{CJK,CJKnumb,CJKulem,times,dsfont,ifthen,mathrsfs,latexsym,amsfonts,color}
\usepackage{amsmath,amsthm,makeidx,fontenc,amssymb,bm,graphicx,psfrag,listings,curves,extarrows,enumitem}
\usepackage[
linktocpage=true,colorlinks,citecolor=magenta,linkcolor=blue,urlcolor=magenta]{hyperref}

\usepackage{cite}
\usepackage{geometry}
\usepackage{indentfirst}
\usepackage{amssymb}
\makeatletter

\newcommand{\Rmnum}[1]{\expandafter\@slowromancap\romannumeral #1@}
\makeatother

\usepackage[showonlyrefs]{mathtools}  
\mathtoolsset{showonlyrefs=true}

\newtheorem{definition}{Definition}[section]
\newtheorem{theorem}{Theorem}[section]
\newtheorem{lemma}{Lemma}[section]
\newtheorem{corollary}{Corollary}[section]
\newtheorem{proposition}{Proposition}[section]
\newtheorem{remark}{Remark}[section]

\newcommand{\al}{\alpha}
\newcommand{\ga}{\gamma}
\newcommand{\Ga}{\Gamma}

\newcommand{\BR}{{\cal B}}
\newcommand{\CR}{{\cal C}}
\newcommand{\DR}{{\cal D}}
\newcommand{\ER}{{\cal E}}
\newcommand{\FR}{{\cal F}}
\newcommand{\GR}{{\cal G}}

\newcommand{\LR}{{\cal L}}
\newcommand{\LRN}{{\cal L}^n}
\newcommand{\LRI}{{\cal L}^\infty}
\newcommand{\MR}{{\cal M}}

\newcommand{\SR}{{\cal S}} 

\newcommand{\N}{{\mathbb N}}

\newcommand{\Z}{{\mathbb Z}}

\newcommand{\R}{{\mathbb R}}
\newcommand{\RN}{{\mathbb R^N}}  
\newcommand{\bean}{\begin{eqnarray*}}
	\newcommand{\eean}{\end{eqnarray*}}
\newcommand{\loc}{\operatorname{\rm loc}}

\newcommand{\sbr}[1]{\left(#1\right)}
\newcommand{\mbr}[1]{\left[#1\right]}
\newcommand{\lbr}[1]{\left\{#1\right\}}
\newcommand{\hbr}[1]{\left\langle  #1 \right\rangle}

\newcommand{\abs}[1]{\left\lvert#1\right\rvert}

\newcommand{\supp}[1]{{\rm supp} ~#1}

\newcommand{\dx}{ ~\mathrm{d} x}
\newcommand{\nm}[1]{\Vert #1 \Vert}

\newcommand{\nt}[1]{\|#1\|_2} 
\newcommand{\ploc}[1]{|#1|_{N/2,\loc}}

\setitemize{itemindent=38pt,leftmargin=0pt,itemsep=-0.4ex,listparindent=26pt,partopsep=0pt,parsep=0.5ex,topsep=-0.25ex}

\numberwithin{equation}{section}

\begin{document}
	\theoremstyle{plain}

	\title{\bf Infinitely many  positive  solutions to  nonlinear scalar field equation  with nonsmooth   nonlinearity   
	}  
	
	\date{}
	\author{Tianhao Liu\thanks{Institute of Applied Physics and Computational Mathematics,  Beijing,  100088,  China.
			Email: liuthmath@gmail.com},\;\;
		Juncheng Wei\thanks{Department of Mathematics, Chinese University of Hong Kong, Shatin, NT, Hong Kong.
			Email: wei@math.cuhk.edu.hk},\;\; Wenming Zou\thanks{Department of Mathematical Sciences, Tsinghua University, Beijing 100084, China.  Email: zou-wm@mail.tsinghua.edu.cn
		}
	}

	\maketitle
	
	\begin{center}
		\begin{minipage}{140mm}
			\begin{center}{\bf Abstract }\end{center}	\small
			
			This paper investigates the  existence of infinitely many positive solutions for the logarithmic scalar field equation
			\begin{equation}
				\tag{$P$} \label{equ1}
				-\Delta u+  V(x) u= u\log u^2, \quad  u\in  H^1(\RN),
			\end{equation}
		and its counterpart with  prescribed $L^2$-norms
			\begin{equation}  	\tag{$P_N$} \label{equ2}
			\left\{ 	\begin{aligned}
				& 	-\Delta u+  V(x) u +\lambda u= u\log u^2, \quad  u\in  H^1(\RN),\\
				&\int_\RN u^2\dx=a^2>0,
			\end{aligned}	\right.
		\end{equation}
			which come from physically relevant situations. Here,  $N\geq 2$, $V:\RN\to \R$ is a {\it non-symmetric and non-periodic}  potential satisfying certain decay conditions, $ a $ is   prescribed constant, and  $\lambda$   arises as an unknown Lagrange multipliers.  For problem \eqref{equ1}, 	using purely variational methods,  we establish  the existence of multi-bump positive solutions  with either finitely or infinitely many bumps.  For   normalized problem   \eqref{equ2}, we prove the existence of normalized multi-bump positive solutions with a finite  number of  bumps.
			The main difficulty comes from 	the   nonsmooth nature of  logarithmic  nonlinearity, which   introduces some challenges to the variational framework.
			In particular,  the corresponding energy  functional   is not of class $C^1$   on $H^1(\RN)$,  which prevents the direct application of standard critical point theory for $C^1$ functional  or any reduction methods for $C^{1+\sigma}$ nonlinearity.  		The main  ingredients  in this paper  are nonsmooth critical point theory, localized variational methods and a max-min argument.
			To the best of our knowledge, this  paper appears to be the first successful application of the localized variational method to nonsmooth functionals.

			\vskip0.23in
			
			
			\vskip0.23in					
		\end{minipage}
	\end{center}

	\section{Introduction}\label{Sect1}
	
%

		The nonlinear Schr\"{o}dinger equation  
	\begin{equation}  \label{equ general}
		i\frac{ \partial  \psi}{\partial t} =\Delta \psi -  W(x) \psi + f(|\psi|)\psi ,\quad  ~~(x,t)\in\RN\times\R^+,
	\end{equation}
	for a wave function $\psi=\psi(x,t)\in\mathbb{C}$ has attracted  considerable attention due to its fundamental role in modeling phenomena such as Bose-Einstein condensates (BECs) and nonlinear optics \cite{BG-1,BG-2,BG-3}.  Since the groundbreaking achievement of BECs in dilute alkali gases in 1995 \cite{BG-6}, a wide variety of potentials  $W(x)$    have  been extensively investigated in both theoretical and experimental studies (cf. \cite{BG-1,BG-4,BG-5}). 
	Of particular physical relevance in these works are power-type nonlinearities of the form 	\begin{equation}  \label{equ pow}
	f(| \psi|) \psi =  \mu  |\psi|^{p-1}\psi.
	\end{equation}  	In BECs, the cubic nonlinearity  ($p=3$) naturally arises from mean-field two-body interactions, while in nonlinear optics, the Kerr effect  ($p=3$) and higher-order nonlinearities ($p>3$) describe self-focusing and multiphoton processes, respectively. 
  An important generalization involves combined power-type nonlinearities of the form
	\begin{equation}  \label{equ combine}
		 f(|\psi|)\psi = \mu_1|\psi|^{p-1}\psi +\mu_2|\psi|^{q-1}\psi ,
	\end{equation}
	which was   initially considered  by Tao, Visan, and Zhang  \cite{Tao-Visan-Zhang}. Over the past few decades, the mathematical analysis of equation \eqref{equ general} with interactions \eqref{equ pow} and \eqref{equ combine}  has been developed extensively.  As a comprehensive review of the   literature is beyond the scope of this work, we refer the reader to \cite{Berestycki-1,Berestycki-2,cerami_CPMA_2013,Soave=JFA=2020,Weinstein=CMP} and reference therein.

	\medbreak
 	While the interactions  \eqref{equ pow} and \eqref{equ combine} are smooth, there are many physical scenarios involving nonsmooth nonlinearities. A canonical example is the   logarithmic nonlinearity  defined by
 	\begin{equation}\label{defi of log nonlinearity}
 		f(|\psi|)\psi= \begin{cases}
 			\psi \log|\psi|^2, \quad& \text{for} ~~\psi \neq 0,\\
 			0, \quad &\text{for}~~ \psi =0,
 		\end{cases}	 
 	\end{equation} 
 	which arises from the logarithmic Schr\"{o}dinger   equation
 	\begin{equation} \label{E2}
 		i\frac{ \partial  \psi}{\partial t} =\Delta \psi - W(x)\psi +\psi \log|\psi|^2,  ~~(x,t)\in\RN\times\R^+.
 	\end{equation}
 This equation 	was first introduced  in the seminal work of Bialynicki-Birula and Mycielski  \cite{Bialynicki=1976} to describe  {\it the separability of non-interacting subsystems} in quantum mechanics.  Since then,  it has attracted  considerable attention due to its applications in diverse areas of physics,   such as quantum mechanics, quantum optics, transport and diffusion
 	phenomena, information theory, quantum gravity, and the theory
 	of BECs (see \cite{Troy=ARMA}   and the references therein). 
	
		\medbreak
	In the study of   \eqref{E2},  	a fundamental class of solution    is   standing  wave,   which takes the form
	\begin{equation}\label{j1}
		\psi(x,t) = e^{i \omega t}u(x).  
	\end{equation}
Here,  $\omega\in\R$ is the frequency and $u\in H^1(\R^N)$  is  a real-valued function. These solutions are pivotal for understanding the dynamics and stability of  \eqref{E2}, see, for example,	\cite{cazenave-1983,cazenave-haraux,cazenave-lions,Carles=DUKE20118,cazenave}. Substituting the ansatz \eqref{j1} reduces the time-dependent problem  \eqref{E2}   to the   
  logarithmic scalar field equation
 \begin{equation} \label{p equ} \tag{$P$}
 	-\Delta u+  V(x) u= u\log u^2, \quad  u\in  H^1(\RN),
 \end{equation}
 where $N\geq 2$ and $V(x)=W(x)-\omega$ is the external potential.   From another perspective, problem  \eqref{p equ}   can   be obtained   as the limiting case  $p \to 1^+$  of two fundamental physical models (see \cite{Troy=ARMA}):     
 \begin{equation}
 	\begin{aligned}
 		\frac{ \partial  u}{\partial t} =\Delta u - V(x)u+ |u|^{p-1}u- u\quad   &\text{(Reaction-diffusion equation) },\\  
 		\frac{ \partial^2  u}{\partial t^2}=\Delta u - V(x)u+ |u|^{p-1}u- u \quad &\text{ (Nonlinear Klein-Gordon equation)  }.
 	\end{aligned}
 \end{equation}
 Observe that, in contrast to the power   nonlinearity $|u|^{p-1}u$,   the  logarithmic nonlinearity  $g(u)=u\log u^2$ is   sign-changing and strongly sublinear  at the origin
  \begin{equation}
  	\lim\limits_{u\to 0} \frac{g(u)}{u}=-\infty,
  \end{equation} 
  which also corresponds to the infinite-mass case in the literature.  
  	\medbreak
  	Problem \eqref{p equ} admits a natural variational structure.  Formally, its solutions   can be  characterized as  critical points of  the functional
  \begin{equation}\label{Functional P}
  	\mathcal{L}(u) := \frac{1}{2} \int_{\RN} \mbr{|\nabla u|^2+ (V(x)+1)u^2} ~\mathrm{d}x-\frac{1}{2} \int_{\RN} u^2\log  u ^2 ~\mathrm{d}x.
  \end{equation}
  A pivotal tool in studying \eqref{Functional P} is the logarithmic Sobolev inequality. This inequality was first introduced by Gross \cite{Gross-AJM1975} for Gauss measures. In the case of the Lebesgue measure  (see  \cite[Theorem 8.14]{lieb-loss}), it  states that    for any   function  $ u\in H ^1(\RN) $ with  $N\geq 2$  and any constant $ a>0$,
  \begin{equation}\label{log Sobolev ineq}
  	\int_{\RN} u^2 \log u^2 \dx\leq \frac{a^2}{\pi} |\nabla u|_{L^2(\RN)}^2 +\mbr{\log|u|_{L^2(\RN)} ^2 -N(1+\log a)} |u|_{L^2(\RN)} ^2.
  \end{equation}	The existence and explicit form of extremal functions for this inequality were established in     \cite{Carlen-jfa1991,Toscani1997}. We would also like to mention that the logarithmic Sobolev inequality  can be viewed as   a limit case of a family of Gagliardo-Nirenberg-Sobolev inequalities \cite{pino-jmpa2002}  or as a large
  dimension limit of the Sobolev inequality \cite{Beckner-1992}. For more recent developments and related results, we refer to works such as \cite{Brigati-jfa2024, Brigati-ponicare2024}.
  
  	\medbreak
  
It is worth pointing out that  the   \textit{nonsmooth nature} of  logarithmic nonlinearity presents some mathematical challenges in the study of \eqref{p equ}.  On the other hand, it is not locally Lipschitz continuous at the origin. Consequently, standard tools such as the implicit function theorem and the contraction mapping principle are not directly applicable.   

	\medbreak
  
  On the other hand,   the logarithmic nonlinearity      induces   a \textit{nonsmooth variational structure}. In particular,   the energy functional  $\mathcal{L}$   is not differentiable in  $ H^1(\R^N)$;  in fact it is not continuous, only lower semi-continuous, because there exists  a  function  $	\tilde{u}\in H^1(\RN) $ 
  \begin{equation}
  	\tilde{u}(x)=\begin{cases}
  		(|x|^{N/2}\log(|x|))^{-1},&|x|\geq3,\\
  		0, &|x|\leq 2,
  	\end{cases}
  \end{equation}		
  such that $ \int_{\RN} \tilde{u}^2\log  \tilde{u} ^2 ~\mathrm{d}x =-\infty$.   Thus,  
  the	standard critical point theory for $C^1$ functionals  or any reduction methods for $C^{1+\sigma}$ nonlinearities  cannot be directly applied to  \eqref{p equ}.

    \medbreak
  	Over recent decades, extensive research has developed various  techniques to overcome these difficulties and study   the existence, multiplicity, and concentration of solutions to \eqref{p equ}. 
  Since a comprehensive review of prior work  is beyond our scope,  we refer to works   \cite{alves-ji-cvpde2020,avenia-2014,W.Shuai=Nonlinearity=2019,ss-2015,ss-2017,tanaka-zhang-cvpde,Wang=ARMA=2019,Zhang=JMPA=2020,zhangCX-ZhangX-2020,Troy=ARMA,cazenave-1983, pino-jfa2003,alves-ji-isral,WangZhangZhang-ADE2024}  and   references therein.     
  This line of research  follows from  the work on classical nonlinear scalar field equation
 \begin{equation}
	-\Delta u + V(x) u = |u|^{p-1} u,\quad  u \in H^1 (\R^N),
\end{equation}
where $ 1<p <(\frac{N+2}{N-2})_{-}$. 
  Since the celebrated works of Berestycki and Lions \cite{Berestycki-1, Berestycki-2}, the existence of ground state and bound state solutions  for such equations  has been widely studied (see, e.g.,  \cite{ao-wei-cvpde,cerami_CPMA_2013,cerami_Poincare_2015,cerami_jde_2014,cerami-molle-cpde-2019,CotiZelati,pino-wei-yao-cvpde2015,DevillanovaSolimini-cvpde2015,Weinstein=CMP,Molle-atti2015,Molle-cvpde2021,LinLiu-1,LinLiu-2,Wei-Yan-cvpde2010} and references therein).   A topic of particular interest has been the construction of   multi-bump solutions, which is the focus of the present work.

  		  \medbreak
  In this paper, we are concerned with the   existence of    infinitely many    solutions   to \eqref{p equ}.  This direction has attracted significant attention in recent years,  with some substantial results established for potentials $V$ that are constant \cite{avenia-2014}, coercive \cite{ji-szulkin2016}, periodic \cite{ss-2015,tanaka-zhang-cvpde}, or radially symmetric with polynomial decay \cite{WangZhangZhang-ADE2024}.

	\medbreak
	 Nevertheless,
	 the existence  of   infinitely many    solutions  to   \eqref{p equ}    remains largely open in the literature,
	 especially when the potential $V:\RN\to \R$ is   {\it non-symmetric  and non-periodic}.   In a groundbreaking  work  \cite{cerami_CPMA_2013}, Cerami, Passaseo and Solimini  were the first to  use purely variational method to  obtain infinitely many positive solutions to
	 \begin{equation}	\label{Ce1}  \tag{$  P_E$}
	 	-\Delta u + V(x) u = |u|^{p-1} u,\quad  u \in H^1 (\R^N),
	 \end{equation}
	 where $ 1<p <(\frac{N+2}{N-2})_{-}$. They assumed the potential 
	 $V (x) $ is a positive  function in $ L_{\loc}^{N/2} (\R^N)$ such that $V(x)\to V_\infty$ and satisfies the following slow decay assumptions
	 \begin{equation} \label{H3}\tag{A$_0$}
	 	\exists  ~~ \bar \eta \in (0,\sqrt{V_{\infty}}  ), \quad  \text{ such that } ~ \lim\limits_{|x|\to +\infty}  \sbr{V(x)-V_{\infty} } e^{\bar\eta |x|}=+\infty ,
	 \end{equation}
	 without requiring any symmetry property.
	 Many further extensions can be found   \cite{ao-wei-cvpde, cerami_Poincare_2015, cerami-molle-cpde-2019}. In all these papers the nonlinearity is either $C^1$ or further $C^{1+\sigma}$ to ensure that the energy functional is $C^1$,  thereby making the standard critical point theory  applicable. We also refer to \cite{DevillanovaSolimini-cvpde2015,pino-wei-yao-cvpde2015,Molle-atti2015,Molle-cvpde2021} for  an extensive discussion of related results and techniques.


	\medbreak
	This paper aims to extend the  localized variational framework introduced by   Cerami-Passaseo-Solimini \cite{cerami_CPMA_2013} to establish the existence of infinitely many positive solutions for the  nonsmooth problem  \eqref{p equ} with  non-symmetric  and non-periodic potential.    To the best of our knowledge, this work presents the first successful application of the localized variational method to both logarithmic scalar field equation and  nonsmooth functionals,   providing    a fresh perspective on their  study.

	\medbreak
	
	Our first main result   can be stated as follows,  which  appears as  the  first    contribution on this   problem.

	\begin{theorem}\label{thm1}
		Assume  that   $N\geq 2$  and       the potential function $V:\RN\to \R$ satisfies
		\begin{equation} \label{H1}\tag{A$_1$}
			V(x)\in L_{\loc}^{N/2}(\RN),    \quad {  V_0:=\inf_{\RN}V(x)>2}, \quad \lim_{|x|\to +\infty}V(x) = V_{\infty},
		\end{equation} 	
		\begin{equation} \label{H2}\tag{A$_2$}
			\exists  ~~ \bar \zeta \in (0,   \bar \eta), \quad  \text{ such that } ~ \lim\limits_{|x|\to +\infty}  \sbr{V(x)-V_{\infty} } e^{\bar \zeta |x|^2}=+\infty.  
		\end{equation}  Then, there exists   a positive constant  $\mathcal{K}=\mathcal{K}\sbr{ N, \bar \zeta, V_0, V_{\infty}} $ such that if  $$	|V(x)-V_{\infty}|_{N/2,\rm loc}:= \sup_{y\in \RN} |V(x)-V_{\infty}|_{L^{N/2}(B_1(y))} <\mathcal{K},$$
		problem \eqref{p equ} admits infinitely many positive solutions. More precisely, for  every $k\in \N\setminus \lbr{0}$, problem \eqref{p equ}  admits a $k$-bumps  positive solution $\tilde u_k \in H^1(\RN)$.

	\end{theorem}
	
	\begin{remark}
		{\rm   
			The assumption \eqref{H1} is   standard   and ensures  that   $|V(x)-V_{\infty}|_{N/2,\rm loc} <+\infty $. 
			Compared to  the slow decay assumption \eqref{H3} introduced  in \cite{cerami_CPMA_2013},  assumption \eqref{H2} covers  a   wider class of potentials.  Indeed,  condition  \eqref{H2} holds when the potential function $V(x) $  converges to $V_\infty$  at a rate  slower than  that of a  Gaussian function $A e^{-\bar \zeta |x|^2} $, such as polynomial decay,   exponential decay $Ae^{-a|x|}$ ($A,a>0$),   stretched-exponential   decay $Ae^{-a|x|^b}$ ($1< b<2$), or suitable 	Gaussian   decay $Ae^{-c|x|^2}$ ($0<c<\bar{\zeta}$). 
			This difference stems from the decay behavior of   certain candidate functions for \eqref{p equ} and \eqref{Ce1}; see Proposition \ref{prop 3.1} and \cite[Lemma 3.4]{cerami_CPMA_2013}, respectively.

		}
	\end{remark} 
		\medbreak

	We recall that       Cerami et al. \cite{cerami_CPMA_2013} established a key technical lemma \cite[Lemma 3.3]{cerami_CPMA_2013} 
	to  describe  the   decay behavior of the submerged part  of certain candidate functions  for problem \eqref{Ce1}.  This lemma provides a general result on the exponential decay of positive solutions satisfying the linear  inequalities  $-\Delta u+lu\leq 0$ and  $u\leq s$ in $\RN\setminus\DR_0$, where $\DR_0$ is a closed set, and  $l,s >0$. In contrast, our study of \eqref{p equ} involves  a logarithmic nonlinearity $u\log u^2$,  which leads to     different decay properties. Specifically, in   Lemma \ref{lemma newdecay}, we prove   a  Gaussian-type  decay behavior   of positive solutions satisfying the inequalities  $-\Delta u+lu\leq u\log u^2 $ and  $u\leq s  $ in $\RN\setminus\DR$, where $\DR$ is a closed  convex set and $s\in \sbr{0,1}$.     The presence of the logarithmic term $u\log u^2$ (which is negative for $u\leq s<1$) alters the decay behavior, leading to a faster, Gaussian-type decay behavior compared to the exponential decay in the linear case.  Furthermore, we  note that our methods in the proof of Lemma \ref{lemma newdecay} may not achieve decay rates faster than the Gaussian-type decay.
	Hence, the quadratic exponent $\bar{\zeta}|x|^2$ in \eqref{H2} seems to be  sharp and cannot be  improved (e.g., $\bar{\zeta}|x|^{q}$ for $q>2$).
		\medbreak
	
	The proof of Theorem \ref{thm1} is based on  the  variational method.  However,  as previously mentioned, the presence of   logarithmic  nonlinearity   introduces some challenges to the variational framework.
	In particular,  the corresponding  functional   is not of class $C^1$   on $H^1(\RN)$,  which prevents the direct application of standard critical point theory for $C^1$ functional  to problem \eqref{p equ}.  Inspired by \cite{szulkin-poincare},  we shall  look for a critical point of the energy functional in the  sub-differential sense.  	The main technical tools in this paper  are a combination of nonsmooth critical point theory, localized variational methods and max-min argument. See Subsection \ref{Section 1.2} for an outline of the proof.

	\vskip 0.1in
	
	The  proof  of Theorem \ref{thm1}  indicates  the   following asymptotic behaviors of  multi-bump solutions  $\tilde u_{k} $ when $\ploc{V(x)-V_{\infty}} \to 0 $.

	\begin{theorem}\label{thm3}
		Let $V_n(x)$ be any sequence of functions satisfying    \eqref{H1}, \eqref{H2} and  $$\ploc{V_n(x)-V_{\infty}}\to 0 \quad \text{ as } n\to +\infty .$$  Then there exists $\bar{n}\in\N$ such that for all $n>\bar{n}$ and for all $k\in \N\setminus\lbr{0}$, there exist  $k$ points $   z_1^n, \ldots,  z_{ k}^n $ in $\RN$    and a positive solution $\tilde u_k^n$ of \eqref{p equ} with potential $V_n(x)$ such that  
		\begin{equation}
			\lim_{n\to+\infty} \min\lbr{ |  {z}_i^n-  {z}_j^n|: i\neq j, ~ i,j=1,\ldots,k}=+\infty,
		\end{equation}
		and for all $r>0$   , there holds
		\begin{equation}
			\lim_{n\to+\infty} \sbr{ \sup \lbr{| \tilde u_k^n(x+   z_i^n) -U(x)| : i\in \lbr{1,\dots,k}, |x|\leq r}}=0 ,
		\end{equation}
		where $U(x)$ is the positive  ground state solution of the limit problem \eqref{p infty}, see \eqref{defi of Gausson}.
	\end{theorem}
	
Motivated by \cite{cerami_Poincare_2015}, we also investigate  the asymptotic behavior of the  multi-bump solutions  $\tilde u_{k}$ in Theorem \ref{thm1}  as   the number of bumps $k$ increases  to infinity.

	\begin{theorem}\label{thm2} 
		Under the assumptions of Theorem \ref{thm1}, then \eqref{p equ} admits a positive solution  $\tilde{u} \in H_{\loc}^1(\RN)$,  which emerges around an unbounded sequence $\lbr{\tilde z_n:n\in \N}$ such that $ \tilde z_n \neq \tilde z_l $ if $n\neq l$.  Moreover, 	 we have 
		\begin{equation} \label{et1}
			\lim_{n\to+\infty} \min\lbr{ |\tilde{z}_n-\tilde{z}_l|: l  \in \N, ~ n\neq l}=+\infty,
		\end{equation}
		and
		\begin{equation} \label{et2}
			\lim_{n\to+\infty} \tilde{u}(x+\tilde z_n)= U(x) ~~ \text{	uniformly on any compact subset }~K\subset \RN,
		\end{equation}
		where $U(x)$ is the positive  ground state solution of the limit problem \eqref{p infty}, see \eqref{defi of Gausson}.
	\end{theorem}
	\vskip0.1in

	Theorem \ref{thm2} shows that \eqref{p equ} admits  a positive solution of infinite energy that has infinitely many positive bumps. We now present the 	 key steps in the proof of Theorem \ref{thm2}. First, we  construct a sequence of positive solutions $\lbr{v_n:n\in\N}$  as a subsequence of the multi-bump positive solutions given by Theorem \ref{thm1},  along with a sequence of balls $B(0,r_n)$ such that    $v_n$ has at least $m$ positive bumps emerging around the points within $B(0,r_m)$ for all $n,m\in\N$.  Next, we prove that ${v_n}$ is uniformly bounded in $H^1(B(0,R))$ for any fixed $R > 0$. Passing to the weak limit, we obtain a solution to \eqref{p equ} with infinitely many positive bumps. 
	Finally, the relations \eqref{et1} and \eqref{et2} are a consequence of the variational characterization of  multi-bump positive solutions $ \tilde{u}_k$  found  in Theorem \ref{thm3}.   Relation \eqref{et1} indicates  the bumps of $\tilde{u}$ become increasingly sparse as the distance from the origin increases.    In this  procedure, 	the    decay  assumption \eqref{H2} on  $V(x)$ plays a fundamental role and provides the key motivation for   obtaining multi-bump solutions,   because the attractive effect of  $V(x)$  dominates  the repulsive disposition of positive masses
	with respect to each other, preventing the bumps from escaping to infinity.

	\begin{remark} \label{remark2}
		{\rm We make some further comments on the study of 	multi-bump solutions  to  \eqref{p equ}. Some  progress under structural assumptions on   $V$
			has been made  via variational methods and Lyapunov-Schmidt reduction, etc., see \cite{tanaka-zhang-cvpde,alves-ji-Sci,WangZhangZhang-ADE2024,alves-ji-isral,Luo-Niu}.
					 The works \cite{tanaka-zhang-cvpde,alves-ji-Sci,WangZhangZhang-ADE2024} consider the problem \eqref{p equ}.  
			  \cite{tanaka-zhang-cvpde} investigated  \eqref{p equ} with  $1$-periodic potentials $V$   of class   $C^1$. By   considering  a   $2L$-periodic problems ($L\gg1$),   they  proved the existence of infinitely many  multi-bump solutions  which are distinct under the $\Z^N$-action. 
				In \cite{alves-ji-Sci}, it was  assumed that the nonnegative potential $V(x)=\tilde{\lambda}Z(x)$  has potential well $ \Omega_1:=\text{int}Z^{-1}(0)$
				which possesses $k$ disjoint bounded components. They  proved that    equation \eqref{p equ} admits  at least $2^k-1$ multi-bump positive  solutions  for large $\tilde{\lambda} $. In a recent study, the work \cite{WangZhangZhang-ADE2024} assumed that   the potential $V$ is   bounded and radially symmetric   $V (x) = V (|x|)$ satisfying $	V(r) =V_0\pm\frac{a}{r^m}+o(\frac{1}{r^m})$, as $  r\to+\infty$,	 
				where  $V_0\in\R$, $a>0$ and $m>1$.  Using   Lyapunov-Schmidt reduction, they proved the existence of infinitely many multi-bump positive and nodal solutions to 
				\eqref{p equ} by analyzing the asymptotic  the multi-bump solutions of perturbed problems   obtained in \cite{Wei-Yan-cvpde2010} through a limiting process. 
			  In \cite{alves-ji-isral,Luo-Niu}, the authors   investigated the semiclassical states of the logarithmic  scalar field equations $ -\varepsilon^2\Delta u+ V(x)u= u\log |u|^2   $ for  $\varepsilon>0 $   small,
				%
			which  can be written as
			\eqref{p equ}	 with potential $V(\varepsilon x)$ by a
			change of variables.    The topology of the sublevel sets of $V$, as well as
			the number and type of its critical points, are shown to strongly influence the existence and	multiplicity of multi-bump solutions. }
	\end{remark}

	\vskip0.05in
	
		Equation \eqref{E2} is Hamiltonian and thus conserves the mass $	M(t) :=	\int_{\RN} |\psi(x,t)|^2 dx.$
	Thus, it is natural to   consider the   logarithmic scalar field equation with prescribed $L^2$-norms
	\begin{equation} \label{normalized problem} \tag{$P_N$}
		\left\{ 	\begin{aligned}
			& 	-\Delta u+  V(x) u +\lambda u= u\log |u|^2, \quad  u\in  H^1(\RN),\\
			&\int_\RN u^2\dx=a^2>0,
		\end{aligned}	\right.
	\end{equation}
	which is  called the {\it normalized problem} in the literature. 	Here, $ a $ is a prescribed constant,   the $L^2$-constraint  reflects the   conservation of mass and the parameter $\lambda$   arises as  Lagrange multiplier 	 that is unknown.  Several results  on the  existence of   normalized solutions to   \eqref{normalized problem} can be found in \cite{cazenave-1983,avenia-2014,alves-ji-jga}. Of particular interest to us are   the existence of normalized solutions in the form of multi-bump solutions, which is  referred to as {\it normalized multi-bump solutions}.  This line of research was first proposed by Ackermann and Weth in \cite{Ackermann}, where they considered the equation
	\begin{equation} \label{normalized problem1}
		\left\{ 	\begin{aligned}
			& 	-\Delta u+  V(x) u +\lambda u=f(u), \quad  u\in  H^1(\RN),\\
			&\int_\RN u^2\dx=a^2,
		\end{aligned}	\right.
	\end{equation}
	for  a periodic potential $V$ and a class of autonomous nonlinearities.  As a by-product of  Theorem  \ref{thm1}, we have the following  existence  results.
	\vskip 0.1in

	\begin{theorem}\label{thm4}
		Under the assumptions of Theorem \ref{thm1}, then  the normalized problem  \eqref{normalized problem}  admits infinitely many  normalized multi-bump positive solutions. More precisely,   for any $k\in \N\setminus\lbr{0}$,  problem   \eqref{normalized problem}  admits a  normalized positive solution with $k$-bumps.
	\end{theorem}

	Theorem \ref{thm4} follows directly from Theorem \ref{thm1} due to the close connection between the  solutions of \eqref{p equ} and  \eqref{normalized problem}. Indeed, for any $k\in \N\setminus\lbr{0}$,  let $\tilde{u}_k $ be the   multi-bump    solution of \eqref{p equ} obtained in Theorem \ref{thm1} with mass $\tilde a_k^2 := |\tilde{u}_k |_2^2<+\infty$. Then, the rescaled function   $ u_k=  e^{ {  \lambda_k }/{2}} \tilde{u}_k $, where  $ {\lambda_k} = \log a^2-  \log \tilde a_k ^2 $, solves problem \eqref{normalized problem}  with   mass $|u_k|_2^2=a^2$.  Moreover, combining this relationship with Theorem \ref{thm3}, we can characterize the asymptotic behavior of the normalized multi-bump solutions as $\ploc{V(x)-V_{\infty}} \to 0 $.

	\vskip 0.2in
	\subsection{Sketch of the proof of Theorem			\ref{thm1} } \label{Section 1.2}

	In this subsection, we outline the main steps and ideas for proving Theorem \ref{thm1}. Our approach follows the variational framework and max-min argument introduced in \cite{cerami_CPMA_2013}.
	
	\begin{itemize}
		\vskip0.1in
		\item[Step 1.]        We define a configuration space consisting of the points around which multi-bump solutions emerge,  and  construct some localized   Nehari constraints   to  look for  such  solutions. See Subection \ref{Sect3.1}.
		\vskip0.04in
		\item[Step 2.]  We consider a minimization problem for the functional
		$\LR$ over  the  localized Nehari constraint, and prove the existence of    minimizer. 	Afterwards,  we   describe some important feature of  the    submerged    and emerging part of  the minimizers.   See Subsection \ref{Sect3.2}. 
		\vskip0.04in
		\item[Step 3.]  We consider a maximization problem among these minimizers as the  points vary in the configuration space.   As part of this step, we obtain some rigorous   priori estimates that will be fundamental in later arguments.  See Subsection \ref{Sect3.3}.
		\vskip0.04in
		\item[Step 4.] We describe the the asymptotic behavior  of the max-min functions as  $\ploc{V(x)-V_{\infty}} $ goes to $0$. See Section \ref{Sect4}.
		\vskip0.04in
		\item[Step 5.] We prove that  when $\ploc{V(x)-V_{\infty}} $  is small, the max-min function is indeed a
		 critical point,  and hence a solution of \eqref{p equ}. See Section   \ref{Sect5}.
		\vskip0.08in
	\end{itemize}

	We now highlight the ideas and  main differences between our proof and that in \cite{cerami_CPMA_2013}.
	In the minimization procedure,   in order to  prove the  boundedness of    minimizing sequence, an important logarithmic inequality is required.   Then by analyzing the compactness of the minimizing sequence, we proved the existence of minimizers.
	Afterwards,  we   describe some important feature of the emerging part   and submerged part   of  the minimizers. 
	
	 On the one hand, due to the lack of $C^1$-smooth of the functional $\LR$ in $H^1(\RN)$,  we apply the nonsmooth critical point theory   from  \cite{ss-2015,szulkin-poincare} and look for critical point in the sub-differential sense.  Then we     prove that  the submerged part of the minimizer solves \eqref{p equ} in the whole space $\RN$ except for the support of    emerging parts. Moreover, we  describe the {\it Gaussian-type asymptotic decay} of the submerged part (see Proposition \ref{prop 3.1}). The decay property follows from an interesting technical lemma   (see Lemma \ref{lemma newdecay}) concerning the Gaussian-type decay of solutions satisfying certain variational inequalities, which    seems to be new in the study of logarithmic scalar field equations.

	  On the other hand,   about the emerging parts, since functional  $\LR$ is of class $C^1$ on $H^1(\Omega)$ for    bounded domain $\Omega$ (see \cite[Lemma 2.2]{ss-2015}),  we obtain that  $\LR^\prime $ satisfies relations \eqref{e5} in which Lagrange multipliers appear by using a deformation argument, see Proposition \ref{lemma 3.2}.
	
	In the maximization procedure,  the bumps may escape to infinity; thus, some key priori estimates are required to  prevent this from happening.
	For     multi-bump positive solutions with $k$ positive bumps,  a key step is to  show that the max-min level $\varLambda_{k}$ satisfies $\varLambda_{1}>\CR^\infty$ and $\varLambda_{k+1}>\varLambda_{k}+\CR^\infty$ for $k\geq1$, where $\CR^\infty$ is the ground state level of limit problem \eqref{p infty}, see Propositions \ref{prop 3.2} and \ref{prop 3.3}. The proof is based on a mathematical induction on the number of bumps $k$.   It is worth noting   that
	the    decay  assumption \eqref{H2} on  $V(x)$ plays a fundamental role in the analysis.  This condition \eqref{H2} provides the key  guarantees for   obtaining multi-bump solutions,   because the attractive effect of  $V(x)$  dominates  the repulsive disposition of positive masses
	with respect to each other, preventing the bumps from escaping to infinity.
	We also emphasize that the logarithmic nonlinearity causes some technical challenges, requiring some key observations and new ideas in our proofs.

	\subsection{Further notations}
	In this subsection, we introduce some further notations.
	\medbreak
	$\bullet$ For any metric space $X$ and any  $x\in X$, $r>0$, we denote  by  $B(x,r) $  the open ball  in $X$ with  radius $r$  centered at $x \in X$.  For any measurable set $\mathcal{O} \subset \RN$,  $|\mathcal{O}|$ denotes its Lebesgue measure.
	\medbreak
	$\bullet$  Let $\Omega$  be  an open set in $\RN$, we denotes by $H^1(\Omega)$  the closure of $C^\infty(\Omega)$ with respect to the norm $$\nm{u}_\Omega=\mbr{ \int_{\Omega}\sbr{  |\nabla u|^2+ u^2} \mathrm{d} x}^{1/2}. $$   The norm in $H^1(\RN)$ is denoted by $$	\nm{u}_H :=\mbr{ \int_{\RN}\sbr{  |\nabla u|^2+u^2} \mathrm{d} x}^{1/2}.$$
	\medbreak
	$\bullet$ 	For any  $\Omega \subseteq \RN$ and $1\leq p \leq \infty$,  the Lebesgue space $L^p(\Omega)$ is equipped with the norm denoted by $|\cdot|_{p, \Omega}$ when $\Omega$ is a proper subset of $\mathbb{R}^N$, and by $|\cdot|_p$ when $\Omega = \mathbb{R}^N$.
	\medbreak
	$\bullet$
	The space $L^p_{ \rm loc}(\RN) $ (resp.  $H^1_{\rm loc}(\RN)$) denotes the set of functions $u$ such that  $\forall x\in \RN$ there exists an open set $ \Omega \Subset \RN$ so that $x\in \Omega$ and $u|_{\Omega} \in L^p(\Omega)$ (resp. $u|_{\Omega} \in H^{1,2}(\Omega)$).
	Moreover,   we consider the subspace of functions belonging to $L^p_{ \rm loc}(\RN) $, referred to as "uniformly locally
	$p$-summable," such that the norm
	$$|u|_{p, \rm loc }:=\sup_{y\in \RN}  |u|_{p, B(y,1)}<+\infty.$$
	\medbreak
	$\bullet$  We use ``$\to$''  (resp.  ``$\rightharpoonup$'' ) to denote the strong (resp. weak)  convergence    in corresponding space.
	\medbreak
	$\bullet$   The capital letter $C$ will denote a constant that may vary from line to line, while  $ C_1, C_2,C_3,C_4  $ represent prescribed constants.
	\subsection{Structure of the paper}
	The paper is organized as follows. In Section \ref{Sect2}, we present some preliminary results, including the known results about the limit problem \eqref{p infty} and  the nonsmooth critical point theory.   In  Section \ref{Sect3}, we onstruct the localized Nehari constraint where we look for the   multi-bump functions of \eqref{p equ}, and adopt a     max-min  argument to find good candidates for critical points  of $\LR$.   The argument is like that used in  in \cite{Molle-jfa2010,Molle-poincare2010}  for  elliptic equations with jumping nonlinearities, as well as in \cite{cerami_CPMA_2013,cerami_jde_2014} for  the  elliptic equation with power-type  nonlinearity. 	However, due to the features of logarithmic nonlinearity, these existing arguments cannot be directly adapted to our setting.  Consequently, our proofs require new ideas and techniques.  Section \ref{Sect4} is devoted to the asymptotic behaviors of those candidate function, and prove   the properties stated in Theorem \ref{thm3}. In Section \ref{Sect5}, we show that the candidate function is indeed the critical point  of $\LR$ and  thus solutions  of \eqref{p equ}, which completes the proof of Theorem \ref{thm1}.  Finally, in Section \ref{Sect6}, we prove Theorem \ref{thm2}.
	
	\section{Preliminary results} \label{Sect2}
	
	We  first recall some known facts about the following limit problem
	\begin{equation}\tag{$P_{\infty}$}  \label{p infty}
		-\Delta u+ V_{\infty}u= u\log |u|^2, \quad  u\in H^1(\RN).
	\end{equation}
	The variational functional   related to problem \eqref{p infty} is
	denoted by
	\begin{equation}\label{Functional P infty}
		\mathcal{L}^\infty(u) := \frac{1}{2} \int_{\RN} \mbr{|\nabla u|^2+ (V_\infty+1)u^2}~\mathrm{d}x-\frac{1}{2} \int_{\RN} u^2\log u^2~\mathrm{d}x.
	\end{equation}
	Then we have the following results on existence, uniqueness  and non-degeneracy of ground state solution of \eqref{p infty}, see \cite{ss-2015,Troy=ARMA,avenia-2014,Bialynicki=1979}.
	\begin{lemma}\label{lemma pinfty} We have the following conclusions.
		\noindent \begin{itemize}
			\item[(1)] \eqref{p infty} has a  positive ground state  solution (unique up to translation),  known as the Gausson, denoted by
			\begin{equation}\label{defi of Gausson}
				U (x):=\exp\lbr{\frac{ (V_\infty +N-|x|^2)}{2}}.
			\end{equation}
			\item[(2)]	Set  $\CR^\infty:=   \LR^\infty(U)$,  then  $	\CR^\infty$ can be characterized as
			\begin{equation}
				\CR^\infty=\inf_{h\in\Ga} \max_{t>0} \LR^\infty(h(t)),
			\end{equation}
			where
			\begin{equation} \label{defi of Ga }
				\Ga:=\lbr{h \in C([0,1],H^1(\RN)):~h(0)=0, ~~ \LR^\infty(h(1))<0 }.
			\end{equation} 	
			\item[(3)] The Gausson $U$ is    non-degenerate in $H^1(\RN)$, that is  $$  	{\rm	Ker }(K)={\rm span}\lbr{\partial_{x_j}U: 1\leq j \leq N  } ,$$
			where $ 	K(v)= -\Delta v  + \sbr{V_{\infty}-2-\log U^2}v  $ is the linearized operator  for \eqref{p infty} at   $U$.
		\end{itemize}	
	\end{lemma}
	
	\vskip0.2in

	As introduced earlier, the functionals $\mathcal{L}(u)$  and $\LRI(u)$   fail  to  be well-defined and differentiable on $H^1(\RN)$. Hence, the critical point theory for $C^1$ functional can not be applied in the study of logarithmic  scalar field equation. We shall  look for a critical point in the  sub-differential sense.    Here we state some definitions that can be found in   \cite{szulkin-poincare}.
	\begin{definition} \label{definition 2.1}	{\rm
			Let $E$ be a Banach space,   	$E'$ be the dual space of $E$ and $\left\langle \cdot,\cdot \right\rangle $ be the duality paring between $E'$  and $E$. Let $\LR: E \to \R$ be a functional of the form $\mathcal{L}(u)=\Phi(u)+\Psi(u)$,
			where 	$\Phi\in  C^1(E,\R)$ and  $\Psi$  is convex and lower semi-continuous.
			\begin{itemize}
				\item[(i)] The sub-differential $\partial \LR(u)$  of the functional $\LR$ at a point  $u\in E$ is the following set
				\begin{equation}
					\partial \LR(u)=	\lbr{ w\in E': \left\langle \Phi^\prime(u),v-u \right\rangle +\Psi(v)-\Psi(u)\geq \hbr{ w,v-u}, ~~ \forall v\in E}.
				\end{equation}\medbreak
				\item[(ii)] A critical point of $\LR$    is a point  $u\in E$ such that  $\LR(u)<+\infty$ and  $0\in \partial \LR(u)$, i.e.,
				\begin{equation} \label{dq1}
					\left\langle \Phi^\prime(u),v-u \right\rangle +\Psi(v)-\Psi(u)\geq 0, ~~ \forall v\in E .
				\end{equation}
				\medbreak
				\item[(iii)] A Palais-Smale sequence at level $d$ for $\LR$   is a sequence  $ \lbr{u_n:n\in \N}  \subset E $ such that  $\LR(u_n) \to d$ and there is a numerical sequence $\varepsilon_n \to 0^+$ with
				\begin{equation}
					\left\langle \Phi^\prime(u_n),v-u_n \right\rangle+\Psi(v)-\Psi(u_n)\geq  -\varepsilon_n \nm{ v-u_n}, ~~ \forall v\in E .
				\end{equation}
				\medbreak
				\item[(iv)]  The set  $D(\LR):=\lbr{u\in H^1(\RN): \LR(u)<+\infty}$  is called the effective domain of $\LR$.
		\end{itemize}}
	\end{definition}
	\vskip0.12in
	In what follows, we consider $E=H^1(\RN)$. In order to decompose the functional  $\mathcal{L}(u)$ into the  sum of a $C^1$ functional $ \Phi$,  and a convex lower semi-continuous functional $\Psi$,  we adapt the technical approaches developed in \cite{ss-2015,ss-2017}.
	For sufficiently small $\sigma>0$, we  decompose $\frac{1}{2}s^2\log s^2$  into
	\begin{equation*}	
		\frac{1}{2}s^2\log s^2:=   	F_2(s) -F_1(s) ,
	\end{equation*}
	where
	\begin{equation*}
		F_1(s):=
		\begin{cases}
			-\frac{1}{2}s^2\log s^2, & \text{for}~|s|<\sigma,\\
			-\frac{1}{2}s^2 \sbr{\log \sigma^2+3} +2\sigma |s|-\frac{1}{2}\sigma^2 ,    &\text{for}~|s|\geq \sigma,
		\end{cases}
	\end{equation*}
	and
	\begin{equation*}
		F_2(s):=	\begin{cases}
			0, & \text{for}~ |s|\leq \sigma, \\
			\frac{1}{2}s^2  \log \sbr{s^2/\sigma^2}  +2\sigma |s| -\frac{3}{2}s^2-\frac{1}{2}\sigma^2,    &\text{for}~|s|> \sigma.
		\end{cases}
	\end{equation*}
	Then we can rewrite the functional   $\mathcal{\LR}(u):=\Phi(u)+\Psi(u)  $, where
	\begin{equation}\label{defi of Phi}
		\Phi(u):= \frac{1}{2} \int_{\RN} \mbr{|\nabla u|^2+ (V(x)+1)u^2} ~\mathrm{d}x-  \int_{\RN} F_2(u) ~\mathrm{d}x,
	\end{equation}
	and
	\begin{equation}\label{defi of Psi}
		\Psi(u):=   \int_{\RN} F_1(u) ~\mathrm{d}x.
	\end{equation}
	It is straightforward to verify that    $F_1$ and $F_2$ satisfy the following  properties:
	\vskip 0.1in
	\begin{itemize}
		\item[(P1)]  $F_1\in C^2(\R,\R)$ is   a  convex function with $F_1(s)\geq 0$ and $F_1^\prime(s)s\geq 0$ for all $s\in \R$.
		\medbreak
		\item[(P2)] $F_2\in C^2(\R,\R) $, and for each fixed $q\in(2,2^*),$ there exists $T_q>0$ such that
		\begin{equation}\label{inequ B}
			|F_2'(s)|\leq T_q |s|^{q-1} \quad \text{for~all}~ s\in \R.
		\end{equation}
		\item[(P3)]  There exists $C_\varepsilon>0$ such that for any $u,v \in \R$ there holds
		\begin{equation}\label{r6}
			| F_2(u)-F_2(v) | \leq C_\varepsilon |u-v| \sbr{ |u|+|v|}^{1+\varepsilon}.
		\end{equation}
	\end{itemize}
	Then we deduce from \eqref{inequ B} that 	$\Phi \in  C^1(H^1(\R^N),\R)$, see \cite[Lemma 3.10]{william=1996}.  Moreover,  it is straightforward to verify that  $\Psi $  is a positive,  convex and  lower semi-continuous functional, see \cite[Proposition 2.9]{avenia-2014}; however,  $\Psi $ is    not   $C^1$-smooth on $H^1(\R^N)$.
	For $u\in D(\LR) $, let
	\begin{equation} \label{ll1}
		\hbr{\LR^\prime(u), v} =\hbr{\Phi^\prime(u), v} +\int_{   \RN } F_1'(u)v \dx, \quad \forall v\in C_0^\infty(\RN).
	\end{equation}
	Then $ \LR^\prime(u)$ is a densely defined linear operator, and we may define
	\begin{equation}
		\nm{\LR^\prime(u)}:=\sup\lbr{  \hbr{\LR^\prime(u), v}:~ v\in C_0^\infty(\RN) ~\text{ and }~ \nm{v}_H \leq 1}.
	\end{equation}
	If  $	\nm{\LR^\prime(u)}$ is finite, then $  \LR^\prime(u)$ may be extended to a bounded operator in $H^1(\RN)$ and is thus an element of  $H^{-1}(\RN)$. Define 
	\begin{equation} 
		\partial \LR(u)=	\lbr{ w\in  H^1(\RN):{  \left\langle \Phi^\prime(u),v-u \right\rangle +\Psi(v)-\Psi(u)\geq  <w,v-u >}, ~ \forall v\in C_0^\infty(\RN)}.
	\end{equation} 
	
	\begin{lemma} {\rm (see \cite{ss-2017})} \label{lemma 2.23}
		Let  $ u\in D(\LR)$ and $	\nm{\LR^\prime(u)}$ be finite, then $w\in\partial \LR(u)$
		if and only if  $ w= \LR^\prime(u)$. If $\lbr{u_n }_n $ is a Palais--Smale sequence, then $ \LR^\prime(u_n) \in H^{-1}(\RN)$ and $ \LR^\prime(u_n) \to 0$.
	\end{lemma}
	
	\vskip 0.12in
	\begin{lemma}\label{lemma 2.7}{\rm (see \cite{ss-2017})}
		If $u\in D(\LR)$ is a critical point of $\LR$,  then $u$ is a solution of  \eqref{p equ}.
	\end{lemma}
	
	\vskip 0.12in
	Now we  recall the following logarithmic inequality, whose proof can  be found in  \cite[Page 153]{pino-jfa2003}.
	\begin{lemma} \label{lemma 2.9}
		There are  constants $A$, $B>0$ such that
		\begin{equation}
			\int_{   \RN } u^2 \log u^2 \dx \leq  A+B \log  \nm{u}_{H}, \quad \text{ for any } ~~ u\in H^1(\RN) \setminus \lbr{0}.
		\end{equation}
	\end{lemma}
	
	As an immediate consequence we have the following result.
	\begin{corollary} \label{coro 2.2}
		There are  constants $C_0, R_0>0$ such that if $u\in H^1(\RN)$ and $\nm{u}_{H} \geq R_0$, then
		\begin{equation}
			\int_{   \RN } u^2 \log u^2 \dx \leq C_0(1+\nm{u}_H).
		\end{equation}
	\end{corollary}
	\vskip 0.12in
	\begin{lemma}\label{lemma 2.10}
		Let $\lbr{u_n }_n $    be a  sequence with finite energy. Then $\lbr{u_n }_n $ is   bounded   in $  H^1(\RN)$. Moreover, $\lbr{u_n^2 \log u_n^2}_n  $ is also   bounded   in $  L^1(\RN)$, and hence,  $u_n \in D(\LR)$.
	\end{lemma}

	\begin{proof}
		Under the assumption, there is $M_0>0$ such that for all $n\in \N$,
		\begin{equation}\label{sss1}
			\begin{aligned}
				M_0   \geq  \LR (u _n) \geq
				\frac{1}{2}  \nm{u _n }_H^2  -\frac{1}{2} \int_{\RN} u_n^2\log u_n^2 ~\mathrm{d}x.
			\end{aligned}
		\end{equation}
		Then we argue by contradiction, and   assume that for any $R>R_0$, there is $n\in\N$ such that $ \nm{u _{n} }_H\geq R> R_0$. By   Corollary \ref{coro 2.2}, we have
		\begin{equation}
			\nm{u _{n}}_H^2\leq 2M_0+ C_0(1+	\nm{u _{n}}_H),
		\end{equation}
		which is impossible since $R$ can be  arbitrarily chosen. Therefore,  $\lbr{u_n}_n $ is   bounded   in $  H^1(\RN)$. This,   together with  \eqref{sss1},  shows that $ \abs{u_n^2 \log u_n^2 }_1$ is bounded from below. The fact that   $ \abs{u_n^2 \log u_n^2 }_1$   has a  upper bound comes directly from Lemma \ref{lemma 2.9}.      The  proof is completed.
	\end{proof}

	\section{ The    max-min argument} \label{Sect3}
	In this section, we adopt a max-min argument to find  good candidates for critical points of $\LR$. This section will be divided into three parts. In the first part (Section \ref{Sect3.1}), we define a configuration space  and  construct a localized Nehari constraint, which  contains the functions emerging around a fixed $k$-tuple  of points within the configuration space. In the second part (Section \ref{Sect3.2}),  we consider a minimization problem on such constraint and find a minimizer associated with the aforementioned  $k$-tuple of points. Finally, in the third part (Section \ref{Sect3.3}), we solve a maximization problem  the  among these minimizers as the   $k$-tuple of points vary in the configuration space.

	\subsection{ The localized  Nehari constraint} \label{Sect3.1}
	In this  section,  we will construct some  localized Nehari constraints, which  contain  the functions emerging around  some  points within a suitable configuration space.
	\medbreak
	First,  we fix a sufficiently small parameter $\delta$   such that
	\begin{equation} \label{defi of delta}
		0<	\delta <\min\lbr{1, e^{(V_0-2)/2}, e^{-1/2}}= e^{-1/2},
	\end{equation}
	then  by \eqref{defi of Gausson}, we can choose $R_{\delta}>0$ such that
	\begin{equation}\label{defi of Rdelta}
		U(x)<\delta \quad \forall ~ x\in \RN\setminus B(0,R_\delta/2).
	\end{equation}
	Throughout this paper,  we set \begin{equation}\label{defi of vartheta}
		\vartheta(x):=V(x)-V_{\infty} ,
	\end{equation} 
	and  for every nonnegative function $u \in H^1(\RN)$, we set 
	\begin{equation} \label{emerging part}
		u^{\delta}(x)=u\vee \delta :=(u-\delta)^+(x),
	\end{equation}
	and 
	\begin{equation}\label{submerged part}
		u_{\delta}(x)=u \wedge\delta:=u(x)-u^\delta(x),
	\end{equation}
	where   $u^+:=\max\lbr{u,0}$, and we call  $	u^{\delta}$ and $	u_{\delta} $ the emerging  and the submerged part of $u$ respectively. Also, we say that a function $u\in H^1(\RN)$ is emerging around the $k$ -tuple  points $z_1,\ldots,z_k\in \RN$ (in $k$ balls of radius $R_\delta$) if 
	\begin{equation}
		u^\delta(x)=\sum_{i=1}^{k}u^{\delta,i}(x),
	\end{equation}
	where for all $i\in\lbr{1,\ldots,k}$, $u^{\delta,i}(x)\geq 0$, $u^{\delta,i}(x)\not \equiv 0$, and $u^{\delta,i}(x)=0$ $\forall x\notin B(z_i,R_\delta)$. Consequently, if $B(z_i,R_\delta) \cap B(z_j,R_\delta)=\emptyset$ for $i\neq j$,  the function $u^{\delta,i}$ is the projection of $u^\delta$ in $H_0^1(B(z_i,R_\delta))$ and we know that
	\begin{equation}
		\sbr{	\RN \setminus \bigcup_{i=1}^k B(z_i,R_\delta)} \subset \sbr{	\RN \setminus  \supp{u^\delta}}.
	\end{equation}
	Now we define some  sets depending on the chosen real number  $R_\delta$ satisfying \eqref{defi of Rdelta}, which  are crucial in looking for solutions.  For all $k\geq 1$, we define the   configuration space
	\begin{equation} \label{defi of configuration space}
		\mathcal{E}_1=\RN, ~	\mathcal{E}_k=\lbr{ \sbr{z_1,\ldots,z_k}\in \sbr{\RN}^k: |z_i-z_j|\geq 3R_\delta ~ \text{ for }~ i\neq j, ~~i,j=1,\ldots,k}, \forall k>1.
	\end{equation}
	Then for   all $ \sbr{z_1,\ldots,z_k}\in \mathcal{E}_k$, since functional  $\LR$ is of class $C^1$  on $H^1(\Omega)$ for    bounded domain $\Omega$,  we consider the  localized Nehari  constraint 
	\begin{align} \label{defi of S}
		\mathcal{S}_{z_1,\ldots,z_k}:=\lbr{u\in H^1(\RN): u\geq 0,~ u~ \text{ emerging around }~ z_1,\ldots,z_k, \right. \\ \left.\text{ and } \hbr{\mathcal{L}'(u),u^{\delta,i}}=0, ~\beta_i(u)=0, ~\forall i=1,\ldots,k},
	\end{align}
	where $\beta_i(u)$ is the barycenter map defined by 
	\begin{equation}\label{barycenter map}
		\beta_i(u)=\frac{1}{|u^{\delta,i}|_2^2}\int_{\RN} \sbr{x-z_i}(u^{\delta,i}(x))^2 ~\mathrm{d}x.
	\end{equation}

	For any function $u\in H^1(\RN)$ with  compact support, we define
	\begin{equation} \label{defi of Fdelta}
		\begin{aligned}
			\FR_\delta(u):= &\frac{1}{2} \int_{\RN} \mbr{|\nabla u|^2 +\sbr{V(x)+1}u^2}~\mathrm{d}x+\int_{\RN} \delta (V(x)+1)u~\mathrm{d}x\\-&\frac{1}{2}\int_{\supp{u}}(\delta+u)^2\log (\delta+u)^2~\mathrm{d}x+\frac{1}{2} \sbr{\delta^2\log \delta^2}|\supp{u}|.
		\end{aligned}
	\end{equation}
	Then a direct calculation shows that
	\begin{equation} \label{22}
		\begin{aligned}
			\LR(u)&= \LR(u_\delta)+	\FR_\delta( u  ^\delta).\\
		\end{aligned}		
	\end{equation}
	
	In this  subsection, we always assume the assumption \eqref{H1}, so we omit this condition in the statements.
	
	\medbreak
	
	\begin{lemma} \label{lemma 2.2}
		For any  $u\in H^1(\RN)$ satisfying  $u^\delta\neq0$,    there exists a function $\xi:H^1(\RN) \to \sbr{0,+\infty}$  such that  $	\xi(u) $ is the unique critical point of the map $$t\mapsto \mathcal{L}(u_\delta+tu^\delta) ,$$ which is also a maximum.  Moreover, the function $u_{\xi}:=u_\delta+ 	\xi(u) u^\delta $ is the projection of $u$ onto the set $$S:=\lbr{ u\in H^1(\RN): \hbr{ \LR^\prime(u) ,u^{\delta }}  =0},$$
		that is  
		$\hbr{\mathcal{L}'(u_{\xi}),u^\delta} =0.$

	\end{lemma}  
	\begin{proof}
		Set $$\mathcal{R}_u(t):=\mathcal{L}(u_\delta+tu^\delta) =	\mathcal{L}(u_\delta)+\mathcal{F}_\delta(tu^\delta).$$ 
		Obviously, $	\mathcal{R}_u(t)$ is continuous and  differential with respect to the parameter $t$. Hence, 
		\begin{equation}\label{e7}
			\begin{aligned}
				\mathcal{R}_u'(t)&=\frac{d}{d t} \mathcal{F}_\delta(tu^\delta) \\
				&=t\int_{\RN} \mbr{|\nabla u^\delta|^2 +V(x)\sbr{u^\delta}^2} ~\mathrm{d}x +\int_{\RN} \delta V(x)u^\delta~\mathrm{d}x\\
				&-\int_{\supp{u^\delta}} (\delta+tu^\delta) u^\delta\log (\delta+{tu^\delta})^2~\mathrm{d}x.
			\end{aligned}
		\end{equation}
		Then   $\lim_{t\to +\infty} \mathcal{R}_u'(t)=-\infty $ and 
		\begin{equation}
			\mathcal{R}_u'(0)\geq \delta (V_0-\log \delta^2) \int_{\RN} u^\delta ~\mathrm{d}x \geq 0.
		\end{equation}
		Moreover,  $\mathcal{R}_u'''(t)<0$ for all $t\in \sbr{0,\infty}$. Therefore, $\mathcal{R}_u'(t)$ is a concave function and  $\mathcal{R}_u (t)$  has a unique maximum point for each $u \in H^1(\RN)$ satisfying  $u^\delta\neq0$, which can be defined by $\xi(u)$.
	\end{proof}

	\begin{lemma}\label{lemma 2.3}
		For   $\sbr{z_1,\ldots,z_k}\in \mathcal{E}_k$,  let  $u\in H^1(\RN)$  be a function emerging around the $k$-tuples $\sbr{z_1,\ldots,z_k}$, then for any $j\in \lbr{1,\ldots,k}$, there exists a function    $\xi_j:H^1(\RN) \to \sbr{0,+\infty}$  such that $\xi_j(u)$ is the unique critical point of the map 
		$$t\mapsto \mathcal{L}(u_\delta+\sum_{i\neq j}u^{\delta,i}+tu^{\delta,j}), $$
		which is also a maximum. 	Moreover, the function $ u_{\xi,j}:= u_\delta+ \xi_j(u) u^{\delta,j}+\sum_{i\neq j}u^{\delta,i} $ is  the projection of $u$ onto the set $$S_j:=\lbr{ u\in H^1(\RN): \hbr{ \LR^\prime(u) ,u^{\delta,j}}  =0},$$
		that is  
		$\hbr{\mathcal{L}'( u_{\xi,j}), u^{\delta,j} }=0.$
	\end{lemma}
	\begin{proof}
		Observe that
		\begin{equation}\label{e32}
			\LR(u_\delta+\sum_{i\neq j}u^{\delta,i}+tu^{\delta,j})=\mathcal{L}(u_\delta)+\sum_{i\neq j}\mathcal{F}_\delta(u^{\delta,i})+\mathcal{F}_\delta(tu^{\delta,j}).
		\end{equation}
		then the argument is  similar to that of Lemma \ref{lemma 2.2}.
	\end{proof}
	\begin{remark}\label{Remark 3.1}
		{\rm 	Let  $ u\in	\mathcal{S}_{z_1,\ldots,z_k} $, then for any   $j\in \lbr{1,\ldots,k}$,   $u $ is its own projection onto  the sets $ S_{j} $, which means that $\xi_j(u)=1$. Moreover, we have the following relations
			\begin{equation}
				\mathcal{L}(u)=\max_{t>0}  ~\LR \sbr{u_\delta+ t u^{\delta,j}+\sum_{i=1,i\neq j }^k  u^{\delta,i}}    , \quad \text{ and }\quad  \FR_\delta(  u^{\delta,j})= \max_{t>0} \FR_\delta \sbr{t u^{\delta,j} } .
			\end{equation}
			Furthermore, we can write
			\begin{equation} \label{e121}
				\mathcal{L}(u)=\max \lbr{\LR(u_\delta+\sum_{i=1}^k t_iu^{\delta,i}):~ t_1,\ldots, t_k>0  },
			\end{equation}
			and the maximum is attained if and only if $t_1=t_2=\ldots=t_k=1$.}
		
	\end{remark}
	
	\vskip 0.12in 
	\begin{lemma}\label{lemma 2.4}
		For  $ \sbr{z_1,\ldots,z_k}\in \mathcal{E}_k$, we have $	\mathcal{S}_{z_1,\ldots,z_k}\neq \emptyset$. 
		
	\end{lemma}
	\begin{proof}
		Let $\varphi \in C_0^\infty(B(0,R_\delta))$ be a positive, radially symmetric function satisfying $\varphi^\delta\neq 0$. It follows from Lemma \ref{lemma 2.3} that  there exists $\xi_j(u)\in \sbr{0,+\infty}$ corresponding to $\varphi(x-z_i)$ such that $$\sum_{i=1}^k\mbr{\varphi_\delta(x-z_i)+\xi_j(u)\varphi^\delta(x-z_i)}
		\in\mathcal{S}_{z_1,\ldots,z_k}. $$	
		The proof is completed.
	\end{proof}

	\vskip 0.12in 
	\begin{lemma}\label{lemma 2.5}
		For   $\sbr{z_1,\ldots,z_k}\in \mathcal{E}_k$ and $u\in \mathcal{S}_{z_1,\ldots,z_k}$, we have 
		\begin{equation}
			\mathcal{L}(u_\delta )>0,  \quad \text{and} \quad   \FR_\delta(u^{\delta,i})>0, \quad  \forall i\in \lbr{1,\ldots,k}.
		\end{equation}
		Hence,  $\LR(u)=\mathcal{L}(u_\delta)+\sum_{i=1}^k\mathcal{F}_\delta(u^{\delta,i})>0$.
	\end{lemma}
	\begin{proof}
		It follows from   $V(x)\geq 2$   and   \eqref{defi of delta}, we deduce from the fact    $u_\delta\leq \delta<1$  that   
		\begin{align}\label{e1}
			\LR(u_\delta)=\frac{1}{2} \int_{\RN} \mbr{|\nabla u_\delta|^2+ (V(x)+1) u_\delta ^2}~\mathrm{d}x-\frac{1}{2} \int_{\RN}  u_\delta ^2\log  u_\delta ^2~\mathrm{d}x>0.
		\end{align}
		Moreover, since $u\in \mathcal{S}_{z_1,\ldots,z_k}$,  we can see from Remark \ref{Remark 3.1}  that
		\begin{equation}
			\FR_\delta(u^{\delta,i})=\max_{t>0} \FR_\delta(tu^{\delta,i})>\FR_\delta(0)=0.
		\end{equation}
		The proof is completed.
	\end{proof}

	\subsection{ Minimization problems on localized Nehari constraint}\label{Sect3.2}
	For all $k\geq 1$ and $\sbr{z_1,\ldots,z_k}\in \ER_k$, we consider the minimization problem
	\begin{equation}
		\alpha\sbr{z_1,\ldots,z_k}:=	\inf_{\SR_{z_1,\ldots,z_k}}\mathcal{L}(  u).
	\end{equation}

	\begin{lemma}\label{lemma 3.1}
		Assume that \eqref{H1} holds. Then  for		 all $k\in \N\setminus \lbr{0}$   and $\sbr{z_1,\ldots,z_k}\in \ER_k$,  
		\begin{equation}\label{der2}
			\exists \tilde{u}_{k} \in \mathcal{S}_{z_1,\ldots,z_k} , ~~ \text{ such that }~~\LR(\tilde{u}_{k} )=\al \sbr{z_1,\ldots,z_k},
		\end{equation}
		which means that 	the infimum $\alpha\sbr{z_1,\ldots,z_k}$ is attained, and   the set 	\begin{equation}\label{fw2}
			\MR_{ z_1,\ldots,z_k} :=\lbr{u\in \SR_{z_1,\ldots,z_k} : \LR(u)=\al \sbr{z_1,\ldots,z_k}},
		\end{equation}is nonempty. Moreover,  $\tilde u_k \in D(\LR)$.		
	\end{lemma}
	
	\begin{proof}
		Let  $k\in \N\setminus \lbr{0}$ and  $\sbr{z_1,\ldots,z_k}\in \ER_k$ be fixed, and let  $\lbr{u_{ n}}_{n } \subset \SR_{z_1,\ldots,z_k}$   be  a     minimizing  sequence for level $\alpha\sbr{z_1,\ldots,z_k}$. Then by   Lemma \ref{lemma 2.10}, we know that  the sequence $\lbr{  u_{ n}}_n$ is   bounded in $H^1(\RN)$. Thus, passing to   subsequence, there exists $ {u }\in H^1(\RN)$ such that as $n\to +\infty$,
		\begin{equation} \label{e2}
			\begin{aligned}
				& {u}_{ n} \rightharpoonup  {u} \quad \text{ weakly  in } H^1(\RN), \\
				& {u}_{ n}  \to  {u} \quad \text{ strongly in }  L^p_{\loc}(\RN),   ~~ p\geq 2,\\
				&  {u}_{ n} \to  {u}   \quad \text{ a.e. in }   \RN.
			\end{aligned}
		\end{equation} 
		It follows that $ {u}_{ n} \to  {u} $ in $ L^p(B(z_i,R_\delta))$, and $0\leq    {u }_\delta \leq \delta $ in $\RN \setminus \bigcup_{i=1}^kB(z_i,R_\delta)$ and $\beta_i( {u })=0$ by the continuity of $\beta_i$.  
		By Lemma \ref{lemma 2.3}, we consider 
		\begin{equation}
			\tilde u =   {u }_\delta+\sum_{i=1}^k \xi_i(  u ) {u} ^{\delta,i} \in\SR_{z_1,\ldots,z_k}\quad \text{ and } \quad \tilde u_{ n}=  \sbr{u_{ n} }_\delta+\sum_{i=1}^k \xi_i(  u ) {u}_{ n}^{\delta,i}.
		\end{equation}
		Integrating  \eqref{r6}  with $\varepsilon= \sbr{2^*-2}/2$ and applying Cauchy-Schwarz inequality,  we deduce from  \eqref{e2} that  as $n\to +\infty$,
		\begin{equation}
			\abs{ \int_{\RN} F_2( 	\tilde{u}_{ n} ) \dx-\int_{\RN} F_2(\tilde{u}  )\dx}  \leq C  |\tilde{u}_{ n}-\tilde{u}  |_2 \sbr{\nm{\tilde{u}_{ n}}_H+ \nm{ \tilde{u}  }_H}^{2^*/2} \to 0.
		\end{equation}                             
		Then from    the weak lower semicontinuity of both the norm and the functional $\Psi$,  we  deduce that                                                                                                     	\begin{equation} 
			\begin{aligned}
				&\LR(\tilde{u} ) = \frac{1}{2} \int_{\RN} \mbr{|\nabla \tilde{u} |^2+ (V(x)+1)\tilde{u} ^2}~\mathrm{d}x- \int_{\RN} F_2( 	\tilde{u} ) \dx+ \Psi ( \tilde{u}) \\
				&\leq \liminf_{n\to+\infty} \mbr{\frac{1}{2} \int_{\RN} \mbr{|\nabla \tilde{u}_{ n}|^2+ (V(x)+1)\tilde{u}_{ n}^2}~\mathrm{d}x- \int_{\RN} F_2( 	\tilde{u}_{ n} ) \dx+ \Psi ( \tilde{u}_{ n})     }
				\\&=\liminf_{n\to+\infty} \LR( \tilde{u}_{ n}) . 		\end{aligned}
		\end{equation}                                               	Moreover, from   Remark \ref{Remark 3.1}, since $ {u_{ n}}\in \SR_{z_1,\ldots,z_k}$, we observe that for all $n\in\N$
		\begin{equation}
			\begin{aligned}		
				\LR( \tilde{u}_{ n}) = \LR  \sbr{\sbr{u_{ n} }_\delta+\sum_{i=1}^k \xi_i(  u) {u}_{ n}^{\delta,i} }   \leq \max_{t_1,\ldots,t_k>0}\LR  \sbr{\sbr{u_{ n} }_\delta +\sum_{i=1}^k t_i   u_{ n}^{\delta,i}} 
				=     \LR( {u}_{ n}).
			\end{aligned}
		\end{equation}	  	        
		Therefore, since  we already know that $\tilde{u}_k\in\SR_{z_1,\ldots,z_k}$, the following relation holds                                                                                                                                            
		\begin{equation} 
			\begin{aligned}
				\alpha(z_1,\ldots,z_k) \leq \LR(\tilde{u} ) 
				=\liminf_{n\to+\infty} \LR( \tilde{u}_{ n})  \leq \lim_{n\to+\infty} \LR(u_{ n})=\alpha(z_1,\ldots,z_k),
			\end{aligned}
		\end{equation}                                            which implies that  $	\alpha(z_1,\ldots,z_k) =\LR(\tilde{u} ) $. Finally,  by a similar argument as used in Lemma \ref{lemma 2.10}, we know  that   $\tilde u  \in D(\LR)$.  This completes the proof.         		
	\end{proof}

	\vskip0.15in

	In what follows, we describe some important properties of the submerged part and the emerging part of  $\tilde u_k $.   Before proceeding, we present a new lemma concerning the Gaussian-type decay of solutions satisfying certain variational inequalities, which    seems to be  new in the study of logarithmic scalar field equations.

	\begin{lemma} \label{lemma newdecay}
		Let $\DR\subset \RN$ be a closed and convex set,  and  the constants $l\in \R^+$, $s\in \sbr{0,1}$. Suppose that    $u$ is  a positive solution of 
		\begin{equation}\label{fu1}
			\begin{cases}
				-\Delta u +lu\leq u\log u^2 \quad &\text{ in } ~~\RN\setminus \DR,\\
				u\leq s \quad &\text{ on }~~ \RN\setminus \DR.\\
			\end{cases}
		\end{equation}
		Then,  for all $b\in \sbr{0,\sqrt{l}}$, there exist  two positive constants $a=a(l,b,N,s)\in \sbr{0,b}$  and  $C_a=C_a(l,b,N,s)$ such that 
		\begin{equation}
			u(x) \leq C_a  e^{-a\sbr{\mathrm{dist}(x,\DR)}^2 }, \quad \forall x\in \RN\setminus \DR.
		\end{equation} 
	\end{lemma}

	\begin{proof}
		Let $\rho(x):=\mathrm{dist}(x,\DR)$  and $ \DR^c :=\RN\setminus \DR$.  Since $\DR\subset \RN$ is a closed and convex set,  we know that the function $\rho^2(x) $ is   convex; hence, $ \Delta \rho^2(x)\geq 0$  and $|\nabla \rho (x)|^2=1 $ for $ x\in \DR^c$.
		Observe that $u$ satisfies \begin{equation}\label{yu2}
			\begin{cases}
				-\Delta u +lu\leq 0 \quad &\text{ in } ~~\RN\setminus \DR,\\
				u\leq s \quad &\text{ on }~~ \RN\setminus \DR,\\
			\end{cases}
		\end{equation} because $s\leq 1$,
		then  it follows from  \cite[Lemma 3.3]{cerami_CPMA_2013} that  for all $b\in \sbr{0,\sqrt{l}}$, there exists a positive constant $C_b=C_b(l,b,N)$ such that 
		\begin{equation}
			u(x) \leq C_b s e^{-b  \mathrm{dist}(x,\DR) 
			}, \quad \forall x\in \DR^c,
		\end{equation} 
		and we may choose a fixed positive constant $R_1=R_1(l,b,N,s)>1$   such that 
		\begin{equation}
			u(x)  \leq    C_b   e^{-b  R_1 } 
			\leq  e^{- \frac{b  R_1}{2}  }  \leq e^{-1}, \quad \forall x\in   \DR_{R_1}^c:=   \DR^c \cap \lbr{   \rho(x)\geq R_1 }. 
		\end{equation}			 
		Define the function
		\begin{equation}
			w(x)=M e^{-c(\rho(x))^2}, ~~\text{ with }~~ M=e^{-\frac{b  R_1}{2}+cR_1^2},
		\end{equation}
		where $ 0<c<1/2$ will be determined later.   Setting $c=mb$,   a direct calculation shows that 
		\begin{equation} \label{yu1}
			\begin{aligned}
				Qw:&=	\Delta w-lw+w\log w^2\\&= w  \mbr{ (4c^2-2c)\rho^2 -c \Delta \rho^2-l+2\log M}  \\
				&\leq   w  \mbr{ (4c^2-2c)R_1^2 +2 \sbr{-\frac{b  R_1}{2}+cR_1^2}} 
				\\& \leq w  \sbr{  4m^2 b^2R_1^2 - bR_1  },  \quad \forall x\in   \DR_{R_1}^c .
			\end{aligned}					
		\end{equation}
		By choosing $m=m(b,l,N,s)>0$ small (hence $c\in \sbr{0,b}$) such that  $4m^2 b^2R_1^2 - bR_1 \leq 0$, then we have 
		\begin{equation}
			Qw \leq 0 \leq Qu, \quad \forall x\in   \DR_{R_1}^c.
		\end{equation}
		Observe that $ w(x)= e^{-\frac{b  R_1}{2}} \geq u(x)$ on $\partial \DR_{R_1}^c$,  and $ u(x),w(x) \in [0,e^{-1}]$ for  $ x\in \DR_{R_1}^c$. Therefore,   we deduce from the maximum principle (see, for instance, \cite[Theorem 10.1]{GT1983}) that 
		\begin{equation}
			u(x) \leq w(x) = M e^{-c(\rho(x))^2}, \quad  \forall ~x\in \DR^c \cap \lbr{   \rho(x)\geq R_1 } .
		\end{equation}
		because  the function $f(t)=-lt+t\log t^2$ is strictly decreasing for  $t\in  [0,e^{-1}]$. Moreover,  there holds
		\begin{equation}
			u(x)\leq C_b s e^{-b  \rho(x) 
			} \leq C_b s e^{-\frac{b}{R_1}  (\rho(x))^2},  \quad  \forall ~x\in \DR^c \cap \lbr{   \rho(x)\leq R_1 }.
		\end{equation}
		Hence,  there exists $a=\min\lbr{ c,\frac{b}{R_1}  } \in \sbr{0,b}$, and a constant $C_a=\max\lbr{M,C_b s }$ such that 
		\begin{equation}
			u(x) \leq C_a e^{-a(\rho(x))^2}, \quad \forall x\in  \RN\setminus \DR.
		\end{equation}
		This completes the proof.					 
	\end{proof}

	\vskip0.15in
	\begin{proposition}\label{prop 3.1}
		Assume that \eqref{H1} and \eqref{H2} hold.  For all $k\geq 1$ and $\sbr{z_1,\ldots,z_k}\in \ER_k$,  let  $\tilde{u}_k$ be  the function  given  in Lemma \ref{lemma 3.1},  then the submerged part $\sbr{\tilde{u}_k}_\delta$ is a solution of 
		\begin{equation}\label{e3}
			\begin{cases}
				-\Delta u +V(x)u=u\log u^2 \quad &\text{ in } ~~\RN\setminus \supp{\tilde{u}_k^\delta},\\
				u=\delta \quad &\text{ on }~~ \supp{\tilde{u}_k^\delta},\\
				u>0 \quad &\text{ in } ~~\RN.
			\end{cases}
		\end{equation}
		Set $ B=  \sbr{\cup_{i=1}^k B(z_i,R_\delta)}\cup  B(0,\widetilde{R})   $ and $d(x):=\mathrm{dist}\sbr{x,B}$, where $ B(0,\widetilde{R}) $  is the smallest ball containing  the set $\supp{\vartheta(x)^- }$. Then 
			there exist   positive  constants $ \zeta \in  \sbr{0, \bar \eta} $ and   $\widetilde{C}=\widetilde{C}\sbr{ \bar \eta, V_{\infty}, N,\delta}$ such that 
			\begin{equation}\label{e4}
				0<\sbr{\tilde{u}_k}_\delta(x)<\widetilde{C}   e^{- \zeta \sbr{d(x)}^2} , \quad  \forall x\in \RN\setminus B,
			\end{equation}
			where $\bar \eta   $ and $\vartheta(x)$ are given by \eqref{H3} and \eqref{defi of vartheta}, respectively.
	\end{proposition}
	\begin{proof}
		In view of  \eqref{e1}, the functional $\LR$
		is coercive and convex on the set
		$$ \mathcal{B}= \lbr{u\in H^1(\RN)  : ~ |u(x)|\leq \delta ~~ \forall x\in \RN, ~ \text{ and  } ~ u=\delta \text{ on } \supp{\tilde{u}_k^\delta}},$$ 
		and $\sbr{\tilde{u}_k}_\delta \in \BR$ is the unique and positive minimizer for the minimization problem $$\min\lbr{ \LR(u): u\in \mathcal{B}}.$$   
		Now we prove that $\sbr{\tilde{u}_k}_\delta $ is a critical point of $\LR$ and solves \eqref{e3}. Indeed, by the convexity of $\Psi$, it follows that for small $t>0$ and $v \in \BR$, 
		\begin{equation}
			\begin{aligned}
				0&\leq \LR((1-t)\sbr{\tilde{u}_k}_\delta+tv) -\LR(\sbr{\tilde{u}_k}_\delta)\\&=\Phi((1-t)\sbr{\tilde{u}_k}_\delta+tv) -\Phi((\sbr{\tilde{u}_k}_\delta)) +\Psi ((1-t)\sbr{\tilde{u}_k}_\delta+tv) -\Psi(\sbr{\tilde{u}_k}_\delta)
				\\&\leq \Phi((1-t)\sbr{\tilde{u}_k}_\delta+tv) -\Phi(\sbr{\tilde{u}_k}_\delta) +t( \Psi (v)- \Psi(\sbr{\tilde{u}_k}_\delta)).
			\end{aligned}
		\end{equation}
		Dividing by $t$ and letting  $t\to 0$, we obtain that  \eqref{dq1} holds. By Lemma  \ref{lemma 3.1}, we know that $ \sbr{\tilde{u}_k}_\delta \in D(\LR)$,   we deduce from  Lemmas \ref{lemma 2.23} and \ref{lemma 2.7} that $\sbr{\tilde{u}_k}_\delta $ is a critical point of $\LR$, and   solves $	-\Delta u +V(x)u=u\log u^2  $   in $\lbr{x\in\RN: |u(x)|<\delta}$ and $u=\delta  $ on $ \supp{\tilde{u}_k^\delta}$.  Since $u=\delta$ is a strict supersolution of \eqref{e3}, then $$\lbr{x\in\RN: |u(x)|<\delta}=\RN\setminus \supp{\tilde{u}_k^\delta} .$$  Furthermore, 
		by the maximum principle in \cite[Theorem 1]{vazquez=AMO=1984}  and  the argument in \cite[Section 3.1]{avenia-2014} word by word,
		we know that  $\sbr{\tilde{u}_k}_\delta$ is a classical solution and $\sbr{\tilde{u}_k}_\delta>0$ in $\RN\setminus \supp{\tilde{u}_k^\delta}$. Therefore, $\sbr{\tilde{u}_k}_\delta>0$ in $\RN$ since $\sbr{\tilde{u}_k}_\delta=\delta>0$  on $ \supp{\tilde{u}_k^\delta}$. Hence, the submerged part  $\sbr{\tilde{u}_k}_\delta$ solves   \eqref{e3}.

			In what follows, we prove the  relation  \eqref{e4}. Indeed,  since $\sbr{\tilde{u}_k}_\delta$ is a solution of \eqref{e3} and $\supp{\tilde{u}_k^\delta}\cup  \supp{\vartheta^- }\subset B $,  then  the function $\sbr{\tilde{u}_k}_\delta$  satisfies 
			\begin{equation}
				-\Delta u + V_\infty u \leq u\log u^2, \quad \forall x\in \RN\setminus B.
			\end{equation}
			Moreover,  we know that $\sbr{\tilde{u}_k}_\delta \leq \delta<1 $   in $\RN$. Then it follows from Lemma \ref{lemma newdecay}  by choosing $s=\delta$,  $ l=V_\infty $, and $b=  \bar \eta \in \sbr{0,\sqrt{V_\infty }}$, there exists 	$  \zeta \in \sbr{0, \bar \eta}$ and $\widetilde{C}=\widetilde{C}\sbr{ \bar \eta, V_{\infty}, N,\delta}$	such that  for any $  x\in \RN\setminus B$, 
			\begin{equation}
				\begin{aligned}
					u(x) &\leq \min \lbr{ \widetilde{C}  e^{- \zeta\sbr{\mathrm{dist}(x,   B(z_i,R_\delta) )}^2 }, \widetilde{C}  e^{- \zeta\sbr{\mathrm{dist}(x,   B(0,\widetilde{R}) )}^2 }: i\in \lbr{1,\ldots,k}} \\
					&\leq   \widetilde{C}  e^{-\zeta \sbr{d(x)}^2  },
				\end{aligned}
			\end{equation}  
			because $ d(x) = \min \lbr{ \mathrm{dist} \sbr{x,   B(z_i,R_\delta)}, \mathrm{dist}(x,   B(0,\widetilde{R}) ) : i\in \lbr{1,\ldots,k}}$.
			The proof is completed.
	\end{proof}
	
	\vskip0.15in
	
	\begin{proposition} \label{lemma 3.2}
		Assume that \eqref{H1}   holds.
		For all $k\geq 1$ and $\sbr{z_1,\ldots,z_k}\in \ER_k$,  let $\tilde{u}_k$ be the function given in  Lemma \ref{lemma 3.1},  then there exist Lagrange multipliers $\lambda_1,\ldots,\lambda_k$ in $\RN$ such that  for  all $ i\in \lbr{1,\ldots,k}$ there holds
		\begin{equation} \label{e5}
			\hbr{\LR'(\tilde{u}_k), \phi } =\int_{ B\sbr{z_i,R_\delta}} \tilde{u}_k^\delta(x)  \phi(x)  \mbr{\lambda_i\cdot\sbr{x-z_i}} ~\mathrm{d}x, \quad \forall~ \phi \in H_0^1( B\sbr{z_i,R_\delta})   .
		\end{equation}
	\end{proposition}
	
		%
	
	\begin{proof}
		Since for any   $ i\in \lbr{1,\ldots,k}$, the functional $\LR$ is of class $C^1$ in $H ^1( B\sbr{z_i,R_\delta} )$, the  relation \eqref{e5} implies that   $\tilde{u}_{k}$ is a critical point for $\LR$ constrained on the set\begin{equation}
			\begin{aligned}
				\Pi_i:=&\lbr{  u\in H ^1(B\sbr{z_i,R_\delta}):u=  \tilde{u}_{k} \text{ on } \partial B\sbr{z_i,R_\delta},  \right.  \\  &\left.\quad \quad \quad u^\delta\neq 0 \text{ in } B\sbr{z_i,R_\delta},   \text{ and }~ \beta_i(u)=0}.
			\end{aligned}
		\end{equation}
		In what follows, for simplicity, we denote  $\tilde{u}_{k}$ by $\tilde{u}$.
		To prove \eqref{e5}, we proceed by contradiction that there exists $i \in \lbr{1,\ldots,k}$ such that    for all $\lambda_{i}\in \RN$, there exists $\phi_{i} \in H_0^1( B\sbr{z_i,R_\delta})$   satisfying
		\begin{equation}
			\hbr{	\LR'(\tilde{u} ), \phi_{i}}   \neq \int_{ B\sbr{z_i,R_\delta}} {\tilde{u}^\delta(x) }  \phi_{i}(x)  \mbr{\lambda_i\cdot\sbr{x-z_i}} ~\mathrm{d}x  .
		\end{equation}
		Then there exists  $ B( \tilde{u}, R) \subset \Pi_i$   such that for all $u\in  B( \tilde{u} , R) $, the tangential component $ -\nabla \LR  {|_{\Pi_i}}(u) \neq 0$, and ${u}^\delta\neq 0$ in $B( z_i , R)  $. Consider
		\begin{equation}
			\begin{cases}
				\frac{d }{ d  \tau} \eta(\tau,u)=-\nabla \LR {|_{\Pi_i}} ( \eta(\tau,u)) \tilde{h}( \eta(\tau,u))\\
				\eta(0,u)=u,
			\end{cases}
		\end{equation}
		where $ \tilde{h} (u)$ is the cut-off  smooth  positive function equal to $1$ in a suitable small ball $B(\tilde{u},\hat{r})$ with $\hat r < R/2$, and equal to $0$ outside of $B(\tilde{u},R /2) $. Then there exists a continuous map $ \eta  :[0,\bar \tau]\times \Pi_i\to \Pi_i$  such that
		\begin{equation}\label{ep1}
			\begin{aligned}
				(i)& \quad \eta(0,u)=u, \quad\quad \quad\quad\forall u \in  \Pi_i.\\
				(ii)	& \quad \eta(\tau,u)=u, \quad\quad\quad\quad \forall \tau \in [0,\bar \tau], ~ ~  u \in  \Pi_i\setminus B(\tilde{u},R/2).\\
				(iii)& \quad  \LR(\eta(\tau,u) ) < \LR(  u ),~ \quad \forall \tau \in (0,\bar \tau], ~ ~u \in   B(\tilde{u},R/2).
			\end{aligned}
		\end{equation}
		Let
		\begin{equation}\small
			t^-:=\inf \lbr{ t<0: \tilde{u}+t{\tilde{u} }^{\delta,i} \in  B( \tilde{u} , R/2)},   
		\end{equation}
		and
		\begin{equation}
			t^+:=\sup \lbr{ t>0: \tilde{u}+t {\tilde{u} }^{\delta,i} \in  B( \tilde{u} , R/2)},
		\end{equation}
		and remark that $t^->-1$. Consider a continuous curve in $\Pi_i$ defined by $\gamma(t)=\eta( \bar \tau, \tilde{u}+t{\tilde{u} }^{\delta,i}) $ with $[t^-, t^+] $. Then from \eqref{ep1} (ii), we deduce
		\begin{equation}
			\gamma(t)=\eta( \bar \tau, \tilde{u}+t {\tilde{u}}^{\delta,i}) =\tilde{u}+t {\tilde{u} }^{\delta,i} \quad \text{ for } t=t^- \text{ or }t^+,
		\end{equation}
		from which, we get
		\begin{equation}
			\hbr{\LR^\prime(\gamma(t) ), \sbr{\gamma(t) }^{\delta,i}} \begin{cases}
				>0, \quad \text{ if }t=t^-,\\
				<0, \quad \text{ if }t=t^+,
			\end{cases}
		\end{equation}
		because $0$ is the unique  critical point (a maximum point)  of $\LR (\tilde{u}+t  {\tilde{u}  }^{\delta,i}) $. Therefore, we can find $\hat{t} \in \sbr{ t^-, t^+}$ such that $ \hbr{\LR^\prime(\gamma(\hat{t}) ), \sbr{\gamma(\hat{t}) }^{\delta,i}}=0$, and $ \tilde{u}+\hat t  {\tilde{u} }^{\delta,i} \in B(\tilde{u},R/2) $. Moreover, by construction, we obtain
		\begin{equation}
			\sbr{\gamma(\hat{t}) }^{\delta,i}\neq 0, \quad \gamma(\hat{t}) = \tilde u ~ \text{ on }~  \partial	 B(z_i,R_\delta) , \quad \beta_i(\gamma(\hat{t}) )=0.
		\end{equation}
		then $\gamma(\hat{t}) \in \SR_{  z_1 ,\ldots,  z_{k} } $. Hence, recall that  $ \tilde u  \in \SR_{  z_1 ,\ldots,  z_{k} } $, we  deduce from \eqref{ep1} (iii) and Remark \ref{Remark 3.1} that 
		\begin{equation}
			\begin{aligned}
				\LR(\gamma(\hat{t}) ) =\LR(  \eta( \bar \tau, \tilde{u}+\hat t  {\tilde{u}  }^{\delta,i}) )< \LR( \tilde{u}+\hat t   {\tilde{u}  }^{\delta,i} ) \leq \LR( \tilde u),
			\end{aligned}
		\end{equation}
		which contradicts to the minimality of $\tilde u$. The proof is completed.
	\end{proof}

	\subsection{Characterization of the ground state energy  of  the limit system}\label{Sect characterization}
	
	In this subsection, we characterize the ground state energy $	\CR^\infty$ of the limit system \eqref{p infty}, which    will be frequently used hereafter.
	To begin, we observe that when    $V(x)=V_\infty$, and we consider the functional $\LRI$, all results obtained for  $\LR$ in Section \ref{Sect3.1} and Subsection \ref{Sect3.2} can be directly adapted.
	Hence,  for any $k \in \N\setminus \lbr{0}$ and $k$-tuple $\sbr{ z_1,\ldots,z_k } \in \ER_{k}$,   we redefine the notations $ \SR_{z_1,\ldots,z_k}$,   $\xi(  u)$ and $\xi_j(  u)$   as  $ \SR^\infty_{z_1,\ldots,z_k}$,   $\xi^\infty(  u)$ and  $\xi_j^\infty(  u)$, respectively.  Moreover, we define   \begin{equation}
		\alpha^\infty\sbr{z_1,\ldots,z_k}:=	\inf_{\SR^\infty_{z_1,\ldots,z_k}}\mathcal{L}^\infty(  u).
	\end{equation}
	In this setting, the infimum is, in fact, a positive minimum, and the set 
	\begin{equation}
		\MR_{z_1,\ldots,z_k}^\infty=\lbr{u\in \SR^\infty_{z_1,\ldots,z_k}: \LR^\infty(u)=\al^\infty\sbr{z_1,\ldots,z_k}}.
	\end{equation}
	is nonempty.
	
	\begin{lemma} \label{lemma 3.3}
		The  relation	
		\begin{equation} \label{m1}
			\CR^\infty =\sup_{x\in \RN} \al^\infty(x)=	 	\sup_{x\in \RN} \inf_{\SR_x^{\infty} } \LRI	 
		\end{equation} holds true, and for all $x\in \RN$,
		\begin{equation} \label{m2}
			\begin{aligned}
				&	\MR_{x}^{\infty}:= \lbr{ u\in \SR_x^{\infty}:  \LRI(u)=\CR^\infty  }=\lbr{U(\cdot-x)  }.
			\end{aligned}					
		\end{equation}
		
	\end{lemma}

	\begin{proof}
		First of all  we  point out that   by the invariance of $\RN$  and $\LRI$ under the action of the translation group,  there holds $	\sup_{x\in \RN} \inf_{\SR_x^{\infty } } \LRI	=\inf_{\SR_y^{\infty } } \LRI $  for all $y\in \RN$. Hence,  to prove the relation \eqref{m1}, we only need to show that $ 	 \CR^\infty= \alpha^{\infty}(0):= \inf_{\SR_0^{\infty} } \LRI$.
		\medbreak

		For   any $ u\in \MR_{0}^{\infty } \subset \SR_{0}^{\infty }$,     in view of Lemma \ref{lemma 2.2} and Remark \ref{Remark 3.1}, we know that $ 1$ is the only maximum point in $\sbr{0,+\infty}$ of the function $\mathcal{R}_u^{\infty } (t):=\mathcal{L}^\infty(u_\delta+t u ^\delta) $. Since   $\lim_{t\to +\infty}\mathcal{R}_u^{\infty } (t) =-\infty $, we can find $\tilde  t$ large enough such that  $ \mathcal{R}_u^{\infty } (\tilde t)<0  $.
		Set
		\begin{equation}
			h(t)=
			\begin{cases}
				2t u_\delta,  &\quad t\in [0,\frac{1}{2}), \\
				u_\delta+(2t-1) \tilde t  {u}^\delta, &\quad  t\in [ \frac{1}{2},1],
			\end{cases}
		\end{equation}
		then $h\in\Gamma$ and  the function  $f(t):= \LR^\infty(h(t))$ is increasing on $[0,\frac{1}{2})$. Therefore,   there exists   $t_0\in [ \frac{1}{2},1]$ such that
		\begin{equation}
			\LR^\infty(h(t_0))= \max_{t>0} \LR^\infty(h(t)).
		\end{equation}
		Then we deduce 	from  Lemma \ref{lemma pinfty} that
		\begin{equation}\label{r7}
			\begin{aligned}
				\LR^\infty(u) &= \max_{t>0} \mathcal{L}^\infty(u_\delta+t {u }^\delta)  \geq \LR^\infty(h(t_0))= \max_{t>0} \LR^\infty(h(t)) \\&\geq \inf_{h\in\Ga} \max_{t>0} \LR^\infty(h(t))= \CR^\infty.
			\end{aligned}
		\end{equation}
		Hence, we obtain
		\begin{equation}\label{r111}
			\CR^\infty  \leq \LRI(u)= \alpha^{\infty}(0)  \leq  \LR^\infty(U ) =\CR^\infty,
		\end{equation}
		since $ U \in \SR_{0}^{\infty } $ by the symmetry of $U $. Therefore, we have established that $\CR^\infty = \alpha^{\infty }(0)$, which is the desired result.
		Finally,  \eqref{m2}   follows directly from  the fact that the Gausson $U  $ is the unique (up to translation) positive ground state solution  of 	  \eqref{p infty}.
	\end{proof}

	\subsection{Maximization problem on the set of minimizers}\label{Sect3.3}
	In this subsection,  for any 	 integer $k\in \N\setminus\lbr{0}$, we show  that if the $k$-tuples     $ \sbr{z_1,\ldots,z_k} $    vary in the configuration space $ \ER_{k} $, then  the supremum of the set of minima $ \alpha\sbr{z_1,\ldots,z_k} $ is  indeed  a maximum. Define
	\begin{equation}\label{defi of Lam}
		\varLambda_{k}:=  \sup_{\ER_ {k}}	~\alpha\sbr{z_1,\ldots,z_k}.
	\end{equation}
	Throughout this subsection we always work under the assumptions \eqref{H1} and \eqref{H2}, so we omit these conditions in the statements. Our main result in this subsection can be stated as follows.
	
	\begin{proposition} \label{prop fixed point}
		Let  $k\in \N\setminus\lbr{0}$ be fixed,   then
		there exists $\sbr{\tilde z_1,\ldots,\tilde z_k} \in \ER_{k,m}$ such that \begin{equation}\label{o3}
			\varLambda_{k}= \alpha\sbr{\tilde z_1,\ldots,\tilde z_k}.
		\end{equation}
	\end{proposition}
	\vskip 0.15in
	\noindent The key step in proving Proposition \ref{prop fixed point} is the following priori estimates
	\begin{equation}
		\varLambda_{1}>	\CR^\infty,  \quad \text{ and }\quad \varLambda_{k}>\varLambda_{k-1}+\CR^\infty,  ~~ \text{ if } k\geq 2,
	\end{equation}
	which  play a crucial role in preventing the positive bumps from escaping to infinity. The proofs will be carried out by mathematical induction on the number $k$.
	\vskip 0.15in
	\begin{proposition} \label{prop 3.2}
		We have
		\begin{equation}\label{e11}
			\varLambda_1>\CR^\infty.
		\end{equation} 
		Moreover, there exists $\tilde z \in\RN$ such that $\varLambda_1= \alpha(\tilde z).$ 
		
	\end{proposition}
	
	\begin{proof}
		
		First, we prove that \eqref{e11} holds.	For that purpose, we  	take $\lbr{z_n}_n$ being a   sequence such that
		\begin{equation}\label{defi of zn}
			z_n=\rho_n  e, \quad  \text{ with }\rho_n\in\R, ~~\rho_n\to +\infty, ~~ \text{ and } e\in \RN, ~~|e|=1.
		\end{equation}
		By Lemma \ref{lemma 3.1},  there exists   $u_n\in \SR_{z_n}$ such that $\al(z_n)=\LR(u_n)$.   By a   similar argument as in that of  Lemma \ref{lemma 2.2}, we can prove that     $u_n ^\infty=(u_n)_\delta+\xi^\infty(u_n) (u_n)^\delta \in \SR_{z_n}^\infty$. Obviously,  for $x\in \partial B(0,\rho_n/2)$ and sufficiently large  $n$, there holds
			\begin{equation}
				\mathrm{dist}\sbr{x,B(z_n,R_\delta) \cup  B(0,\widetilde{R}) }  \geq \frac{\rho_n}{4}.
		\end{equation} 
		By Lemma \ref{lemma 3.3}, we have 
		\begin{equation}\label{e6}
			\begin{aligned}
				\varLambda_1\geq \LR(u_n)&=\max_{t>0}\LR((u_n)_\delta+t (u_n)^\delta)\geq \LR(u_n ^\infty)\\&=\LR^\infty(u_n ^\infty)+\frac{1}{2} \int_{\RN} \vartheta(x) \sbr{u_n ^\infty(x)}^2 ~\mathrm{d} x\\
				&\geq \CR^\infty+\frac{1}{2} \int_{\RN} \vartheta(x) \sbr{u_n ^\infty (x)}^2 ~\mathrm{d} x.
			\end{aligned}
		\end{equation}
		By   \eqref{H2}, we know that $\supp{\vartheta^-}$ is bounded. Then  for $n$ large enough, we have  $\supp{ \vartheta^-} \subset B(0,\widetilde{R})\subset B(0,\rho_n/2)$. Therefore, 
		\begin{equation}
			\int_{\RN} \vartheta(x) \sbr{u_n ^\infty(x)}^2 ~\mathrm{d} x \geq  \int_{ B\sbr{z_n,R_\delta}} \vartheta(x) \sbr{u_n ^\infty(x)}^2 \dx- \sup_{B(0,\rho_n/2)} \sbr{u_n ^\infty }^2~ \int_{\supp{ \vartheta^-} }|\vartheta(x)| \dx.
		\end{equation}
	 	By the maximum principle and Proposition \ref{prop 3.1}, we have 
			\begin{equation}
				\sup_{B(0,\rho_n/2)} \sbr{u_n ^\infty }^2 \leq 	\sup_{\partial B(0,\rho_n/2)} \sbr{u_n ^\infty }^2 =\sup_{\partial B(0,\rho_n/2)} \sbr{u_n }_\delta^2  \leq C e^{-\bar \zeta \rho_n^2},
			\end{equation}
			where $\bar \zeta=  {\zeta}/{8} \in \sbr{ 0,\bar{\eta}}$, and $\zeta$ is given in Proposition \ref{prop 3.1}, and  $C$ is independent of $n$.    
		Using  \eqref{H2} again,   we know that $\int_{ B\sbr{z_n,R_\delta}} \vartheta(x) \sbr{u_n ^\infty(x)}^2 \dx$ decays slower than $e^{-\bar \zeta \rho_n^2}$ for $n$ large enough. Hence, for  $n$ large enough, we have 	\begin{equation} \label{e16}
			\int_{\RN} \vartheta(x) \sbr{u_n ^\infty(x)}^2 ~\mathrm{d} x  >0.
		\end{equation}
		Finally,	it follows from \eqref{e6} that $\varLambda_1>\CR^\infty$. 	
		
		
		\medbreak
		
		Next, we will prove that  there exists $\tilde z \in\RN$ such that $$\varLambda_1= \alpha(\tilde z).$$  By the definition of $\varLambda_1$, we can choose a sequence of points $\lbr{ \tilde z_n}_n$, $ \tilde z_n \in\RN$ such that $\lim_{n\to+\infty} \alpha( \tilde z_n)=\varLambda_1$. Then we claim that $\lbr{ \tilde z_n}_n$ is bounded in $\RN$. Otherwise, up to subsequence,  we assume that $|\tilde z_n|\to + \infty$ as $n \to +\infty$. Let   $U_{\tilde z_n}:=U(x- \tilde z_n)$, then  by Lemma \ref{lemma 2.2},    we know that there exists $\xi(U_{\tilde z_n}) \in \sbr{0,+\infty}$ such that $$\tilde v_n:=\sbr{U_{\tilde z_n}}_\delta+ \xi(U_{\tilde z_n})\sbr{U_{\tilde z_n}}^\delta \in \SR_{\tilde z_n}.$$
		\medbreak
		We claim that
		\begin{equation}\label{upper bound}
			\text{	the sequence  }\lbr{\xi(U_{\tilde z_n})}_n \text{is  bounded  in} ~ \R.
		\end{equation} 
		Indeed, by definition, we have  $\xi(U_{\tilde z_n}) $  is positive for all $n\in \N$. Then we prove that $\lbr{\xi(U_{\tilde z_n})}_n$ has a    upper bound. Assume by contradiction,  we can choose a subsequence, still denoted by  $\xi(U_{\tilde z_n})$, such that $\xi(U_{\tilde z_n})\to +\infty$ as $n\to+ \infty$.      Since  $\supp{ \sbr{U_{\tilde z_n}}^\delta} \subset B(\tilde z_n, R_\delta)$ and  $\tilde{v}_n \in \SR_{\tilde z_n}$, we deduce from \eqref{e7} that 
		\begin{equation} \label{e8}
			\begin{aligned}
				&	\int_{ B(\tilde z_n, R_\delta)} \mbr{|\nabla \sbr{U_{\tilde z_n}}^\delta|^2 +V(x)\sbr{\sbr{U_{\tilde z_n}}^\delta}^2} ~\mathrm{d}x	 \\
				= &\int_{ B(\tilde z_n, R_\delta)} \sbr{U_{\tilde z_n}}^\delta\sbr{\frac{\delta}{\xi(U_{\tilde z_n})}+ \sbr{U_{\tilde z_n}}^\delta} \log\sbr{\delta+\xi(U_{\tilde z_n})\sbr{U_{\tilde z_n}}^\delta}^2~\mathrm{d}x \\  -& \frac{1}{\xi(U_{\tilde z_n})} \int_{ B(\tilde z_n, R_\delta)} \delta V(x)\sbr{U_{\tilde z_n}}^\delta~\mathrm{d}x.
			\end{aligned}
		\end{equation}
		Then it is easy to see that the left-hand side of \eqref{e8} is bounded, and the right-hand side of \eqref{e8} tends to $+\infty$ as $n\to +\infty$, which is a contradiction.   The claim is proved. Therefore,  $\lbr{ \tilde v_n}_n  $ is   bounded in $L^2(\RN)$. 
		\medbreak 
		Note that  
		\begin{equation}\label{e9}
			\begin{aligned}
				\alpha(\tilde z_n)&= \inf_{ u\in \SR_{\tilde z_n}} \LR(u) \leq \LR \sbr{\tilde v_n} \\
				&=\LR^\infty(\tilde v_n)+\frac{1}{2} \int_{\RN} \vartheta(x) \sbr{\tilde v_n(x)}^2 ~\mathrm{d} x\\
				&\leq \LR^\infty(U_{\tilde z_n})+\frac{1}{2} \left|\int_{\RN} \vartheta(x) \sbr{\tilde v_n(x)}^2 \dx\right|.
			\end{aligned}
		\end{equation}
		Then it follows from  \eqref{H1} that  for all $\varepsilon>0$, there exists $R=R(\varepsilon)$ such that $|\vartheta(x)|<\varepsilon$ for all $|x|>R$.	For $n$ large enough, 	since  $\sup_{B(0,R)}\sbr{\tilde v_n(x)}^2 \to 0$ as $n\to +\infty$ and  $\lbr{ \tilde v_n}_n  $ is   bounded in $L^2(\RN)$, we have 
		\begin{equation} \label{e19}
			\begin{aligned}
				\left|\int_{\RN} \vartheta(x) \sbr{\tilde v_n(x)}^2 \dx\right| &\leq \sup_{B(0,R)}\sbr{\tilde v_n(x)}^2  \int_{B(0,R)} |\vartheta(x) |\dx\\&+ \varepsilon \int_{\RN\setminus B(0,R)} \sbr{\tilde v_n(x)}^2\dx   \leq C \varepsilon.
			\end{aligned}
		\end{equation}
		It follows from    \eqref{e9}  and the fact $ \CR^\infty= \LR^\infty(U_{\tilde z_n} )$ that $\varLambda_1\leq \CR^\infty$, which contradicts to \eqref{e11}. Therefore,  $\lbr{ \tilde z_n}_n$ is bounded in $\RN$.  
		As a consequence, up to subsequence, there exists  $\tilde z \in\RN$ such that $ \tilde z_n \to \tilde z$ as $n\to +\infty$. 
		\medbreak
		Take $\tilde w \in \MR_{\tilde z}$   and  set $w_n(x):= \tilde{w}(x+\tilde{z} -\tilde{z}_n)$, then $\tilde w_n:=(w_n )_\delta+\xi(w_n) (w_n)^\delta \in \SR_{\tilde z_n}$. Therefore, 
		\begin{equation}
			\varLambda_1=\lim_{n\to+\infty} \alpha(\tilde z_n) \leq \lim_{n\to+\infty} \LR(\tilde w_n) =\LR(\tilde w) =\alpha(\tilde z)\leq \varLambda_1.
		\end{equation}
		Hence,  $\alpha(\tilde z)= \varLambda_1$. The proof is completed.
	\end{proof}
	\vskip0.15in
	\begin{remark} \label{Remark 3.4}
		{ \rm It is worth pointing out that from the proof of Proposition \ref{prop 3.2}, we can  draw the following conclusion: for any  sequence  $\lbr{z_n}_n$  such that $|z_n| \to +\infty$ as $n\to +\infty$,  there holds
			\begin{equation}
				\alpha(z_n) \leq \CR^\infty +o_n(1).
		\end{equation} }
	\end{remark}
	
	\vskip0.1in

	Before proceeding, we present a continuity result. 
	\begin{lemma}\label{lemma 3.5}
		Let $k\in \N\setminus\lbr{0}$ and  $\sbr{z_1^n,\ldots, z_k^n} \in \ER_{k}$ such that $ \sbr{z_1^n,\ldots, z_k^n}   \to \sbr{z_1 ,\ldots, z_k }$ as $n\to +\infty$,
	then	we have  
		\begin{equation}\label{m7}
			\lim_{n\to+\infty}  \alpha\sbr{z_1^n,\ldots, z_k^n}  = \alpha\sbr{z_1 ,\ldots, z_k }.
		\end{equation}
		
	\end{lemma}
	\begin{proof}
		
		Take $ u \in \MR_{ z_1,\ldots, z_k}$, and set 
		\begin{equation}
			\tilde w_n := \sbr{ \hat w_n}_\delta +\sum_{i=1}^k \xi_i( \hat w_n)   {\hat w_n}^{\delta,i}.
		\end{equation}
		Here,  $ {\hat w_n}^{\delta,i}(x) :=    u  ^{\delta,i}(x + {z}_i-{z}_i^n)$ for $i\in \lbr{1,\ldots,k}$, and $\sbr{ \hat w_n}_\delta$ is    the unique and positive minimizer for the minimization problem  $\min\lbr{ \LR(u): u\in \mathcal{B}_n^1}$, where 
		$$ \mathcal{B}_n^1:= \lbr{u\in H^1(\RN)  : ~ |u(x)|\leq \delta ~~ \forall x\in \RN, ~ \text{ and  } ~ u=\delta \text{ on } \bigcup_{i=1}^k \supp{\sbr{\hat w_n}^{\delta,i}}}.$$  
		Therefore,    $\tilde w_n \in \SR_{ z_1^n,\ldots,  z_k^n}$ and 
		\begin{equation}\label{m6}
			\limsup_{n\to+\infty}	\alpha( z_1^n,\ldots,  z_k^n)\leq \limsup_{n\to+\infty} \LR(\tilde{w}_n) = \LR( {u}) =\alpha \sbr{  z_1,\ldots,  z_k} .
		\end{equation}
		On the other hand, choose  $ u_n \in \MR_{ z_1^n,\ldots, z_k^n}$, and set 
		\begin{equation}
			\tilde v_n := \sbr{ \hat  v_n}_\delta +\sum_{i=1}^k \xi_i( \hat v_n)   {\hat v_n}^{\delta,i}.
		\end{equation}
		Here,  $ {\hat v_n}^{\delta,i}(x) :=    u_n  ^{\delta,i}(x + {z}_i^n-{z}_i)$ for $i\in \lbr{1,\ldots,k}$ and $\sbr{ \hat v_n}_\delta$ is    the unique and positive minimizer for the minimization problem  $\min\lbr{ \LR(u): u\in \mathcal{B}_n^2}$, where 
		$$ \mathcal{B}_n^2:= \lbr{u\in H^1(\RN)  : ~ |u(x)|\leq \delta ~~ \forall x\in \RN, ~ \text{ and  } ~ u=\delta \text{ on } \bigcup_{i=1}^k \supp{ {\hat v_n}^{\delta,i}}}.$$   
		Therefore,      $\tilde v_n \in \SR_{ z_1,\ldots,  z_k}$,  and 
		\begin{equation}
			\alpha \sbr{  z_1,\ldots,  z_k} \leq  \liminf_{n\to+\infty} \LR(\tilde v_n) =\liminf_{n\to+\infty} \LR( u_n)=\liminf_{n\to+\infty} 
			\alpha( z_1^n,\ldots,  z_k^n) ,
		\end{equation}
		which, together with \eqref{m6}, gives \eqref{m7}. This completes the proof.
	\end{proof}

	\begin{proposition} \label{prop 3.3}
		For all $k\in \N\setminus\lbr{0}$,  the following statements hold true:
		\vskip 0.05in
		\begin{itemize}
			\item[(i)]there exists $\sbr{\tilde z_1,\ldots,\tilde z_k} \in \ER_k$ such that $\varLambda_k= \alpha\sbr{\tilde z_1,\ldots,\tilde z_k}.$  
			\item[(ii)] there holds $	\varLambda_{k+1}>\varLambda_{k}+\CR^\infty$.
		\end{itemize}
		
	\end{proposition}

	\begin{proof}
		The proof is performed by mathematical induction on the number of bumps $k$.  
		
		\medbreak
		{\it Step 1}.  We consider the case  $k=1$.
		\medbreak
		Obviously, the conclusion (i) has already been proved in Proposition \ref{prop 3.2}. In such case, we only need to prove that  $	\varLambda_{2}>\varLambda_{1}+\CR^\infty$.
		\medbreak
		Let $\tilde z$ be the point obtained in Proposition \ref{prop 3.2} such that $\varLambda_{1}=\alpha(\tilde z)$, and consider a sequence $\lbr{z_n}_n$  satisfying  \eqref{defi of zn}.  Set 
		\begin{equation}
			\Theta_n:=\lbr{x\in \RN: \frac{\rho_n}{2}-1< (x\cdot e)<\frac{\rho_n}{2}+1}.
		\end{equation}
		By \eqref{H2}, we know that    $\Theta_n \cap \supp{\vartheta^-}=\emptyset$ for $n$ large enough. 
		For all $n \in \N$,  it follows from Lemma \ref{lemma 3.1} that there exists $u_n\in \MR_{\tilde z,z_n} $ such that $\LR(u_n)=\alpha(\tilde z,z_n) \leq \varLambda_{2}$.  Note that $u_n$ can be rewritten as  
		\begin{equation}
			u_n(x)=\sbr{u_n}_\delta (x)+    u_{n,1} ^\delta(x) +  u_{n,2}^\delta(x),
		\end{equation}
		where $ u_{n,1}^\delta$ and $ u_{n,2}^\delta$ are the emerging part around $\tilde z$ and $z_n$, respectively. Note that  for $n$ large enough, 
		\begin{equation}
			\supp{ u_{n,1}^\delta} \subset	B(\tilde z, R_\delta) \subset \lbr{x\in \RN:   (x\cdot e)<\frac{\rho_n}{2}-1},
		\end{equation}
		\begin{equation}
			\supp{  u_{n,2}^\delta} \subset	B( z_n, R_\delta) \subset \lbr{x\in \RN:   (x\cdot e)>\frac{\rho_n}{2}+1}.
		\end{equation}
		Hence, for   sufficiently large  $n$, we know that  $u_n=\sbr{u_n}_\delta\leq \delta  <1$  on $\Theta_n$, and 
		\begin{equation}\label{ui2}
			\mathrm{dist}\sbr{x,B(\tilde z,R_\delta)\cup B(z_n,R_\delta) \cup  B(0,\widetilde{R}) }  \geq \frac{\rho_n}{4}, \quad \forall x\in	\Theta_n.
		\end{equation} 
			Set $F_n(x)= \chi_n(x)u_n(x)$  with 	$\chi_n(x)=\chi\sbr{\left|(x\cdot e) -\frac{\rho_n}{2} \right|}$, where   $\chi(x)\in C^\infty(\R^+, [0,1])$    satisfies  
			\begin{equation}\label{defi of chi}
				\chi(t)= 0 ~~\text{ for } ~0\leq t<\frac{1}{2} , ~~\text{ and } ~~ \chi(t)=1~~  \text{ for} ~ t>1 .
			\end{equation}
			Then we have 
			\begin{equation}\label{e12}
				\begin{aligned}
					\LR(F_n) 	&= \frac{1}{2} \int_{\RN} \mbr{|\nabla \sbr{ \chi_n u_n }|^2 + \sbr{ V(x)+1 } \sbr{\chi_n u_n}^2} \dx  -\frac{1}{2} \int_{\RN} \sbr{\chi_n u_n}^2\log \sbr{\chi_n u_n}^2 \dx \\
					&= \frac{1}{2} \int_{\RN} \chi_n^2\mbr{|\nabla u_n|^2+  \sbr{ V(x)+1 }  { u_n}^2   } \dx +\frac{1}{2} \int_{\RN} \sbr{|\nabla \chi_n|^2-\frac{1}{2}\Delta \chi_n^2 }u_n^2 \dx
					\\
					&  -\frac{1}{2} \int_{\RN} \sbr{\chi_n u_n}^2\log \sbr{\chi_n u_n}^2  \dx  \\
					&\leq  \LR(u_n) + \frac{1}{2} \int_{\Theta_n} \sbr{|\nabla \chi_n|^2-\frac{1}{2}\Delta \chi_n^2  -\frac{1}{2}    \chi_n ^2 \log \chi_n^2       }u_n^2   \dx\\&+\frac{1}{2}\int_{\RN} \sbr{1-\chi_n^2} u_n^2\log u_n^2 \dx \\
					&	\leq \varLambda_{2} + C  \int_{\Theta_n} \sbr{\sbr{u_n}_\delta^2 +  \sbr{u_n}_\delta^2 \log  \sbr{u_n}_\delta^2} \dx\\
					&\leq \varLambda_{2} + C  \int_{\Theta_n}  \sbr{u_n}_\delta^2  \dx  ~ ~\leq \varLambda_{2} + O(e^{-\bar \zeta \rho_n^2}),
				\end{aligned}
			\end{equation}
			where $\bar \zeta =\zeta/8$; moreover,	 the last inequality of \eqref{e12} follows    from  \eqref{e4}, \eqref{ui2} and the fact   $\Theta_n \cap \supp{\vartheta^-}=\emptyset$ for $n$ large enough. 
			\medbreak
			On the other hand, by the choice of $\chi$, we can decompose $f_n(x)$ into $$F_n(x)= F_{n,1}(x)+F_{n,2}(x)$$ with $F_{n,1}(x)\in \SR_{\tilde{z}  } $ defined by
			\begin{equation}\label{m8}
				F_{n,1}(x):=	 \begin{cases}
					0, \quad & (x\cdot e) \geq \frac{\rho_n}{2}-\frac{1}{2} ,\\
					\chi_n(x)\sbr{u_n}_\delta (x)+   u_{n,1} ^\delta(x), \quad & (x\cdot e) <\frac{\rho_n}{2}-\frac{1}{2} ,
				\end{cases}
			\end{equation}
			and   $F_{n,2}(x) \in \SR_{  z_n}$  defined by
			\begin{equation}\label{m9}
				F_{n,2}(x):=	 \begin{cases}
					\chi_n(x)\sbr{u_n}_\delta (x)  +   u_{n,2} ^\delta(x), \quad & (x\cdot e) > \frac{\rho_n}{2}+\frac{1}{2} ,\\
					0, \quad & (x\cdot e) \leq \frac{\rho_n}{2}+\frac{1}{2} .
				\end{cases}
			\end{equation}
			Then we have  $\supp{F_{n,1}}\cap \supp{F_{n,2}}=\emptyset$, and   $$\LR(F_{n,1}) \geq \alpha (\tilde z) =\varLambda_{1 }.$$ Consider $   F_{n,2}^\infty\in \SR_{  z_n}^{\infty }$ defined by
			$$F_{n,2}^\infty:= \sbr{F_{n,2}}_\delta   + \xi^{\infty }(F_{n,2} ) \sbr{F_{n,2} }^\delta.$$
			By using a similar argument     for proving \eqref{e11},  we can prove that
			\begin{equation}\label{e14}
				\LR(F_{n,2})\geq \CR^\infty+\frac{1}{2}   \int_{ B\sbr{z_n,R_\delta}} \vartheta(x) \sbr{F_{n,2}^\infty(x)}^2 \dx- O(e^{-\bar \zeta \rho_n^2}).
			\end{equation}
			Hence, we deduce from \eqref{e12} that  for $n$ large enough,
			\begin{equation}\label{e17}
				\begin{aligned}
					\varLambda_{2 }&\geq \LR(F_n)-  O(e^{-\bar \zeta \rho_n^2})= \LR(F_{n,1})+\LR(F_{n,2})-  O(e^{-\bar \zeta \rho_n^2}) \\
					&\geq \varLambda_{1 }+\CR^\infty+\frac{1}{2}   \int_{ B\sbr{z_n,R_\delta}} \vartheta(x) \sbr{F_{n,2}^\infty (x)}^2 \dx- O(e^{-\bar \zeta \rho_n^2})\\
					&>\varLambda_{1 }+\CR^\infty,
				\end{aligned}
			\end{equation}
			where the  last inequality  is similar  as for proving \eqref{e16}.  Therefore, we finish the proof of Step 1.

			\vskip 0.15in
			
			{\it Step 2}. We consider the case  $k\geq 2$. 
			\medbreak
			We suppose by induction hypothesis that (i) and (ii) hold true for $k-1$, and let us prove them for $k$.
			\medbreak
			First, we prove that there exists $\sbr{\tilde z_1,\ldots,\tilde z_k} \in \ER_k$ such that $\varLambda_k= \alpha\sbr{\tilde z_1,\ldots,\tilde z_k}.$   For that purpose, let $\lbr{\sbr{\tilde z_1^n,\ldots,\tilde z_k^n}}_n \subset \ER_k$ be a sequence such that 
			\begin{equation}
				\varLambda_{k} = \lim_{n\to+\infty} \alpha \sbr{\tilde z_1^n,\ldots,\tilde z_k^n}.
			\end{equation}
			Then we claim  that  $\lbr{\sbr{\tilde z_1^n,\ldots,\tilde z_k^n}}_n$ is bounded. Otherwise, up to subsequence, we may assume that $|\tilde{z}_k^n|\to +\infty $ as $n\to +\infty$. For all $n\in \N$, we take 
			\begin{center}
				$\tilde{v}_n \in \MR_{ \tilde z_k^n} $ ~~and~~ $\tilde{w}_n \in \MR_{\tilde z_1^n,\ldots,\tilde z_{k-1}^n} $.
			\end{center}
			By Remark \ref{Remark 3.4}, we obtain that 
			\begin{equation}
				\LR(\tilde{v}_n ) = \alpha(\tilde z_k^n) \leq \CR^\infty +o_n(1).
			\end{equation}
			Moreover, we have 
			\begin{equation}
				\LR(\tilde{w}_n ) = \alpha(\tilde z_1^n,\ldots,\tilde z_{k-1}^n ) \leq \varLambda_{k-1}.
			\end{equation}
			Observe that $\tilde{v}_n \vee \tilde{w}_n \in \SR_{\tilde z_1^n,\ldots, \tilde z_k^n}$ and $0\leq \tilde{v}_n \wedge \tilde{w}_n \leq \delta$,  then by  Lemma \ref{lemma 2.5} we have 
			\begin{align}\label{e13}
				\alpha \sbr{\tilde z_1^n,\ldots,\tilde z_k^n} &=\inf_{\SR_{\tilde z_1^n,\ldots, \tilde z_k^n}} \LR(u) \leq \LR(\tilde{v}_n \vee \tilde{w}_n)
				\\&=\LR(\tilde v_n)+\LR(\tilde w_n)-\LR(\tilde{v}_n \wedge \tilde{w}_n) \\
				&\leq \LR(\tilde v_n)+\LR(\tilde w_n)\leq \varLambda_{k-1}+\CR^\infty +o_n(1).
			\end{align}
			Let $n\to +\infty$, we obtain $\varLambda_{k} \leq \varLambda_{k-1}+\CR^\infty$, which contradicts to the induction hypothesis. Hence, the claim is true and $\lbr{\sbr{\tilde z_1^n,\ldots,\tilde z_k^n}}_n$ is bounded in $\sbr{\RN}^k$.  Therefore,  we may assume that, up to subsequence, there exists  $\sbr{\tilde z_1,\ldots,\tilde z_k} \in \ER_k$  such that  as $n \to +\infty$,  
			\begin{equation}
				\sbr{\tilde z_1^n,\ldots,\tilde z_k^n} \to \sbr{\tilde z_1,\ldots,\tilde z_k}.
			\end{equation}
			Then by Lemma \ref{lemma 3.5}, we deduce that 
			\begin{equation}
				\varLambda_{k} = \lim_{n\to+\infty}	\alpha(\tilde z_1^n,\ldots, \tilde z_k^n)    =\alpha \sbr{\tilde z_1,\ldots,\tilde z_k} .
			\end{equation}
			
			\medbreak
			In the following, we will prove that   $\varLambda_{k+1}>\varLambda_{k}+\CR^\infty$.  For this purpose, let   $\lbr{z_n}_n$ be a sequence  satisfying   \eqref{defi of zn}.  Consider $\hat v_n \in \MR_{\tilde z_1 ,\ldots,\tilde z_{k}, z_n}$, then we write
			\begin{equation}
				\hat v_n (x)= \sbr{\hat v_n}_\delta(x)+\sum_{i=1}^{k+1}  	\hat v_n ^{\delta,i}(x) ,
			\end{equation} 
			where $	\hat v_n ^{\delta,i}$ is the  emerging part of $\hat v_n (x)$  around $\tilde z_i$ for $i\in \lbr{ 1,\ldots,k}$,  and   $ 	\hat v_n ^{\delta,k+1}$ is the emerging part of $\hat v_n (x)$  around $   z_n$. Note that for $n$ large enough, 
			\begin{equation}
				\bigcup_{i=1}^k  \supp{\hat v_n ^{\delta,i}} \subset 	\bigcup_{i=1}^k	B(\tilde z_i , R_\delta) \subset \lbr{x\in \RN:   (x\cdot e)<\frac{\rho_n}{2}-1},
			\end{equation}
			\begin{equation}
				\supp{ {	\hat v_n}^{\delta,k+1}} \subset	B( z_n, R_\delta) \subset \lbr{x\in \RN:   (x\cdot e)>\frac{\rho_n}{2}+1}.
			\end{equation}
			Let $\hat F_n(x):= \chi_n(x)\hat{v}_n(x) $, where $\chi_n$ is given by \eqref{defi of chi}. 	By the same argument as used in \eqref{e12}, we have 
			\begin{equation}\label{e15}
				\LR( \hat F_n) \leq \varLambda_{k+1} + O(e^{-\bar \zeta \rho_n^2}).
			\end{equation}	Similar to Step 1,   we can write  $$\hat F_n(x) = \hat F_{n,1}(x)+\hat   F_{n,2}(x),$$ where  $  \hat F_{n,1} \in \SR_{ \tilde z_1 ,\ldots,\tilde z_{k}  }$,~$ \hat    F_{n,2} \in \SR_{  z_n}$ and $\supp{ \hat    F_{n,1}}\cap \supp{\hat    F_{n,2}}=\emptyset$.  
			Therefore, we  have  
			\begin{equation}\label{e20}
				\LR(  \hat    F_{n,1}) \geq \alpha( \tilde z_1 ,\ldots,\tilde z_{k} ) =\varLambda_{k}.
			\end{equation}
			On the other hand,   arguing as for proving \eqref{e14}, we get 
			\begin{equation}\label{e21}
				\LR(  \hat    F_{n,2})\geq \CR^\infty+\frac{1}{2}   \int_{ B\sbr{z_n,R_\delta}} \vartheta(x) \sbr{\hat    F_{n,2} ^\infty(x)}^2 \dx- O(e^{-\bar \zeta \rho_n^2}),
			\end{equation}
			where $ \hat    F_{n,2} ^\infty(x)= ({    \hat   F_{n,2}})_\delta +\xi^{\infty} ( \hat    F_{n,2} )  { \hat    F_{n,2}}^\delta\in \SR_{  z_n}^\infty$. Similar to the proof of  \eqref{e17}, we  can deduce from \eqref{e15}, \eqref{e20} and \eqref{e21} that
			\begin{equation}
				\begin{aligned}
					\varLambda_{k+1}> \varLambda_{k}+\CR^\infty,
				\end{aligned}
			\end{equation}
			which is the desired result.  The proof of Proposition \ref{prop 3.3} is completed. 
		\end{proof}
		
		\begin{remark}\label{Remark 4.3}
			{\rm We would like to point out that from the proof of Proposition \ref{prop 3.3}, we have the following facts: 	for any $k\in \N\setminus\lbr{0}$ and $\sbr{z_1,\ldots,z_k} \in \ER_{k}$, there exists $y\in\RN$ such that $\sbr{z_1,\ldots,z_k,y} \in \ER_{k+1}$ and 
				\begin{equation}
					\alpha(z_1,\ldots,z_k) +\CR^\infty <\alpha(z_1,\ldots,z_k,y).
			\end{equation}}
		\end{remark}

		\section{ The limit behavior of  max-min functions }\label{Sect4}
		
		In this section, we study the limit behavior of the  max-min functions  for $\LR$ obtained in Section \ref{Sect3}, as. The results of Proposition \ref{prop 4.1} shows that the distance between the points around which
		these max-min functions emerge increase goes to infinity  as $|\vartheta|_{N/2,\rm loc} \to 0  $. In Proposition \ref{prop 4.2}, we describe the asymptotic shape of the emerging parts of these  max-min functions,  and we show that  they approach,    as $|\vartheta|_{N/2,\rm loc} \to 0  $,  to the Gausson \eqref{defi of Gausson}.
		Consider 
		\begin{equation}
			\GR=\lbr{ \vartheta \in L_{\loc}^{N/2}(\RN):~\vartheta(x)\geq V_0-V_\infty,  ~\lim_{|x|\to +\infty} \vartheta(x)=0, ~ \lim_{|x|\to +\infty} \vartheta(x)e^{\bar \zeta |x|^2}=+\infty }.
		\end{equation}
		For a sequence $\vartheta_n  \in\GR$  with $ |\vartheta_n|_{N/2,\rm loc} \to 0 $,     we consider the problem
		\begin{equation}
			-\Delta u+ V_n(x)  u= u\log |u|^2, \quad  u\in  H^1(\RN),
		\end{equation}
		with  $V_n(x):=\vartheta_n(x)+V_\infty$, and the corresponding functional is defined by
		\begin{equation}\label{defi of LRN}
			\LR^n  (u) := \frac{1}{2} \int_{\RN} \mbr{|\nabla u|^2+ V_n(x)u^2} ~\mathrm{d}x-\frac{1}{2} \int_{\RN} u^2\sbr{\log u^2-1}~\mathrm{d}x.
		\end{equation}
		In what follows, 
		we denote  by  $  {\FR_\delta^n} $,  $\xi^n$, $\xi_i^n $,  $\SR_{  z_1 ,\ldots,  z_{k}  }^n$, $\MR_{z_1,\ldots,z_k}^n$, $\alpha^n( z_1 ,\ldots,  z_{k} )$, $\varLambda_{k}^n$, with respect to $\LR^n$, in the same way  that  $  {\FR_\delta}  $, $\xi$, $\xi_i  $, $\SR_{  z_1 ,\ldots,  z_{k}  } $, $\MR_{z_1,\ldots,z_k} $, $\alpha ( z_1 ,\ldots,  z_{k} )$, $\varLambda_{k} $ are defined with respect to $\LR$.
		
		\begin{proposition} \label{prop 4.1}
			For all $r>R_\delta$, there exists a constant $C_r>0$ such that for all $\vartheta \in \GR$ with  $\ploc{\vartheta}<C_r$, and for all integers $k\geq 2$,     let $\sbr{z_1 , \ldots, z_{k }  } \in \ER_{k }$ be a $ k $-tuples for which $ \varLambda_{k } =\alpha \sbr{ z_1,\ldots,z_k   }$, then
			
			\begin{equation}
				\min \lbr{|z_i- z_j|: ~i\neq j, ~~ i,j=1,\ldots, k } >r.
			\end{equation}
		\end{proposition}	 
		
		\begin{proof}
			Assume by contradiction that there exists $r>R_\delta$, $ \vartheta_n  \in \GR$ with  $\ploc{\vartheta_n}\to 0$, $k_n\geq2$ and  $\sbr{z_1^n, \ldots, z_{k_n}^n} \in \ER_{k_n}$ and $u_{n}\in \MR_{z_1^n, \ldots, z_{k_n}^n}^n$ such that 
			\begin{equation}
				\varLambda_{k_n}^n=\alpha^n( z_1^n, \ldots, z_{k_n}^n)=\LR^n(u_n ), 	 	\end{equation}
			and for all $n\in\N$
			\begin{equation}
				\min \lbr{|z_i^n- z_j^n|: ~i\neq j, ~~ i,j=1,\ldots, k_n }\leq r.
			\end{equation}
			Without loss of generality, we assume  that $|z_1^n-z_2^n|\leq r$ for all $n\in\N$.   We claim that  
			\begin{equation}\label{e23}
				\limsup_{n\to+\infty}  \mbr{ \alpha^n(z_1^n,z_2^n, \ldots, z_{k_n}^n ) -\alpha^n(z_2^n, \ldots, z_{k_n}^n )} <\CR^\infty.
			\end{equation}
			Then once the claim is proved, we are done, since $ \alpha^n(z_1^n,  \ldots, z_{k_n}^n )= \varLambda_{k_n}^n$ and $\alpha^n(z_2^n, \ldots, z_{k_n}^n)\leq   \varLambda_{k_n-1}^n$. Therefore, the inequality \eqref{e23} contradicts    (ii) of Proposition \ref{prop 3.3} for large $n$.
			\medbreak
			In the remainder of this proof, we will focus on proving   \eqref{e23}.  Let $$	v_{z_1^n}:=\sbr{U_{z_1^n}}_\delta+ \xi^{n  }(U_{z_1^n})\sbr{U_{z_1^n} }^\delta \in \SR_{z_1^n}^{n }, $$ and  $ w_n\in \MR_{z_2^n, \ldots, z_{k_n}^n }^n$ be  such that
			\begin{equation}
				\LRN(w_n)=\alpha^n(z_2^n, \ldots, z_{k_n}^n  ).
			\end{equation}
			Then  $v_{z_1^n} \vee w_n  \in \SR_{z_1^n,z_2^n, \ldots, z_{k_n}^n}^n$, and $0\leq v_{z_1^n}\wedge w_n \leq \delta$. Since $$\LR^n(v_{z_1^n} \vee w_n)=\LR^n( w_n)   
			+\LR^n(v_{z_1^n} )   -\LR^n(v_{z_1^n}\wedge w_n ),$$ we have 
			\begin{equation}
				\begin{aligned}
					\alpha^n( z_1^n, \ldots, z_{k_n}^n) -\alpha^n(z_2^n, \ldots, z_{k_n}^n )   &\leq \LR^n(v_{z_1^n}\vee w_n )-\LR^n( w_n)   
					=\LR^n(v_{z_1^n})   -\LR^n(v_{z_1^n} \wedge w_n ) .
				\end{aligned}
			\end{equation}
			Arguing  as for     proving \eqref{upper bound},  we  know that $\lbr{\xi^{n}(U_{z_1^n})}_n$ is bounded. Then a direct calculation shows that 
			\begin{equation}\label{e25}
				\begin{aligned}
					\LR^n \sbr{  v_{z_1^n}} 
					&=\LR^\infty(  v_{z_1^n})+\frac{1}{2} \int_{\RN} \vartheta_n(x) \sbr{ v_{z_1^n}(x)}^2 ~\mathrm{d} x\\
					&\leq \LR^\infty(U_{z_1^n})+\frac{1}{2} \left|\int_{\RN} \vartheta_n(x) \sbr{  v_{z_1^n}(x)}^2 \dx\right|\\
					&\leq \CR^\infty  + C \left|\int_{\RN} \vartheta_n(x) \sbr{ U_{z_1^n}}^2 \dx\right|.
				\end{aligned} 
			\end{equation}  
			Set $\RN=\bigcup_{i=1}^\infty Q_i$, where $Q_i$  are $N$-dimensional disjoint unit hypercubes. We deduce from the H\"{o}lder inequality that 
			\begin{equation}\label{e26}
				\begin{aligned}
					\left|\int_{\RN} \vartheta_n(x) \sbr{ U_{z_1^n}}^2 \dx\right| &\leq \sum_{i=1}^\infty	\int_{Q_i} \left|\vartheta_n(x) \sbr{ U_{z_1^n}}^2\right| \dx\\
					&\leq  \ploc{\vartheta_n} \sum_{i=1}^\infty	  |U_{z_1^n}|_{2^*,Q_i}^2 \leq C\ploc{\vartheta_n},
				\end{aligned}
			\end{equation}
			where  the last inequality follows from the  definition of    $U$ (see \eqref{defi of Gausson}). Therefore, it follows from \eqref{e25} that 
			\begin{equation} \label{e31}
				\LR^n \sbr{   v_{z_1^n}}  \leq \CR^\infty  + o_n(1).
			\end{equation}
			\vskip0.12in
			We  now estimate the term $\LR^n( v_{z_1^n} \wedge w_n)$. Note that  $0\leq  v_{z_1^n}\wedge w_n \leq \delta \leq 1$, then  by a similar argument as used in that of \eqref{e1}, we have 
			\begin{equation}
				\begin{aligned}
					\LR^n( v_{z_1^n} \wedge w_n)&\geq \frac{1}{2} (V_0+1)\int_{\RN}  {   ( v_{z_1^n} \wedge w_n)^2}~\mathrm{d}x  = \frac{1}{2} (V_0+1) \int_{\RN}  {   ( U_{z_1^n} \wedge w_n)^2}~\mathrm{d}x.
				\end{aligned}
			\end{equation}
			Since  $r>R_\delta$, we have $\supp{ w_n^{\delta,2}} \subset B(z_2^n,R_\delta) \subset B(z_1^n, 2r)$. Hence,  
			\begin{equation}
				\LR^n( v_{z_1^n} \wedge w_n) \geq C  \int_{\supp{ w_n^{\delta,2}}  }  {   ( U_{z_1^n}  )^2}~\mathrm{d}x \geq  C b_r^2 |\supp{ w_n^{\delta,2}}|,
			\end{equation} 
			where $b_r:=\inf_{B(0, 2r)} U(x)>0$. 
			In the following, we only need to show that 
			\begin{equation}\label{e35}
				\liminf_{n\to+\infty} |\supp{ w_n^{\delta,2}}| >0.
			\end{equation}
			By contradiction, we may assume that, up to subsequence, $\lim_{n\to+\infty}|\supp{ w_n^{\delta,2}}| =0$. Then up to subsequence, the relation
			\begin{equation}\label{ss33}
				\lim_{n\to+\infty} \frac{\nm{ w_n ^{\delta,2}}_H}{| w_n ^{\delta,2}|_{p}}=+\infty.
			\end{equation}
			must be true  for all $2<p<2^*$.
			Indeed, if the statement is false,  then there exists $p_0\in \sbr{2,2^*}$ and  $\hat{w}\in H^1(\RN) $ such that  $     \sbr{w_n }^{\delta,2} (\cdot-z_2^n)  /{|  w_n  ^{\delta,2} |_{p_0}} $ is bounded in $ H^1( B(0,R_\delta))$, which, up to subsequence, weakly converges to a function $\hat{w}\in H^1(\RN) $ in $ H^1( B(0,R_\delta))$, and strongly converges in $L^{p_0} ( B(0,R_\delta)) $. Hence,  $\hat w =0$ a.e. on $\RN$, which is impossible since $|\hat{w}|_{p_0}=1$. Therefore, the relation \eqref{ss33} holds.
			
			Recall that  for all $p\in \sbr{2,2^*}$, there exists a constant $T_p>0$ such that $ s^2\log s^2 \leq T_{p} s^{p}$ for all $s\in  [0,+\infty)$. Then
			we can deduce from Remark  \ref{Remark 3.1}   that 
			\begin{equation} \label{e33}
				\begin{aligned}
					\lim_{n\to+\infty} \FR_\delta^n(w_n^{\delta,2})	& =\lim_{n\to+\infty} \max_{t>0}  \FR_\delta^n(t w_n^{\delta,2})  \geq \lim_{n\to+\infty} \mathcal{F}_\delta^n\sbr{\frac{ w_n^{\delta,2}}{|w_n^{\delta,2}|_{p}}}\\ &\geq \frac{1}{2} \mbr{\sbr{\frac{\nm{w_n^{\delta,2}}_H }{ |w_n^{\delta,2}|_{p}  }}^2 - T_{p} \int_{ B\sbr{z_2^n,R_\delta} } \sbr{\delta+\frac{ w_n^{\delta,2}}{|w_n^{\delta,2}|_{p}}}^{p} ~\mathrm{d}x -C_1 } \\
					&\geq \frac{1}{2}\sbr{\frac{\nm{w_n^{\delta,2}}_H }{ |w_n^{\delta,2}|_{p} }}^2 -C_2  \to +\infty.
				\end{aligned}
			\end{equation} 
			%
			Let 	  $$v_{z_2^n}:=\sbr{U_{z_2^n}}_\delta+ \xi^n(U_{z_2^n})\sbr{U_{z_2^n}}^\delta \in \SR_{z_2^n}^n\quad \text{ and  } \quad  a_n:=(w_n)_\delta+\sum_{i=3}^{k_n} w_n^{\delta,i}.$$  
			Arguing as for proving \eqref{e31}, we have 
			\begin{equation}\label{ui4}
				\LR^n \sbr{  v_{z_2^n}}  \leq \CR^\infty  + o_n(1).
			\end{equation}
			Obviously, we know that  $\tilde w_n:=v_{z_2^n} \vee a_n  \in \SR_{z_2^n, \ldots, z_{k_n}^n}^n$ and $0\leq a_n \wedge v_{z_2^n}\leq \delta$, from which we  obtain that for all $n\in \N$, 
			\begin{equation}\label{e34}
				\LR^n(w_n)=\alpha^n(z_2^n, \ldots, z_{k_n}^n) \leq \LR^n( \tilde w_n) \quad \text{ and  } \quad \LR^n(a_n\wedge v_{z_2^n})>0.
			\end{equation} 
			Moreover, it follows from    \eqref{e32} that  $\LR^n(w_n)-\LR^n(a_n)=\FR_\delta^n(w_n^{\delta,2})  $,
			from which we get
			\begin{equation}
				\begin{aligned}
					\LR^n(\tilde w_n)-\LR^n(w_n)&= \LR^n(a_n)+\LR^n(v_{z_2^n})- \LR^n(a_n\wedge v_{z_2^n})-  \LR^n(w_n)\\
					&\leq \LR^n(a_n)+\LR^n(v_{z_2^n})-  \LR^n(w_n)
					\\&= \LR^n(v_{z_2^n})- \FR_\delta^n(w_n^{\delta,2}) . 
				\end{aligned}
			\end{equation}
			Hence  we deduce from \eqref{e33}  and \eqref{ui4} that $\LR^n(\tilde w_n)<\LR^n(w_n)$ for $n$ large enough, which contradicts to \eqref{e34}. Therefore, \eqref{e23} holds, and the proof is completed.
		\end{proof}
		\vskip0.15in
		\begin{corollary}\label{coro 4.1}
			Let $\lbr{\vartheta_n}_n$ be any sequence of functions in $\GR$ such that $\ploc{\vartheta_n}\to 0$ and for all $n$, $k_n\in \N$, $k_n \geq 2$, we consider $\sbr{z_1^n, \ldots, z_{k_n}^n} \in \ER_{k_n}$ and $u_n\in \MR_{z_1^n, \ldots, z_{k_n}^n}^n$ such that 
			\begin{equation} 
				\varLambda_{k_n}^n= \alpha^n( z_1^n, \ldots, z_{k_n}^n) =\LR^n(u_n ), 	\end{equation}
			then 
			\begin{equation}\label{l1}
				\lim_{n\to+\infty} \min \lbr{|z_i^n- z_j^n|: ~i\neq j, ~~ i,j=1,\ldots, k_n }=+\infty.
			\end{equation}
		\end{corollary}
		\vskip0.15in
		
		\begin{lemma}\label{lemma 4.1}
			Let $\lbr{\vartheta_n}_n$ be a  sequence  satisfying the conditions in Corollary \ref{coro 4.1}, and $ V_n(x)=\vartheta_n(x)+V_\infty$. Let $k_n\in \N\setminus \lbr{0}$,  $\sbr{z_1^n,\ldots,z_{k_n}^n}\in \ER_{k_n}$, and  $u_{n} \in \MR_{z_1^n,\ldots,z_{k_n}^n}^n$.  Then  $\sbr{u_{n}}_\delta$ is a solution of 
			\begin{equation}\label{e555}
				\begin{cases}
					-\Delta u +  V_n(x)u=u\log u^2 \quad &\text{ in } ~~\RN\setminus \supp{u_{n}^\delta},\\
					u=\delta \quad &\text{ on }~~ \supp{u_{n}^\delta},\\
					u>0 \quad &\text{ in } ~~\RN.
				\end{cases}
			\end{equation}
			Set   $d(x):=\mathrm{dist}\sbr{x,\cup_{i=1}^{k_n} B(z_i^n,R_\delta) }$. Then 
				there exists    $  a \in  \sbr{0, \sqrt{V_0}} $  such that 
				\begin{equation}\label{h4}
					0<\sbr{u_{n}}_\delta(x)<\widetilde{C}_{ a} e^{-  a \sbr{d(x)}^2}, \quad \forall x\in \RN\setminus \sbr{\cup_{i=1}^{k_n} B(z_i^n,R_\delta) },
				\end{equation}
				where   $\widetilde{C}_{  a}$ depends neither on $n$ or on  $k_n$.
		\end{lemma} 
		\begin{proof}
			Obviously, we know that $\sbr{u_{n}}_\delta$  solves    $ -\Delta u +  lu\leq u\log u^2$ in $\RN\setminus \sbr{\cup_{i=1}^{k_n} B(z_i^n,R_\delta)}$ with  $l=V_0$. Then  the relation \eqref{h4} follows a similar argument as for proving  \eqref{e4}.
		\end{proof}
		
		\vskip 0.15in
		\begin{proposition} \label{prop 4.2}
			Let $\lbr{\vartheta_n}_n$ be any sequence of functions in $\GR$ such that $\ploc{\vartheta_n}\to 0$ and for all $n$, $k_n\in \N\setminus\lbr{0}$,  we consider $\sbr{z_1^n, \ldots, z_{k_n}^n} \in \ER_{k_n}$ 	and $u_n\in \MR_{z_1^n, \ldots, z_{k_n}^n}^n$ with	\begin{equation} \label{r12}
				\lim_{n\to+\infty} \min \lbr{|z_i^n- z_j^n|: ~i\neq j, ~~ i,j=1,\ldots, k_n }=+\infty   \quad \text{when }~ k_n\geq 2.
			\end{equation} 
			Then  for all $r>0$   , there holds
			\begin{equation}
				\lim_{n\to+\infty} \sbr{ \sup \lbr{|u_n (x+z_i^n) -U(x)| : i\in \lbr{1,\dots,k_n}, |x|\leq r}}=0   .
			\end{equation}
		\end{proposition}
		\vskip 0.1in
		\begin{proof}
			
			Let $k_n\geq 1$, we take $i\in \lbr{1,\ldots,k_n}$ being fixed. Without loss of generality, we only consider the case $i=1$ and prove  that the relation
			\begin{equation} \label{e24}
				\lim_{n\to+\infty} |u_n (x+z_1^n) -U(x)|=0, \quad \text{ for } ~|x|\leq r
			\end{equation}
			holds   for all $r>0$. The other cases $i\in \lbr{2 ,\ldots,k_n}$ follow in the same way.  
			\medbreak
			For $k_n \geq 1$, we set 
			\begin{equation}\label{r3}
				f_n(x):=	 \begin{cases}
					u_n(x), \quad &\text{ when } k_n=1,\\ 
					\tau_n(x)u_n(x), \quad & \text{ when } k_n \geq 2,
				\end{cases}
			\end{equation}
			with $\tau_n(x)=\tau\sbr{\left| x-z_1^n \right|-\frac{d_n}{2}}$,  where    $\tau(x)\in C^\infty(\R , [0,1])$  satisfies \begin{equation}\label{defi of  tau}
				\tau(t)= 0 ~~\text{ for } ~  |t|<\frac{1}{2} , ~~\text{ and } ~~ \tau(t)=1~~  \text{ for} ~ | t|>1 ,
			\end{equation} and $d_n= \min \lbr{|z_i^n- z_j^n|: ~i\neq j, ~~ i,j=1,\ldots, k_n }$. Then by \eqref{r12}, for $n$ large enough,  when $k_n\geq 2$, we can decompose $f_n(x)$ into $f_n(x)=f_{n,1}(x)+f_{n,2}(x)$ with $f_{n,1}(x) \in \SR_{z_1^n}^n$  and  $f_{n,2}(x)\in \SR_{ z_2^n, \ldots, z_{k_n}^n}^n$.   For convenience, in the following,       we      define $f_{n,1}=u_n$ and $f_{n,2}=0 $ when $k_n=1$.   Moreover, since $d_n\to \infty$, for $n$ large enough, $\supp{f_{n,1}}\cap \supp{f_{n,2}}=\emptyset$ and 
			\begin{equation}
				\supp{f_{n,1}} \subset	B(z_1^n, R_\delta) \subset \lbr{x\in \RN:  \left| x-z_1^n \right|  <\frac{d_n}{2}-\frac{1}{2}},
			\end{equation}
			\begin{equation}
				\supp{f_{n,2}} \subset \bigcup_{i=2}^{k_n}	B( z_i^n, R_\delta) \subset \lbr{x\in \RN:   \left| x-z_1^n \right| >\frac{d_n}{2}+\frac{1}{2}}.
			\end{equation}
			The proof of \eqref{e24} will be carried out by proving  the following  three points.
			\medbreak
			\begin{itemize}
				\item[(i). ]   $  \LR^n (f_{n,1})\to \CR^\infty $ as $n\to +\infty$.
				\medbreak
				\item[(ii). ]  $f_{n,1}(\cdot+ z_1^n) \rightarrow U$ strongly in $H^1(\RN)$ as $n\to +\infty$.
				
				\medbreak
				\item[(iii).]  $f_{n,1}(\cdot+ z_1^n) \rightarrow U$ uniformly in $K$ as $n\to +\infty$, $\forall K \subset \subset B(0, R_\delta)$ compact.
			\end{itemize} 
			\medbreak
			\noindent The remainder  part of this proof    is mainly devoted to proving that points (i), (ii),  (iii) hold.
			\medbreak
			
			{\it Proof of point (i).  }   First of all we want to show that 
			\begin{equation}\label{r1}
				\lim_{n\to+\infty} \LR^n(f_{n,1}) =\lim_{n\to+\infty}  \alpha^n(z_1^n) .
			\end{equation}
			When $k_n=1$, the relation \eqref{r1} is obviously true because $ f_{n,1} (x) =u_n(x) \in \MR_{z_1^n}^n $. 
			In the following,  we focus on  the case  when $k_n\geq 2$.  Since $f_{n,1} \in \SR_{z_1^n}^n$, it is easy to see that 
			\begin{equation}\label{r5}
				\lim_{n\to+\infty}  \alpha^n(z_1^n)    \leq \lim_{n\to+\infty}  \LR^n(f_{n,1}).
			\end{equation}
			Now we prove the reserve inequality.
			Arguing as in \eqref{e12} and taking into account the asymptotic decay  \eqref{h4} of $u_n$, we have 
			\begin{equation}
				\LR^n(f_{n,1})+\LR^n(f_{n,2})	=\LR^n(f_n)   \leq \LR^n(u_n) +O(e^{-\bar a d_n^2}),
			\end{equation}
			where $\bar a=a/8$.
			Let $   v_n \in \SR_{z_1^n}^n$ be such that it satisfies   $	\LR^n(    v_n)=\alpha^n(z_1^n)$. Then we define  $$\bar v_n(x):=\varsigma_n(x) {v}_n(x)$$ with $\varsigma_n(x)=\varsigma\sbr{\left| x-z_1^n \right|-\frac{d_n}{2}}$,  where    $\varsigma(x)\in C^\infty(\R , [0,1])$  satisfies $	\varsigma(t)= 1$ if $t <\frac{1}{4}$ and  $\varsigma(t)=0$  if $t >\frac{1}{2}$. Then  by  the  construction, we have 
			\begin{equation}
				\bar 	v_n \in \SR_{z_1^n}^n \quad \text{ and } \quad \bar v_n+f_{n,2} \in \SR_{z_1^n,z_2^n, \ldots, z_{k_n}^n}^n. 
			\end{equation}
			Moreover,  it is not difficult to see that $\LR^n(\bar v_n)= \LR^n( v_n)+ O(e^{-\bar a d_n^2})$ and 
			\begin{equation}
				\LR^n(u_n)= \alpha^n(z_1^n, \ldots,z_{k_n}^n)\leq \LR^n (\bar v_n+f_{n,2})= \LR^n (\bar v_n )+ \LR^n (  f_{n,2}).
			\end{equation}
			Then we deduce that 
			\begin{equation}
				\begin{aligned}
					\LR^n(f_{n,1})& \leq \LR^n(u_n)-\LR^n(f_{n,2}) +O(e^{-\bar a d_n^2})\\
					&\leq  \LR^n (  {v}_n)+O(e^{-\bar a d_n^2})\\
					&=\alpha^n(z_1^n)+O(e^{-\bar a d_n^2}),
				\end{aligned}
			\end{equation}
			from which, we have  $\lim\limits_{n\to+\infty}  \alpha^n(z_1^n)   \geq \lim\limits_{n\to+\infty}  \LR^n(f_{n,1})$ since  $d_n \to  +\infty$. Hence, we deduce from  \eqref{r5} that   \eqref{r1} holds.
			
			\medbreak
			Now we aim to prove that  
			\begin{equation}
				\lim_{n\to+\infty}  \alpha^n(z_1^n)=\CR^\infty.
			\end{equation}
			It follows directly from   \eqref{e25}, \eqref{e26} and \eqref{e31} that 
			\begin{equation}\label{r9}
				\lim_{n\to+\infty}	\LR^n(   v_n)= \lim_{n\to+\infty}\alpha^n(z_1^n) \leq  \lim_{n\to+\infty}\LR^n(v_{z_1^n})\leq \CR^\infty.
			\end{equation}
			In order to prove  the reserve inequality, 
			we  consider $
			{v}_n^\infty=\sbr{ v_n}_\delta+  \xi^\infty(  v_n) { { {v}_n^{\delta}}}  \in \SR_{z_1^n}^\infty $. Observe that 
			\begin{equation}
				\begin{aligned}
					\LR^n(  v_n)	 =\max_{t>0} \LR^n \sbr{\sbr{  v_n}_\delta+ t    v_n ^\delta } \geq  \LR^n({v}_n^\infty),
				\end{aligned}
			\end{equation} 
			it follows from \eqref{r9} that   $\LR^n({v}_n^\infty)$ is bounded in $\R$.
			Arguing analogous for proving   Lemma \ref{lemma 2.10}, we deduce that   $ {v}_n^\infty$ is bounded in $H^1(\RN)$.   
			Moreover, by a similar argument as used in that of \eqref{e26}, we have 
			\begin{equation}
				\lim_{n\to+\infty} \int_{\RN} \vartheta_n(x) \sbr{ {v}_n^\infty (x)}^2 ~\mathrm{d} x=0.
			\end{equation} 
			Then  we infer  from 
			\begin{equation}
				\LR^n(  v_n) \geq \LR^n({v}_n^\infty) =\LR^\infty({v}_n^\infty) +\frac{1}{2} \int_{\RN} \vartheta_n(x) \sbr{ {v}_n^\infty  (x)}^2 ~\mathrm{d} x, 
			\end{equation}
			that   $	\lim\limits_{n\to+\infty}	\LR^n(  v_n) \geq \CR^\infty$,  which completes the proof of (i).
			
				%
			
			\vskip 0.1in 
			{\it Proof of point  (ii).  }    Consider
			$     f_{n,1}^\infty:=\sbr{  f_{n,1}}_\delta+ \xi^\infty(   f_{n,1})  {   f_{n,1} ^\delta} \in \SR_{z_1^n}^\infty$,  we first claim that  
			\begin{equation} \label{f1}
				\lim_{n\to+\infty}  \LRI ( f_{n,1}^\infty)=\CR^\infty.
			\end{equation}
			Arguing as in   point (i), we can similarly prove that $ f_{n,1}^\infty$ is bounded in $H^1(\RN)$. Moreover,
			\begin{equation}\label{f2}
				\lim_{n\to+\infty} \int_{\RN} \vartheta_n(x) \sbr{  f_{n,1}^\infty (x)}^2 ~\mathrm{d} x=0.
			\end{equation}  Since $f_{n,1} \in \SR_{z_1^n}^n$ and $ f_{n,1}^\infty \in \SR_{z_1^n}^\infty$, using  Lemma \ref{lemma 3.3},  we deduce  from point (i) and \eqref{f2}  that 
			\begin{equation}
				\begin{aligned}
					\CR^\infty &=\lim_{n\to+\infty} \LRN(f_{n,1})=\lim_{n\to+\infty} \max_{t>0} \LRN \sbr{\sbr{  f_{n,1}}_\delta+ t     f_{n,1}  ^\delta}   \\
					&\geq  \lim_{n\to+\infty} \LRN( f_{n,1}^\infty  ) 
					\\&=\lim_{n\to+\infty} \LRI( f_{n,1}^\infty  ) +\frac{1}{2}	\lim_{n\to+\infty} \int_{\RN} \vartheta_n(x) \sbr{  f_{n,1}^\infty  (x)}^2 ~\mathrm{d} x\\
					&\geq  \CR^\infty,  
				\end{aligned}
			\end{equation}
			which implies that \eqref{f1} holds. Moreover, we can see that   
			\begin{equation}\label{f5}
				\xi^\infty( f_{n,1}) \to 1, \quad \text{ as }n\to +\infty.
			\end{equation} 
			Set  $g_n:=  f_{n,1}^\infty(\cdot+ z_1^n) \in \SR_{0}^\infty$. Then $\lbr{g_n}_n$ is bounded in $H^1(\RN)$, so up to subsequence,  there exists $g \in H^1(\RN)$, such that   as $n \to +\infty$, 
			\begin{equation}  \label{f3}
				\begin{aligned}
					&g_n \to g  \quad \text{ weakly  in }  H^1(\RN), \\
					&g_n  \to g  \quad \text{ strongly in }  L^p_{\loc}(\RN),   \\
					& g_n \to g  \quad \text{ a.e. in }   \RN.
				\end{aligned}
			\end{equation}
			It follows that $g_n \to g$ in $ L^p(B(0,R_\delta))$.
			Moreover,  we have $g_\delta \leq \delta $ in $\RN \setminus B(0,R_\delta)$ and $\beta_i(g)=0$ by the continuity of $\beta_i$.  Furthermore, by the uniform exponential decay of $f_{n,1}$  (inherited from  $u_n$) and observe that $ f_{n,1}^\infty=f_{n,1}$ on $\RN \setminus \supp{  f_{n,1} ^\delta}$ and $\supp{ f_{n,1} ^\delta} \subset B(z_1^n, R_\delta)$, the following statement is true:
			\begin{equation}
				g_n \to g  \quad \text{ strongly in }  L^p (\RN),   ~~ p\geq 2.
			\end{equation} 
			Consider  $ \hat g : =  {g }_\delta+ \xi^\infty(g )  {g }^\delta$ and $\hat g_n := \sbr{g_n }_\delta+ \xi^\infty(g ) g_n  ^\delta $, then $\hat g \in \SR_{0}^\infty$ and  as $n \to +\infty$, 
			\begin{equation}\label{r11}
				\begin{aligned}
					& \hat g_n  \rightharpoonup \hat g  \quad \text{ weakly  in }  H^1(\RN),  \quad 	\hat	g_n \to \hat g  \quad \text{ strongly in }  L_{\loc}^p (\RN),   ~~ p\geq 2. 
				\end{aligned}
			\end{equation}
			Integrating  \eqref{r6}  with $\varepsilon= \sbr{2^*-2}/2$ and applying Cauchy-Schwarz inequality  we deduce from  \eqref{r11} that  
			\begin{equation}
				\abs{ \int_{\RN} F_2( 	\hat g_n ) \dx-\int_{\RN} F_2(\hat	g)\dx}  \leq C  |\hat	g_n -\hat	g|_2 \sbr{\nm{\hat	g_n }_H+ \nm{ \hat	g}_H}^{2^*/2} \to 0, \quad \text{ as }~n\to +\infty.
			\end{equation}  
			Since $g_n  \in \SR_{0}^n$, we have 
			\begin{equation}
				\LRI(  g_n ) = \LRI \sbr{ \sbr{g_n }_\delta+     \sbr{g_n }^\delta  } =  \max_{t>0} \LRI \sbr{ \sbr{g_n }_\delta+ t   \sbr{g_n }^\delta  }.
			\end{equation}
			Therefore, by the minimality of  $\CR^\infty$ and formula \eqref{f1},  along with  the weak lower semicontinuity of both the norm and $\Psi$,  we  deduce that 
			\begin{equation}\label{f4}
				\begin{aligned}
					\CR^\infty & \leq \LRI( \hat g) = \frac{1}{2} \int_{\RN} \mbr{|\nabla \hat g|^2+ (V_\infty+1)\hat g^2}~\mathrm{d}x-   \int_{\RN} F_2(\hat	g)\dx+ \Psi ( \hat g) \\
					&\leq \liminf_{n\to+\infty} \mbr{\frac{1}{2} \int_{\RN} \mbr{|\nabla \hat g_n|^2+ (V_\infty+1)\hat g_n^2}~\mathrm{d}x -\int_{\RN} F_2(\hat	g_n)\dx+ \Psi ( \hat g_n)   }
					\\&=\liminf_{n\to+\infty} \LRI( \hat g_n) = \liminf_{n\to+\infty} \LRI( \sbr{g_n }_\delta+ \xi^\infty(g )   \sbr{g_n }^\delta )  \\
					&\leq   \liminf_{n\to+\infty} \LRI(  g_n ) =\liminf_{n\to+\infty} \LRI( f_{n,1}^\infty ) =\CR^\infty.
				\end{aligned}
			\end{equation}
			Therefore, we obtain that 
			\begin{equation}
				\begin{aligned}
					\lim_{n\to+\infty}	\nm{ \hat g_n}_H =\nm{ \hat g }_H, \quad  	\lim_{n\to+\infty} \Psi ( \hat g_n)=\Psi ( \hat g ), \quad  	\lim_{n\to+\infty}\int_{\RN} F_2( 	\hat g_n ) \dx= \int_{\RN} F_2( 	\hat g  ) \dx.
				\end{aligned}
			\end{equation}
			Therefore,
			\begin{equation} \label{f6}
				\hat g_n  \xrightarrow{n \rightarrow  +\infty} \hat g  \quad \text{ strongly  in }    ~H^1(\RN).
			\end{equation}
			From \eqref{f4}, it follows  that $    \LRI( \hat g)= \CR^\infty$.    Therefore, we deduce from  Lemma \ref{lemma 3.3} that  $  \hat g = U$ since  $\hat g \in \SR_{0}^\infty$. Then using \eqref{f4} again, there holds $\xi^\infty(g )  =1$.  Combined with the definitions of   $g_n$, $\hat g_n$, $ f_{n,1}^\infty $, along with \eqref{f5} and  \eqref{f6}, we have 
			\begin{equation}
				f_{n,1}(\cdot+ z_1^n) \xrightarrow{n \rightarrow  +\infty} U \quad \text{ strongly in }  H^1(\RN).
			\end{equation}
			This completes the proof of point (ii).

			\vskip 0.1in 
			{\it Proof of point  (iii).  }  Since $u_n\in \MR_{z_1^n, \ldots, z_{k_n}^n}^n$ and $d_n \to +\infty$, by  Proposition \ref{lemma 3.2}, for all $n\in \N$ there exists $\lambda_1^n \in  \RN$ such that 
			\begin{equation}  \label{r88}
				\hbr{	\sbr{\LRN}'(f_{n,1}), \phi } =\int_{ B\sbr{z_1^n,R_\delta}}\sbr{ f_{n,1}(x)}^\delta  \phi(x)  \mbr{\lambda_1^n\cdot\sbr{x-z_1^n}} ~\mathrm{d}x, 
			\end{equation} 
			for all $ \phi \in H_0^1( B\sbr{z_1^n,R_\delta}) $.	We claim that 
			\begin{equation}\label{r97}
				\lim_{n\to+\infty}	\lambda_1^n = 0.
			\end{equation}
			In fact, suppose by contradiction that $|	\lambda_1^n| \geq C>0$ for all $n\in\N$, so up to subsequence,  there exists $\lambda\neq 0$ such that 
			\begin{equation}
				\lim_{n\to+\infty} \frac{ \lambda_1^n}{|	\lambda_1^n|} =  \lambda.
			\end{equation}
			Setting  $ \widetilde V_n(x )= V_\infty+\vartheta_n(x+z_1^n)$.   Testing  \eqref{r88} with   $$ \varphi_n\sbr{x-z_1^n}= \sbr{ \frac{ \lambda_1^n}{|	\lambda_1^n|}  \cdot\sbr{x-z_1^n}}  \varphi\sbr{x-z_1^n},$$ where $\varphi \in C_0^\infty(B(0,R_\delta))$ is radial and $\varphi >0$ in $B(0,R_\delta)$,   we get 
			\begin{equation} \label{r98}
				\begin{aligned}
					& \int_{B(0,R_\delta)} \mbr{\nabla f_{n,1} (x+z_1^n)  \cdot \nabla  \varphi_n(x) +\widetilde V_n(x) f_{n,1}(x+z_1^n)  \varphi_n(x)} ~\mathrm{d}x \\
					&-\int_{B(0,R_\delta)}  f_{n,1} (x+z_1^n) \varphi_n(x)\log  \sbr{f_{n,1}  (x+z_1^n)}^2~\mathrm{d}x \\
					&= |\lambda_1^n|\int_{ B\sbr{0,R_\delta}}\sbr{ f_{n,1} }^\delta(x+z_1^n) \varphi (x)  \sbr{ \frac{ \lambda_1^n}{|	\lambda_1^n|} \cdot x}^2 ~\mathrm{d}x .
				\end{aligned}
			\end{equation}
			We deduce from point (ii) that the left hand side of \eqref{r98} tends to $\hbr{\sbr{\LRI}^\prime(U),\sbr{\lambda \cdot x}\varphi }=0$  as $n \to \infty$,  and the right hand side of \eqref{r98} satisfies 
			\begin{equation}
				\begin{aligned}
					&|\lambda_1^n|\int_{ B\sbr{0,R_\delta}}\sbr{ f_{n,1} }^\delta(x+z_1^n) \varphi (x)  \sbr{ \frac{ \lambda_1^n}{|	\lambda_1^n|} \cdot x}^2 ~\mathrm{d}x\\
					&=\sbr{\int_{ B\sbr{0,R_\delta}} U^\delta(x) \varphi (x)  \sbr{  \lambda \cdot x}^2 ~\mathrm{d}x +o_n(1)}|\lambda_1^n| \geq C|\lambda_1^n|,
				\end{aligned}	
			\end{equation}so we get a contradiction since $|	\lambda_1^n| \geq C>0$, and we finish the proof of \eqref{r97}.
			\medbreak 
			Finally,  we will show that $f_{n,1}(\cdot+ z_1^n) \rightarrow U$ uniformly in $K$ as $n\to +\infty$, $\forall K \subset \subset B(0, R_\delta)$ compact.
			In fact, by \eqref{r88}, we infer that $f_{n,1} (\cdot +z_1^n)$  is a weak solution to the  following equation
			\begin{equation}\label{f12}
				\begin{aligned}
					-\Delta u + \widetilde V_n(x) u  = u  \log {u }^2 +  (\lambda_1^n\cdot x)u^\delta , \quad x\in B(0,R_\delta).
				\end{aligned}
			\end{equation}
			Observe that  any  nonnegative  solution to \eqref{f12} satisfies the inequality
			\begin{equation}
				-\Delta u+V_0 u \leq (u  \log {u }^2)^++\tilde{c}u^\delta,
			\end{equation}
			where $\tilde{c}>0$ is independent of $n$. Since $0\leq u^\delta\leq u$, then for any $\theta>0$, there exists $C_\theta>0$ such that 
			\begin{equation}
				-\Delta u \leq l(u), \quad  l(s)=  (\tilde{c} -V_0)s+C_\theta s^{1+\theta}.
			\end{equation}
			Since we have $|l(s)|\leq C(1+|s|^{1+\theta})$ for all $s\in\R$ and some $C>0$, then repeating the bootstrap argument of the proof of \cite[Lemma B.3]{Struwe-book}, one can prove that $f_{n,1} (\cdot +z_1^n)  \in L^p( B(0,R_\delta))$ for any $p>2$. Then, by standard regularity arguments, $f_{n,1} (\cdot +z_1^n) \in C^{1+\varepsilon}(K)$ for  any compact  set $  K \subset \subset B(0, R_\delta)$. Moreover, it follows from \eqref{r88} and \eqref{r97} that the sequence $\lbr{f_{n,1} (\cdot +z_1^n)}_n$ is uniformly bounded in    $C^{1+\varepsilon}(K)$ for some $\varepsilon>0$. Thus it is uniformly convergence to $U$  in all compact set $  K \subset \subset B(0, R_\delta)$.  The proof of point (iii) is completed.
			\vskip 0.1in
			Since  we  complete the proofs  of  points (i), (ii),  (iii), we are in position to prove the relation \eqref{e24}, which completes the proof of Proposition \ref{prop 4.2}.
			Indeed, the choice of $\delta$ and $R_\delta$ implies that $U(x)<\delta$ for $|x|\geq R_\delta/2$. Observe that   $u_n =f_{n,1} $
			in $B(z_1^n,{d_n}/{2}-1)$ and $d_n \to+ \infty$ as $n\to +\infty$,  then  for all fixed $R\geq R_\delta$ and for large $n$,  we  deduce from lemma \ref{lemma 4.1} and \eqref{f12} that  $f_{n,1}(\cdot+z_1^n)$  is a weak solution to the following equation,
			\begin{equation} 
				-\Delta u + \widetilde  V_n(x) u= u\log u^2 \quad \text{ in } ~~ B(0,R) \setminus B(0,R_\delta/2).
			\end{equation}
			Arguing as the regularity arguments above and taking into account point (ii),  the fact  $\widetilde  V_n(x)$ converges to $    V_\infty$ uniformly in $B(0,R) \setminus B(0,R_\delta/2)$  that 
			\begin{equation}
				f_{n,1}(\cdot+z_1^n)\xrightarrow{n \rightarrow  +\infty} U \quad \text{ uniformly in }~~ B(0,R) \setminus B(0,R_\delta/2),
			\end{equation}
			which, together with point (iii),  implies the relation  \eqref{e24}. Here we also use the fact that  the gausson $U$  is the unique (up to translation) positive $C^2$-solution of the problem \eqref{p infty}. 
		\end{proof}
		
		\begin{corollary} \label{coro 4.2}
			Under the assumptions of Proposition \ref{prop 4.2}, for all $ i\in \lbr{1,\ldots, k_n} $ and  for large $n$, there  exists $\lambda_i^n \in \RN$ such that 
			\begin{equation} 	 \label{l2}
				-\Delta u_n(x) +  V_n(x) u_n (x) = u_n(x)  \log  u_n^2(x)   +  \sum_{i=1}^{k_n}u_n^{\delta,i} (x)  \mbr{  \lambda_i^n\cdot \sbr{x-z_i^n}}  , \quad x\in \RN,   
			\end{equation}
			where $ V_n(x )= V_\infty+\vartheta_n(x )$. 
		\end{corollary}
		\begin{proof}
			The proof is almost identical to that of \cite[Corollary 5.5]{cerami_CPMA_2013}. For the sake of completeness, we provide the details here. It follows from \eqref{e24} that for all $ i\in \lbr{1,\ldots, k_n} $ and  for large $n$, we have $\supp{ u_n^{\delta,i}} \subset B(z_i^n, R_\delta/2)$. Then by Proposition \ref{prop 3.1}, $u_n$ is a solution of 
			\begin{equation}\label{l3}
				-\Delta u_n +  V_n(x) u_n  = u_n \log  u_n^2 \quad \text{ in } \RN \setminus \bigcup_{i=1}^{k_n}  B(z_i^n, R_\delta/2).
			\end{equation}
			Moreover, since  $u_n$ satisfies  \eqref{e5},  we can proceed with the regularity arguments as outlined in point (iii) of Proposition \ref{prop 4.2}, and obtain that  $u_n$ satisfies \eqref{l2} in $B(z_i^n,\frac{3}{4}R_\delta)$, which together with \eqref{l3} gives \eqref{l2}. This completes the proof.
		\end{proof}

		\section{Existence of solutions with finitely many bumps }\label{Sect5}
		In this section, we finish the proof Theorem \ref{thm1}.		Our strategy for finding the infinitely many positive solutions of \eqref{p equ} is the following: for all  $\vartheta\in \GR$ such that   $|\vartheta|_{N/2,\rm loc}$  is small and arbitrary integer $k\geq 1$, the max-min argument, as presented in the Section \ref{Sect3}, allows us to find $k$-tuples  $\sbr{  z_1,\ldots, z_k} \in \ER_{k}$  and  functions $\tilde{u}_{k}\in\MR_{z_1, \ldots, z_{k}}$ that  are good candidates  to be critical points.
		In view of Proposition \ref{prop 3.1}, we conclude that   $\tilde u_{k}$ is the solution of \eqref{p equ} in the whole space $\RN$ expect the support of its emerging positive and negative parts. Furthermore, by Proposition \ref{lemma 3.2} we know that $\tilde u_{k }$   satisfies the relation \eqref{e5}. Hence, to complete the argument, we need to show that all the Lagrange multipliers   in \eqref{e5} are zero when  $|\vartheta|_{N/2,\rm loc}$  is small.
		
		\begin{proof}[\bf Proof of Theorem \ref{thm1}]

			For all  $\vartheta\in \GR$,  and  $k \in \N\setminus\lbr{0} $, it follows from   Lemma \ref{lemma 3.1} and Proposition \ref{prop fixed point} that  there exist  $\sbr{  z_1,\ldots, z_k } \in \ER_{k}$ and  $\tilde{u}_{k } \in\MR_{z_1, \ldots, z_{k} }$
			such that $$\LR(\tilde u_{k  })=\alpha(  z_1, \ldots,  z_{k} )=\varLambda_{k}.$$
			To prove  Theorem \ref{thm1}, it suffices to prove that if $|\vartheta|_{N/2,\rm loc}$  is small, then $\tilde u_{k }$ is  a positive solution of \eqref{p equ}, that is $\LR^\prime(\tilde u_{k })=0$, which is defined in \eqref{ll1}.
			\medbreak
			
			Arguing by contradiction, we   assume that for every integer $n\in \N\setminus\lbr{0}$, there exist a function $\vartheta_n \in\GR$ with $	|\vartheta_n|_{N/2,\rm loc} \xrightarrow{n \rightarrow  +\infty}0$, an integer  $k_n\geq 1$   and a function    $\tilde u_{k_n}  \in \MR_{z_1^n, \ldots, z_{k_n}^n}^n $    with  $k_n$-tuples $\sbr{z_1^n,\ldots,z_{k_n}^n} \in \ER_{k_n}$ such that 
			\begin{equation}
				\varLambda_{k_n}^n=  \alpha ^n( z_1^n,\ldots,z_{k_n}^n)=\LRN(u_n), \quad\sbr{\LRN}'(u_n)  \neq 0,
			\end{equation}
			In what follows,  for simplicity, we will denote  $\tilde u_{k_n} $ as $ u_n$ in the following. Then	 from relation \eqref{e5}, we know that there exist  some  $\lambda_i^n \in \RN\setminus\lbr{0}$ such that 
			\begin{equation}   \label{q8}
				\hbr{	\sbr{\LRN}'(u_n), \phi } =\int_{ B\sbr{z_i^n,R_\delta}} u_n^{\delta,i}(x)  \phi(x)  \mbr{\lambda_i^n\cdot\sbr{x-z_i^n}} ~\mathrm{d}x, 
			\end{equation} 
			for all $ \phi \in H_0^1( B\sbr{z_i^n,R_\delta})$. Without loss of generality, we may assume that 
			\begin{equation}
				|\lambda_1^n|=\max \lbr{ |\lambda_i^n|: i=1,\ldots,k_n},
			\end{equation}
			then $|\lambda_1^n|\neq 0$ for all $n$. Therefore,  up to subsequence, we assume that 
			\begin{equation}
				\lim_{n\to+\infty} \frac{ \lambda_1^n}{ |\lambda_1^n|}=\lambda.
			\end{equation}
			Let $\lbr{\sigma_n}_n$ be a sequence of   positive numbers  such that 
			\begin{equation} \label{y8}
				\lim_{n\to+\infty} \sigma_n=0,\quad \lim_{n\to+\infty} \sigma_n k_n=0,  \quad \lim_{n\to+\infty} \frac{\sigma_n}{|\lambda_1^n|}=0, \quad  \lim_{n\to+\infty}  \frac{\sigma_n k_n}{|\lambda_1^n|}=0,
			\end{equation}
			and   for all $n\in \N$,  the  $k_n$-tuples $\sbr{y_1^n,\ldots,y_{k_n}^n} $ is defined by 
			\begin{equation}\label{q10}
				y_1^n =z_1^n+\sigma_n \lambda, \quad \text{ and } \quad y_i^n=z_i^n, ~\text{ for }~  2\leq i\leq k_n.
			\end{equation}
			Considering  $v_n \in \MR_{y_1^n, \ldots, y_{k_n}^n}^n$, then by the definition of $\varLambda_{k_n}^n$,  we have
			\begin{equation} \label{q1}
				\LRN(v_n)\leq  \varLambda_{k_n}^n =\LRN(u_n).
			\end{equation}
			By Taylor expansion, we have 
			\begin{equation}\label{q2}
				\begin{aligned}
					\LRN(v_n)-\LRN(u_n)&= \hbr{(\LRN)^\prime(u_n),v_n-u_n}\\
					&+	\frac{1}{2}\mbr{ \int_{\RN} |\nabla (v_n-u_n)|^2 \dx + \int_{\RN}\sbr{ V_n(x)-2}  (v_n-u_n) ^2 \dx}\\
					&-\frac{1}{2} \int_{\RN} (v_n-u_n)^2 \log w_n^2 \dx,
				\end{aligned} 
			\end{equation}
			with  $w_n(x)= u_n(x)+ \tilde{w}_n(x)\sbr{v_n(x)-u_n(x)}$ for some $\tilde{w}_n(x) \in \mbr{0,1}$.
			\medbreak
			To conclude  this proof, we shall derive a contradiction with \eqref{q1} by  meticulously calculating the terms of \eqref{q2}.
			Specifically, our aim is to prove that for $n$ sufficiently large, the following inequality holds:
			\begin{equation}\label{q6}
				\LRN(v_n)-\LRN(u_n)>0.
			\end{equation}		 	
			For this purpose,  we first prove that  
			\begin{equation} \label{q5}
				\lim_{n\to+\infty}  |v_n-u_n|_\infty=0.
			\end{equation}
			\medbreak
			On the one hand, since $|y_1^n-z_1^n|\xrightarrow{n \rightarrow  +\infty} 0$, by applying Proposition \ref{prop 4.2} to $u_n$ and $v_n$,    
			one can  deduce   that 
			\begin{equation}\label{q3}
				\lim_{n\to+\infty}  |v_n-u_n|_{\infty,\bigcup_{i=1}^{k_n}B(z_i^n,R)} =0, \quad \forall R>0.
			\end{equation} 
			
			On the other hand, we know that $0\leq u_n(x),  v_n(x)\leq \delta$  on  $\RN\setminus \bigcup_{i=1}^{k_n}B(z_i^n,R)$ for $R>R_\delta$. Then by Corollary \ref{coro 4.2}, we infer that 
			\begin{equation}
				\begin{aligned}
					\Delta (v_n-u_n)&=(v_n-u_n) \sbr{ V_n(x)-\cfrac{v_n \log v_n^2-u_n\log u_n^2}{v_n-u_n}}\\
					&= (v_n-u_n) \sbr{ V_n(x)- \log \theta_n^2(x)-2},
				\end{aligned}	 
			\end{equation}
			with $\theta_n(x)= u_n(x)+ \tilde{\theta}_n(x)\sbr{v_n(x)-u_n(x)}$ for some $\tilde{\theta}_n(x) \in \mbr{0,1}$. 	
			\medbreak
			If $0\leq u_n(x)<v_n(x) \leq \delta$,   by  \eqref{defi of delta}, we know that  $\theta_n(x)\leq \delta \leq e^{(V_0-2)/2}$, which, together with \eqref{H1}, gives that  $$\Delta (v_n-u_n) \geq 0. $$
			Similarly, if $0\leq v_n(x)<u_n(x) \leq \delta$, we deduce that $$\Delta (v_n-u_n) \leq 0. $$
			Hence, for $R> R_\delta$, taking account into the asymptotic behavior of $u_n$ and $v_n$, we can obtain that the maximum of  $|v_n-u_n|$   $\RN\setminus \bigcup_{i=1}^{k_n}B(z_i^n,R)$ is attained  on the boundary $\bigcup_{i=1}^{k_n}B(z_i^n,R)$. Therefore, it follows from \eqref{q3} that 
			\begin{equation}\label{q4}
				\lim_{n\to+\infty}  |v_n-u_n|_{\infty, \RN\setminus \bigcup_{i=1}^{k_n}B(z_i^n,R)} =0.
			\end{equation}
			Therefore, we deduce from  \eqref{q3} and \eqref{q4} that \eqref{q5} holds.
			
			\medbreak
			Let $\hat{a}_n=|v_n-u_n|_\infty$, then $\hat{a}_n>0$ for all  $n\in \N$ since $x_1^n\neq y_1^n$. Then we define 
			\begin{equation}
				\phi_n(x):=\cfrac{ v_n(x)-u_n(x)}{\hat{a}_n}, \quad \text{ and } \quad \phi_{n,i}(x):=\phi_n(x+z_i^n)  .
			\end{equation}
			for $i \in \lbr{1,\ldots,k_n}$.
			\vskip0.15in
			
			\underline{Now we claim  that}:  there exists a     subsequence of  $\lbr{\phi_n}_n$,  still denoted by 	 $\lbr{\phi_n}_n$, such that for all $i\in \lbr{1,\ldots,k_n}$ the sequence  $\phi_{n,i}(x)$ converges in   $H^1_{loc}(\RN)$, and uniformly on compact subsets of $\RN$, to a solution $\phi$ of  the equation
			\begin{equation}
				-\Delta \phi +( V_\infty-2- \log U^2) \phi  =0.
			\end{equation}
			Suppose for now that the claim is true   (the proof will be provided later in the argument). In the following steps, we aim to confirm  relation \eqref{q6}. However, this relation leads to a contradiction with   \eqref{q1}. Consequently, we conclude the proof of Theorem \ref{thm1}.
			
			\medbreak
			
			In order  to prove \eqref{q6},   we  fix $\tilde{R}\in \sbr{\frac{3}{4}R_\delta,R_\delta}$   and introduce a smooth, decreasing,  cut-off function $h\in C^{\infty}(\R^+,[0,1])$ such that	\begin{center}
				$h(t)=1$ \quad if~ $t\leq \tilde{R}$ \quad  and \quad $h(t)=0$ \quad if~ $t\geq R_\delta$.
			\end{center} Then for all $i\in \lbr{1,\ldots, k_n}$, we define $ h_{z_i^n}(x)=\eta \sbr{|x-z_i^n|} $. Obviously, 
			\begin{equation}
				h_{z_i^n}(x)\mbr{ v_n(x)-u_n(x)}  \in H_0^1(B(z_i^n,R_\delta)).
			\end{equation} 
			Moreover, by the choice of $R_\delta$ \eqref{defi of Rdelta} and Proposition \ref{prop 4.2}, for $n$ sufficiently large, we have 
			\begin{equation}\label{q9}
				\supp{ u_n^\delta} \subset \bigcup_{i=1}^{k_n} B(z_i^n,\tilde{R}), \quad 	\supp{ v_n^\delta} \subset \bigcup_{i=1}^{k_n} B(z_i^n,\tilde{R})
			\end{equation}  
			Therefore, 	we have 
			\begin{equation}
				\int_{\R^N} u_n^\delta(x)   (1-\sum_{i=1}^{k_n}h_{z_i^n}(x))                        \mbr{ v_n(x)-u_n(x)}  \dx =0.
			\end{equation}
			Recalling that $u_n\in\MR_{z_1^n, \ldots, z_{k_n}^n}^n $, one can easily deduce  from Proposition \ref{prop 3.1} that 
			\begin{equation}
				\hbr{ (	\LRN)'(u_n),  (1-\sum_{i=1}^{k_n}h_{z_i^n})\sbr{ v_n-u_n}}=0.
			\end{equation}
			Therefore, from Proposition \ref{lemma 3.2},   we deduce   that 
			\begin{equation}
				\begin{aligned}
					\hbr{(\LRN)^\prime(u_n), v_n-u_n}&=\sum_{i=1}^{k_n}\hbr{ (	\LRN)'(u_n), h_{z_i^n} \sbr{ v_n-u_n}} \\
					&= \sum_{i=1}^{k_n}\int_{ B\sbr{z_i^n,R_\delta}} u_n^{\delta}(x)  h_{z_i^n}(x)\mbr{ v_n(x)-u_n(x)} \mbr{\lambda_i^n\cdot\sbr{x-z_i^n}} ~\mathrm{d}x\\
					&= \sum_{i=1}^{k_n}\int_{ B\sbr{z_i^n,R_\delta}} u_n^\delta(x)  \mbr{ v_n(x)-u_n(x)}\mbr{\lambda_i^n\cdot\sbr{x-z_i^n}} ~\mathrm{d}x,
				\end{aligned}
			\end{equation}
			where the last equality follows from \eqref{q9}. 
			\medbreak
			Following the arguments  in   \cite[  (6.12) and  (6.13)]{cerami_CPMA_2013}, we obtain that, for all $n\in\N$,  when $i\neq 1$,
			\begin{equation} \label{y7}
				\int_{ B\sbr{z_i^n,R_\delta}} u_n^\delta(x)  \mbr{ v_n(x)-u_n(x)}\sbr{x-z_i^n} ~\mathrm{d}x =		 
				O(\hat{a}_n^2),
			\end{equation}
			and  when $i=1$,
			\begin{equation} \label{y6}
				\int_{ B\sbr{z_i^n,R_\delta}} u_n^\delta(x)  \mbr{ v_n(x)-u_n(x)}\sbr{x-z_i^n} ~\mathrm{d}x =		 
				\frac{1}{2}\sigma_n \lambda 	\int_{ B\sbr{z_i^n,R_\delta}} \sbr{v_n^\delta(x)	}^2 \dx +O(\hat{a}_n^2).
			\end{equation}
			Since we suppose the claim is true, 	it follows from the non-degeneracy of Gausson $U$ that  for all $i\in \lbr{1,\ldots,k_n}$, there exists a vector  $ \mu_i \in \RN$ such that 
			\begin{equation} \label{y9}
				\lim_{n\to+\infty} \phi_n(x+z_i^n) = \nabla U(x)\cdot \mu_i,
			\end{equation} 
			then using  \eqref{y7} and Proposition \ref{prop 4.2},   we obtain that for all  $i \neq 1$
			\begin{equation} \label{y0}
				\begin{aligned}
					0&=\lim_{n\to+\infty} 	\int_{ B\sbr{0,R_\delta}} u_n^\delta(x+z_i^n)  \phi_n (x+z_i^n)x ~\mathrm{d}x\\
					&= \int_{ B\sbr{0,R_\delta}} U^\delta(x)  \sbr{\nabla U(x)\cdot \mu_i}x \dx \\
					&= \frac{1}{2}  \int_{ B\sbr{0,R_\delta}}  \sbr{\nabla (U^\delta(x))^2\cdot \mu_i}x \dx 
					\\&= - \frac{1}{2} \mu_i \int_{ B\sbr{0,R_\delta}}(U^\delta(x))^2 \dx.
				\end{aligned}
			\end{equation}
			Therefore,  it follows that   $\mu_i=0$ for all $i\neq 1$. Similarly, when $i=1$,  using \eqref{y6} and Proposition \ref{prop 4.2}, we deduce that
			\begin{equation}
				\frac{1}{2} \lambda  \sbr{\lim_{n\to+\infty} \frac{\sigma_n}{\hat{a}_n}}\int_{ B\sbr{0,R_\delta}}(U^\delta(x))^2 \dx= - \frac{1}{2} \mu_1 \int_{ B\sbr{0,R_\delta}}(U^\delta(x))^2 \dx,
			\end{equation}
			which implies that 
			\begin{equation}
				\lambda  \sbr{\lim_{n\to+\infty} \frac{\sigma_n}{\hat{a}_n}} =-\mu_1 .
			\end{equation}
			Then we have $\mu_1\neq 0$. Otherwise, from \eqref{y9}, we deduce that $  |\phi_{n}|_{\infty,\bigcup_{i=1}^{k_n}B(z_i^n,R)} =o_n(1)$ for all $R>0$.
			Moreover, arguing as in proving \eqref{q4}, we have $  |\phi_{n}|_{\infty,\RN\setminus\bigcup_{i=1}^{k_n}B(z_i^n,R)} =o_n(1)$ for $R\geq R_\delta$. Hence, we have $\lim_{n\to+\infty}  |\phi_{n}|_{\infty} =0$, which contradicts with the fact $ |\phi_{n}|_{\infty} =1$. 
			\medbreak
			Set 
			\begin{equation} \label{y11}
				{\lim_{n\to+\infty} \frac{\sigma_n}{\hat{a}_n}} =\kappa \in \R^{+}\setminus \lbr{0} \quad \text{ and } \quad \lambda \kappa=-\mu_1.
			\end{equation}
			Then  using \eqref{y8} and  \eqref{y0}, we deduce that 
			\begin{equation}
				\begin{aligned}
					&	\lim_{n\to+\infty} \frac{1}{\hat{a}_n|\lambda_1^n|} 	\hbr{(\LRN)^\prime(u_n),v_n-u_n } \\
					= &\lim_{n\to+\infty} \mbr{\frac{\lambda_1^n }{ |\lambda_1^n|}\cdot  \int_{ B\sbr{0,R_\delta}} u_n^\delta(x+z_1^n)   \phi_{n}(x+z_1^n) x  ~\mathrm{d}x+\sum_{i=2}^{k_n}\frac{1}{\hat{a}_n }\frac{\lambda_1^n }{ |\lambda_1^n|} O(\hat{a}_n^2)}
					\\=&  \lambda   \cdot\int_{ B\sbr{0,R_\delta}} U^\delta(x) \sbr{\nabla U(x) \cdot \mu_1}  x  \dx\\
					=&  \frac{1}{2}  |\lambda|^2 \kappa\int_{ B\sbr{0,R_\delta}}(U^\delta(x))^2 \dx >0.
				\end{aligned}
			\end{equation}
			Furthermore,  taking account into \eqref{y8}, \eqref{y11} and the assumption $V_n(x)\geq V_0>2$,  and  choosing $\hat{R}>R_\delta$ such that $u_n(x) \leq \delta<1$ and $v_n(x) \leq \delta<1$ on  $\RN\setminus \bigcup_{i=1}^{k_n}B(z_i^n,\hat{R})$, we  have 
			\begin{equation}\label{q0}
				\begin{aligned}
					&	\lim_{n\to+\infty}\frac{1}{\hat{a}_n|\lambda_1^n|} 		 \int_{\RN}  \mbr{|\nabla (v_n-u_n)|^2  + \sbr{ V_n(x)-2}  (v_n-u_n) ^2  
						-    (v_n-u_n)^2 \log w_n^2} \dx 
					\\\geq &\lim_{n\to+\infty}-\frac{1}{\hat{a}_n|\lambda_1^n|}   \mbr{ \sum_{i=1}^{k_n}  \int_{ B\sbr{z_i^n,\hat{R}}}   (v_n-u_n)^2 \log w_n^2 \dx   + \int_{\RN\setminus   \bigcup_{i=1}^{k_n}B(z_i^n,\hat{R})}    (v_n-u_n)^2 \log w_n^2        \dx          }	
					\\\geq &\lim_{n\to+\infty}-\frac{1}{\hat{a}_n|\lambda_1^n|}   ~{ \sum_{i=1}^{k_n}  \int_{ B\sbr{z_i^n,\hat{R}} \cap \lbr{w_n \geq 1}}   (v_n-u_n)^2 \log w_n^2 \dx    }	\\
					\geq &\lim_{n\to+\infty}-\frac{\hat{a}_n}{|\lambda_1^n|}   ~{ \sum_{i=1}^{k_n}  \int_{ B\sbr{z_i^n,\hat{R}} \cap \lbr{w_n \geq 1}}     \log w_n^2 \dx    }	
					\\ \geq &\lim_{n\to+\infty}-\frac{Ck_n\sigma_n}{|\lambda_1^n|}\times  \frac{\hat{a}_n }{\sigma_n}=0.
				\end{aligned} 
			\end{equation}	 
			Therefore, it follows from \eqref{q2} that 
			\begin{equation}
				\lim_{n\to+\infty} \frac{\LRN(v_n)-\LRN(u_n)}{\hat{a}_n|\lambda_1^n|} 	\geq 	\frac{1}{2}  |\lambda|^2 \kappa\int_{ B\sbr{0,R_\delta}}(U^\delta(x))^2 \dx >0,
			\end{equation}
			which implies that the relation \eqref{q6} holds.
			\vskip 0.15in 
			
			Finally, from the previous arguments, it suffices to {\bf prove the claim}.  In fact, since $u_n \in \MR_{z_1^n, \ldots, z_{k_n}^n}^n$ and $v_n \in \MR_{y_1^n, \ldots, y_{k_n}^n}^n$, by Corollary \ref{coro 4.2}, we have 	 
			\begin{equation} 	 \label{g2}
				-\Delta u_n(x) +  V_n(x) u_n (x) = u_n(x)  \log  u_n^2(x)   +  \sum_{i=1}^{k_n}u_n^{\delta,i} (x)  \mbr{  \lambda_i^n\cdot \sbr{x-z_i^n}},  
			\end{equation}
			\begin{equation} 	 \label{g3}
				-\Delta v_n(x) +  V_n(x)v_n (x) = v_n(x)  \log  v_n^2(x)   +  \sum_{i=1}^{k_n}v_n^{\delta,i} (x)  \mbr{  \tilde{\lambda}_i^n\cdot \sbr{x-y_i^n}}  ,
			\end{equation}
			where we denote $\lambda_i^n\in \RN$ and $\tilde{\lambda}_i^n \in \RN$  as  the Lagrange multipliers associated with $u_n$ and $v_n$, respectively. 
			Arguing as in \eqref{r97}, we have 
			\begin{equation} \label{h5}
				\lim_{n\to+\infty} \lambda_i^n  =0  \quad \text{ and } \quad  \lim_{n\to+\infty}	\tilde{\lambda}_i^n  =0.
			\end{equation}
			Moreover, from \eqref{g2} and  \eqref{g3}, along with \eqref{q10}, we  obtain
			\begin{equation} \label{g1}
				\begin{aligned}
					&	-\Delta (v_n -u_n )(x) +V_n(x)(v_n -u_n)(x)\\
					&=( v_n \log  v_n^2 -u_n  \log  u_n^2)(x)+ \sum_{i=1}^{k_n}  u_n^{\delta,i}(x) (\tilde{\lambda}_i^n -\lambda_i^n)\cdot (x-z_i^n)\\
					&	+\sum_{i=1}^{k_n} (v_n^{\delta,i}-u_n^{\delta,i})(x) \tilde{\lambda}_i^n  \cdot (x-z_i^n) + v_n^{\delta,1}(x) \tilde{\lambda}_1^n\cdot ( z_1^n-y_1^n).
				\end{aligned}
			\end{equation}
			For any fixed $j\in \lbr{1,\ldots,k_n}$,  we set $\hat{a}_{n,j}:=\max \left\lbrace \hat{a}_n,|\tilde{\lambda}_j^n -\lambda_j^n|\right\rbrace $ and 
			\begin{equation} \label{h1}
				\hat{\phi}_{n,j}(x):=\cfrac{ v_n(x+z_j^n)-u_n(x+z_j^n)}{\hat{a}_{n,j}}\ .
			\end{equation}
			Obviously, there holds 
			\begin{equation} \label{h6}
				|\hat{\phi}_{n,j}	|_\infty \leq 1  \quad \text{ and } \quad\frac{|\tilde{\lambda}_j^n -\lambda_j^n|}{\hat{a}_{n,j} }\leq 1.
			\end{equation} It follows  from \eqref{g1} that 
			\begin{equation}\label{h8}
				\begin{aligned}
					&	-\Delta 	\hat{\phi}_{n,j}(x) +V_n(x+z_j^n)	\hat{\phi}_{n,j}(x)\\
					&=	b_{n,j}(x+z_j^n)\hat{\phi}_{n,j}(x)+ \sum_{i=1}^{k_n}  u_n^{\delta,i}(x+z_j^n) \sbr{\frac{\tilde{\lambda}_i^n -\lambda_i^n}{\hat{a}_{n,j}}\cdot (x+z_j^n-z_i^n)}\\
					&	+\sum_{i=1}^{k_n} 	\hat{\phi}_{n,j}^{\delta,i}(x+z_j^n)  \sbr{\tilde{\lambda}_i^n  \cdot (x+z_j^n-z_i^n) }\\
					&+ v_n^{\delta,1}(x+z_j^n) \sbr{\frac{\tilde{\lambda}_1^n}{\hat{a}_{n,j}}\cdot ( z_1^n-y_1^n)},
				\end{aligned}
			\end{equation}
			where  
			\begin{equation}
				\hat{\phi}_{n,j}^{\delta,i}(x+z_j^n):=\frac{ v_n^{\delta,i}(x+z_j^n)-u_n^{\delta,i} (x+z_j^n)}{\hat{a}_{n,j}} ~,
			\end{equation}
			and 
			\begin{equation} \label{g4}
				\begin{aligned}
					b_{n,j}(x)& :=\frac{v_n(x) \log  v_n^2(x) -u_n(x)  \log  u_n^2(x)}{v_n(x)-u_n(x)} 
					= 2+\log \sbr{l_n(x)}^2  ,
				\end{aligned}
			\end{equation}
			with  $l_n(x)= u_n(x)+\tilde{l}(x)(v_n(x)-u_n(x))$ for some $\tilde{l}(x) \in \mbr{0,1}$.  
			\medbreak
			By the definition of barycenter map \eqref{barycenter map}, we have 
			\begin{equation}
				z_1^n-y_1^n= 	 \frac{1}{|u_n^{\delta,1}|_2^2}\int_{\RN} x(u_n^{\delta,1}(x))^2 ~\mathrm{d}x- \frac{1}{|v_n^{\delta,1}|_2^2}\int_{\RN} x(v_n^{\delta,1}(x))^2 ~\mathrm{d}x.
			\end{equation}
			Considering that  $|u_n^{\delta,1}(x)-v_n^{\delta,1}(x)|\leq \hat{a}_n$ and $  u_n^{\delta,1} (x+z_1^n) \xrightarrow{n \rightarrow  +\infty} U^{\delta}(x)  $, along with $  v_n^{\delta,1} (x+y_1^n) \xrightarrow{n \rightarrow  +\infty} U^{\delta}(x)  $, a direct computation  leads us to infer that for  sufficiently large $n$,  
			\begin{equation} \label{h3}
				|	z_1^n-y_1^n| \leq C \hat{a}_n.
			\end{equation}
			Then for all $r>R_\delta$  and for $n $ large enough, we have $$r>R_\delta+|	z_1^n-y_1^n|=R_\delta+\sigma_n \lambda.$$	Observe that   
			\begin{equation}
				-\Delta 	\hat{\phi}_{n,j}(x) +V_n(x+z_j^n)	\hat{\phi}_{n,j}(x) =	b_{n,j}(x+z_j^n)\hat{\phi}_{n,j}(x), \quad \forall x\in B(0,2r)\setminus    B(0,R_\delta+\sigma_n \lambda).
			\end{equation}
			By   \eqref{defi of delta}, we deduce  that $ b_{n,j} \leq 2+\log \delta^2<0$.  Since $|	\hat{\phi}_{n,j}|$ satisfies that 
			\begin{equation}
				\begin{aligned}
					-\Delta 	|\hat{\phi}_{n,j}| + V_0	|\hat{\phi}_{n,j}| & \leq 	-\Delta 	|\hat{\phi}_{n,j}| + V_n(x+z_j^n)	|\hat{\phi}_{n,j}|  =b_{n,j}(x+z_j^n)|\hat{\phi}_{n,j}	|<0 
				\end{aligned}	 
			\end{equation}
			in $B(0,2r)\setminus    B(0,R_\delta+\sigma_n \lambda)$, then using \cite[Lemma 3.3]{cerami_CPMA_2013}, we have that 
			\begin{equation}\label{h2}
				|\hat{\phi}_{n,j}(x)|\leq \widetilde{C}   e^{-\tilde{\bar \zeta} \bar d(x)}, \quad \forall x\in B(0,2r)\setminus    B(0,R_\delta+\sigma_n \lambda)
			\end{equation}
			for $\tilde{\bar \zeta} \in \sbr{0,\sqrt{V_0}}$ and  $\bar d(x)=\mathrm{dist}(x, \supp{ |\hat{\phi}_{n,j}|^\delta})$. Moreover, from \eqref{h1}, we know that  $|\hat{\phi}_{n,j}	|_\infty \leq 1$. This, combined with   \eqref{h2}, implies  that  there exists a constant $\widehat{c}_1>0$ (independent of  $n$ and $r$) such that for all $r>0$ and  for $n$ large enough,
			\begin{equation} \label{h7}
				\int_{ B(0,2r) } |\hat{\phi}_{n,j}(x)| \dx \leq \widehat{c}_1.
			\end{equation}	 
			It follows from \eqref{h5}, \eqref{h6}, \eqref{h3},  \eqref{h7} and  \eqref{h8}, along with Proposition \ref{prop 4.2} 	 that  there exists a constant $\widehat{c}_2>0$ (independent of  $n$ and $r$) such that   for $n$ large enough,
			\begin{equation}
				\int_{ B(0,2r) } |\Delta \hat{\phi}_{n,j}(x)||  \hat{\phi}_{n,j}(x)| \dx \leq \widehat{c}_2.
			\end{equation}
			Let $ \tilde{h} \in \DR(B(0,2))$ be a positive function such that $\tilde{h}=1$ on $B(0,1)$.	Set $\tilde{h}_r(x)=\tilde{h}(x/r)$, then $\int_{ B(0,2r) } |\Delta \tilde{h}_r|\dx =\widehat{c}_3r^{N-2}$. Therefore, we have 
			\begin{equation}
				\begin{aligned}
					\int_{ B(0,r)  }|\nabla \hat{\phi}_{n,j}|^2\dx &\leq \int_{\R^N } |\nabla \hat{\phi}_{n,j}|^2\tilde{h}_r \dx =-\int_{\R^N }  \hat{\phi}_{n,j}\nabla(\tilde{h}_r\nabla \hat{\phi}_{n,j} ) \dx\\
					&=- \frac{1}{2}		\int_{\R^N }   \nabla \tilde{h}_r \nabla\hat{\phi}_{n,j} ^2 \dx -  	\int_{\R^N }  \hat{\phi}_{n,j}\tilde{h}_r\Delta \hat{\phi}_{n,j} \dx\\
					&\leq \int_{  B(0,2r) }|\hat{\phi}_{n,j}| |\Delta \hat{\phi}_{n,j}| \dx +  \frac{1}{2}		\int_{ B(0,2r)}\Delta \tilde{h}_r \hat{\phi}_{n,j} ^2 \dx \\
					&\leq \widehat{c}_2+ \widehat{c}_4 e^{-2\tilde{\bar \zeta} r} r^{N-2},
				\end{aligned}
			\end{equation}
			which implies that 	$\lbr{\hat{\phi}_{n,j}}_n$ is bounded in $H_{loc}^1(\RN)$.
			
			\medbreak 
			Therefore,  we may assume that  there exists $\phi_j $ such that $\hat{\phi}_{n,j} \rightharpoonup \phi_j $ in $H_{loc}^1(\RN)$. It follows from 
			\eqref{h6} that  $ |\hat{\phi}_{n,j}-\phi_j |_{p,\loc} \to 0 $ for all $p<+\infty$. Then we can pass the limit  in \eqref{h8} obtaining
			\begin{equation} \label{y1}
				-\Delta \phi_j(x) +( V_\infty-2- \log U^2) \phi_j(x) =U^\delta (x) (\hat{\lambda}_j\cdot x) ,
			\end{equation}
			where $\hat{\lambda}_j= \lim_{n\to+\infty}   (\tilde{\lambda}_j^n -\lambda_j^n)/ {\hat{a}_{n,j}}$.
			\medbreak 
			At the conclusion of the proof of the claim, it remains to prove  that $\hat{\lambda}_j=0$. Indeed, in this case, for $n$ sufficiently large, we have  $\hat{a}_{n,j}=\hat{a}_n$. Therefore, the proof provided above for $\hat{\phi}_{n,j} $ corresponds exactly to  the statement stated within the claim. Set $U^\delta (x) (\hat{\lambda}_j\cdot x):= H(x)$. Since \eqref{y1} holds, for any $\nu \in \RN\setminus \lbr{0}$,  by a direct calculation, we deduce that 
			\begin{equation}\label{y2}
				\begin{aligned}
					&\int_{\RN} H(x) (\nabla U \cdot \nu) \dx \\
					=&\int_{\RN} -\Delta \phi_j(\nabla U \cdot \nu)  \dx +\int_{\RN} ( V_\infty-2- \log U^2)  (\nabla U \cdot \nu)\phi_j  \dx\\
					= &\int_{\RN} -\Delta \phi_j(\nabla U \cdot \nu)  \dx+\int_{\RN} \nabla ( V_\infty U -U\log U^2)\cdot ( \nu\phi_j)\dx\\
					=&\int_{\RN} -\Delta \phi_j(\nabla U \cdot \nu)  \dx +\int_{\RN} \nabla ( \Delta U)\cdot ( \nu\phi_j)\dx\\
					=&~0.
				\end{aligned}
			\end{equation}
			On the other hand, choosing  $\nu =\hat{\lambda}_j$ in \eqref{y2}, we have 
			\begin{equation}
				\begin{aligned}
					0&=	\int_{\RN} H(x) (\nabla U \cdot \hat{\lambda}_j) \dx= 	\int_{\RN} U^\delta (\nabla U \cdot \hat{\lambda}_j)(\hat{\lambda}_j\cdot x)  \dx \\
					&=\frac{1}{2} 	\int_{\RN} \mbr{\nabla\sbr{ U^\delta}^2 \cdot \hat{\lambda}_j }(\hat{\lambda}_j\cdot x)  \dx =-\frac{1}{2} |\hat{\lambda}_j|^2	\int_{\RN} 
					\sbr{ U^\delta}^2 \dx,
				\end{aligned}
			\end{equation}
			which implies that $\hat{\lambda}_j=0$. This completes the proof.
		\end{proof}

		\section{ Existence of a solution  with infinitely many bumps}\label{Sect6}
		This section is devoted to the proof of Theorem \ref{thm2} by studying the asymptotic behavior of   the multi-bump positive solutions  obtained in Theorem  \ref{thm1} as   the number of positive    bumps  increases going to infinity. Our proof of Theorems \ref{thm2} is  inspired by the original idea of \cite{cerami_Poincare_2015}, but as we will see,    some key modifications are needed.     In  Proposition \ref{prop 6.1}, we  establish the existence of a  positive solution of infinite energy that has infinitely many bumps.  Propositions \ref{prop 6.2} and \ref{prop 6.3} are focused on proving \eqref{et1} and \eqref{et2}, respectively.
		
		\medbreak
		By Theorem \ref{thm1}, we obtain a sequence $\lbr{\tilde{u}_k}_k$ of solutions to equation \eqref{p equ} such that $\tilde{u}_k$ emerges around $k$-tuples  $\sbr{\tilde z_1^k,\ldots,\tilde z_k^k } \in \ER_{k}$.  Moreover, we have 
		\begin{equation} \label{m4}
			\varLambda_{k}= \alpha\sbr{\tilde z_1^k,\ldots,\tilde z_k^k } =	\LR(\tilde{u}_{k}) .
		\end{equation} 	 
		For all $r>0$, we denote by $\ga(\tilde{u}_k,r)$ the number of points around which $\tilde{u}_k$ is emerging and that are contained in $B(0,r)$.
		\vskip0.1in
		\begin{lemma} \label{lemma 6.1}
			For all $  m \in \N$, there exist a real number $r_m >0$ and a  positive integer $k_m\in \N$ such that for all $k>k_m$, there holds
			\begin{equation}
				\ga(\tilde{u}_k,r_m)\geq m.
			\end{equation}
		\end{lemma}
		\vskip0.1in
		\begin{proof}
			By contradiction, we may assume that there exists $\tilde{m}\in \N$ and sequences $\lbr{r_n}_n$, $\lbr{k_n}_n  $  with $r_n>0$, $k_n\in \N$, such that $r_n, k_n\to +\infty$ as $n \to +\infty$ and  a sequence of positive solutions  $\lbr{\tilde{u}_{k_n}}_{k_n}$ that emerges  around the  $k_n$-tuples  $\sbr{\tilde z_1^n,\ldots,\tilde z_{k_n}^n }\in \ER_{k_n}$ such that
			\begin{equation}
				\ga(\tilde{u}_{k_n},r_n)<  \tilde{m}, \quad \forall n\in \N.
			\end{equation}
			Then we may assume that, up to subsequence,  there exists $j<\tilde{m}$ such that as $n\to +\infty$
			\begin{equation}
				\tilde 	z_i^n \to \tilde  z_i ,\quad  \forall i \leq j, \quad\text{ and } \quad 	|\tilde  z_i^n| \to +\infty ,\quad  \forall i >j.
			\end{equation}
			It follows from Remark \ref{Remark 4.3} that there exists $y\in\RN$ such that $\sbr{\tilde z_1,\ldots,\tilde z_{j},y} \in \ER_{j+1}$ and 
			\begin{equation} \label{m0}
				\alpha(\tilde z_1 ,\ldots,\tilde z_{ j} ) +\CR^\infty <\alpha(\tilde z_1 ,\ldots,\tilde z_{j} ,y).
			\end{equation} 
			Now we claim that 
			\begin{equation} \label{me1}
				\limsup_{n\to+\infty} \mbr{	\alpha(\tilde z_1^n,\ldots,\tilde z_{k_n}^n)  -	\alpha(\tilde z_{j+2}^n,\ldots,\tilde z_{k_n}^n) } \leq  \alpha(\tilde z_1 ,\ldots,\tilde z_{j} ) +\CR^\infty, 
			\end{equation}
			and 
			\begin{equation}\label{me2}
				\lim_{n\to+\infty} \mbr{	\alpha(\tilde z_1^n,\ldots, \tilde z_j^n, y,\tilde z_{j+2}^n, \ldots  \tilde z_{k_n}^n)  -	\alpha(\tilde z_{j+2}^n,\ldots,\tilde z_{k_n}^n) } = \alpha(\tilde z_1 ,\ldots,\tilde z_{j},y ). 
			\end{equation}
			Once we  prove  that  the claim is true, the proof of Lemma \ref{lemma 6.1} will be concluded. Indeed, from \eqref{m4}, \eqref{m0}, \eqref{me1} and \eqref{me2}, we deduce that  for $n$ sufficiently large, 
			\begin{equation}
				\varLambda_{k_n} =	\alpha(\tilde z_1^n,\ldots,  \tilde z_{k_n}^n)  <	\alpha(\tilde z_1^n,\ldots, \tilde z_j^n, y,\tilde z_{j+2}^n, \ldots  \tilde z_{k_n}^n) \leq 	\varLambda_{k_n}  ,
			\end{equation}
			which  is a impossible.
			
			\medbreak
			To prove \eqref{me1}, we consider 
			\begin{equation}
				f_n \in \MR_{ \tilde z_1^n,\ldots, \tilde z_j^n}, \quad g_n \in \MR_{\tilde z_{j+1}^n}, \quad h_n  \in \MR_{ \tilde z_{j+2}^n,\ldots, \tilde z_{k_n}^n}.
			\end{equation}
			Then $f_n \vee g_n\vee h_n \in   \SR_{ \tilde z_1^n,\ldots, \tilde z_{k_n}^n}$. Moreover,     $0\leq  g_n\wedge h_n\leq \delta $ and $0\leq f_n \wedge \sbr{g_n\vee h_n} \leq \delta $. Therefore, we deduce from Lemma \ref{lemma 2.5} that 
			\begin{equation}\label{kd1}
				\begin{aligned}
					\LR(f_n \vee g_n\vee h_n) &=\LR(f_n  )+  	\LR(   g_n\vee h_n)- 	\LR(f_n \wedge \sbr{g_n\vee h_n}) \\
					&=\LR(f_n  )+  \LR( g_n )+	\LR( h_n) -	\LR(   g_n\wedge h_n)- 	\LR(f_n \wedge \sbr{g_n\vee h_n})\\
					&\leq  \LR(f_n  )+  \LR( g_n )+	\LR( h_n).
				\end{aligned}
			\end{equation} 
			Hence, 
			by Lemma  \ref{lemma 3.5} and Remark \ref{Remark 3.4}, we infer that 
			\begin{equation}
				\begin{aligned}
					\alpha(\tilde z_1^n,\ldots,\tilde z_{k_n}^n) &   \leq 	\LR(f_n \vee g_n\vee h_n) \\&   \leq 	\LR(f_n  )+	\LR( g_n )+	\LR( h_n)\\
					&=  \alpha\sbr{ \tilde z_1^n ,\ldots, \tilde z_j^n } +  \alpha(\tilde z_{j+1}^n) +\alpha(\tilde z_{j+2}^n,\ldots,\tilde z_{k_n}^n)\\
					& \leq \alpha\sbr{ \tilde z_1 ,\ldots, \tilde z_j } +\CR^\infty +\alpha(\tilde z_{j+2}^n,\ldots,\tilde z_{k_n}^n) +o_n(1),
				\end{aligned}
			\end{equation}
			which implies that \eqref{me1} holds.
			
			\medbreak
			
			To prove \eqref{me2}, we consider 
			\begin{equation}
				\tilde	f_n \in \MR_{\tilde z_1^n,\ldots, \tilde z_j^n, y,\tilde z_{j+2}^n, \ldots  \tilde z_{k_n}^n}, \quad   \tilde g_n \in \MR_{\tilde z_1^n,\ldots, \tilde z_j^n, y} .
			\end{equation}
			Then by  Lemmas \ref{lemma 2.5} and  \ref{lemma 3.5}, we deduce that 
			\begin{equation}
				\begin{aligned}
					\alpha(\tilde z_1^n,\ldots, \tilde z_j^n, y,\tilde z_{j+2}^n, \ldots  \tilde z_{k_n}^n)& \leq \LR(  \tilde g_n \vee h_n )=\LR(\tilde g_n  )+	\LR(h_n ) -\LR(\tilde g_n  \wedge h_n )
					\\&\leq \LR(  \tilde g_n  ) +\LR(    h_n )
					\\ &=	\alpha(\tilde z_1^n,\ldots, \tilde z_j^n, y)
					+	\alpha(\tilde z_{j+2}^n,\ldots,\tilde z_{k_n}^n) \\
					& = 	\alpha\sbr{ \tilde z_1 ,\ldots, \tilde z_j , y}  +\alpha(\tilde z_{j+2}^n,\ldots,\tilde z_{k_n}^n) +o_n(1)
				\end{aligned} 
			\end{equation}
			Hence, we have 
			\begin{equation}\label{m10}
				\limsup_{n\to+\infty} \mbr{ 	\alpha(\tilde z_1^n,\ldots, \tilde z_j^n, y,\tilde z_{j+2}^n, \ldots  \tilde z_{k_n}^n)-\alpha(\tilde z_{j+2}^n,\ldots,\tilde z_{k_n}^n) } \leq \alpha\sbr{ \tilde z_1 ,\ldots, \tilde z_j , y}.
			\end{equation}
			To prove the reserve inequality, we  set
			\begin{equation}
				\rho_n:= \min \lbr{|\tilde z_i^n|:i=j+2,\ldots,n}.
			\end{equation}
			Then $\rho_n \to +\infty$ as $n\to+ \infty$ and for $n$ sufficiently large, we have 
			\begin{equation}
				\bigcup_{i=1}^j B(\tilde z_i^n,R_\delta) \cup B(y,R_\delta)  \subset B(0,\frac{\rho_n}{2}-1),
			\end{equation}
			\begin{equation}
				\bigcup_{i=j+2}^{k_n} B(\tilde z_i^n,R_\delta)  \subset \RN \setminus B(0,\frac{\rho_n}{2}+1).
			\end{equation}
			Let $ \widetilde{F}_n(x):= \tilde \chi_n(x)\tilde{f}_n(x)$, where $\tilde \chi_n=\chi\sbr{  |x| -\frac{\rho_n}{2}  }$ and   $\chi(x)$     is given by \eqref{defi of chi}. Using a similar argument as in the proof of \eqref{e12}, we deduce  
			\begin{equation}  \label{m11}
				\begin{aligned}
					\LR(\widetilde{F}  _n)  
					&\leq  \LR(\tilde f_n)  + O(e^{-\bar \zeta \rho_n^2})\\&= 	\alpha(\tilde z_1^n,\ldots, \tilde z_j^n, y,\tilde z_{j+2}^n, \ldots  \tilde z_{k_n}^n) + O(e^{-\bar \zeta \rho_n^2}),
				\end{aligned}
			\end{equation}
			Similar to \eqref{m8} and \eqref{m9}, we can decompose $\widetilde{F}_n$ into  $$
			\widetilde{F}  _n(x)= \widetilde{F}_{n,1}(x)+    \widetilde{F}_{n,2}(x)$$ with $\widetilde{F}_{n,1} \in \SR_{ \tilde z_1^n,\ldots, \tilde z_j^n, y  }$,~$\widetilde{F}_{n,2}  \in \SR_{  \tilde z_{j+2}^n,\ldots, \tilde z_{k_n}^n}$ and $\supp{\widetilde{F}_{n,1}}\cap \supp{\widetilde{F}_{n,2}}=\emptyset$. 
			Therefore, we  have  
			\begin{equation} \label{m12}
				\LR(\widetilde{F}_n)=	\LR( \widetilde{F}_{n,1}) +\LR( \widetilde{F}_{n,2}) \geq \alpha( \tilde z_1^n,\ldots, \tilde z_j^n, y ) + \alpha ( \tilde z_{j+2}^n,\ldots, \tilde z_{k_n}^n) .
			\end{equation}
			Hence, we deduce from \eqref{m11}, \eqref{m12} and Lemma  \ref{lemma 3.5} that 
			\begin{equation} \begin{aligned}
					& \liminf_{n\to+\infty} \mbr{ 	\alpha(\tilde z_1^n,\ldots, \tilde z_j^n, y,\tilde z_{j+2}^n, \ldots  \tilde z_{k_n}^n)-\alpha(\tilde z_{j+2}^n,\ldots,\tilde z_{k_n}^n) } \\ \geq &	\liminf_{n\to+\infty}  \mbr{\alpha( \tilde z_1^n,\ldots, \tilde z_j^n, y )- O(e^{-\bar \zeta \rho_n^2}) }\\  =&\alpha( \tilde z_1,\ldots, \tilde z_j, y ),
			\end{aligned}\end{equation}
			which, together with \eqref{m10}, gives \eqref{me2}. The proof is completed.
		\end{proof}
		\vskip0.15in
		
		\begin{proposition} \label{prop 6.1}
			Under the assumptions of Theorem \ref{thm1}, there exists a positive solution $\tilde{u}$  of equation \eqref{p equ}, which emerges around an unbounded sequence $\lbr{\tilde z_n }_n$ such that $|\tilde z_n-\tilde z_l|\geq 3R_\delta$ if $n\neq l$.
		\end{proposition}
		
		\begin{proof}
			From Lemma \ref{lemma 6.1},  for all $  m \in \N$, there exist  $r_m >0$ and  $k_m\in \N$ such that $	\ga(\tilde{u}_k,r_m)\geq m$ for all $k>k_m$. Moreover, since $\tilde{u}_k $ emerges around the $k$-tuples $\sbr{\tilde z_1^k,\ldots,\tilde z_k^k} \in \ER_{k}$ and $|\tilde z_i^k-\tilde z_j^k|\geq 3R_\delta$, there exists $T_m\in \N$ and $T_m \geq m$ such that $\ga(\tilde{u}_k,r_m) \leq T_m$ for all $k\in \N$. Hence, $r_m \to +\infty$ as $m\to +\infty$, and without loss of generality, we may assume that $r_m\leq r_{m+1}$ for all $m\in \N$.
			\medbreak
			In what follows, let  $\lbr{\tilde{u}_{k_n}^m}_n$  be  a subsequence of $\lbr{\tilde{u}_{k}}_{k}$ such that 
			\begin{center}
				$\forall m\in \N$, ~ $\forall n\in \N$, ~there exists $H_m\in \N$  with $ m\leq H_m \leq T_m$ such that $\ga(\tilde{u}_{k_n}^m,r_m)=H_m$.
			\end{center}
			\begin{center}
				$\forall m\in \N$, ~ the sequence   $\lbr{\tilde{u}_{k_n}^{m+1}}_n$ is a subsequence of  $\lbr{\tilde{u}_{k_n}^m}_n$.
			\end{center}
			Set \begin{equation}
				v_n:= \tilde{u}_{k_n}^n.
			\end{equation}
			Then   $\lbr{v_n}_n$ is subsequence of $\lbr{\tilde{u}_{k_n}^m}_n$ for all $m \in \N$.
			Furthermore, the sequence of points around  which $\lbr{v_n}_n$ are emerging, denoted by $\lbr{\sbr{ \tilde z_1^n,\ldots, \tilde  z_{k_n}^n}}_n$, are converging as $n \to +\infty$. Moreover, the limit points have interdistances greater or equal that $3R_\delta$ and make up an unbounded countable subset of $\RN$, which  we denote  by $$\varXi:=\lbr{\tilde z_n:n\in \N}. $$
			\medbreak
			It remains  to prove  that for any fixed $m\in\N$, the sequence  $ \left\| v_n\right\|_{B(0,r_m)} $ is bounded.    Indeed, if this assertion holds, then there exists a function $\tilde{u}\in H^1(B(0,r_m))$ such that 
			\begin{equation}
				v_n= \tilde{u}_{k_n}^n\rightharpoonup \tilde{u} \quad \text{ weakly in }~ H^1(B(0,r_m)) .
			\end{equation}
			Recall that  
			\begin{equation}\label{qq1}
				|s^2\log s^2| \leq Cs^{2-\tau}+Cs^{2+\tau},\quad  \text{ for}\quad \tau\in \sbr{0,1},
			\end{equation} then  by using  the dominated convergence theorem, one gets
			\begin{equation}
				\lim_{n\to+\infty} \int_{B(0,r_m)} v_n\varphi\log v_n^2 =\int_{B(0,r_m)} \tilde{u}\varphi\log \tilde{u}^2 \ \  \text{ for any } \varphi\in C_0^\infty(B(0,r_m)) .
			\end{equation}
			Hence, it is easy to verify that  $\tilde{u}$ is a solution of 
			\begin{equation}
				-\Delta \tilde{u} (x)+V(x)\tilde{u}(x)=\tilde{u}(x) \log \tilde{u}^2(x), \quad  x \in B(0,r_m).
			\end{equation}  
			Moreover, arguing as in the proof of Proposition \ref{prop 4.2}, we infer that passing to subsequence,  $\lbr{v_n}_n$ uniformly  converges on every compact subsets of $\RN$ since  $r_m \to \infty$ as $m\to +\infty$, and the solution $\tilde{u}$ has at least $m$ emerging parts around points in $B(0,r_m)$.
			\medbreak
			Since the set $\varXi$ is unbounded,   then  $\varXi \cap B(0, r_m+\frac{4}{3}R_\delta)$ is finite, we assume that there exists $\bar m \geq m$ such that 
			\begin{equation}
				\varXi \cap B(0, r_m+\frac{4}{3}R_\delta)=\lbr{ \tilde z_1, \ldots,  \tilde z_{\bar m}}.
			\end{equation}
			Consider the set $$A_m:= B(0,r_m) \setminus \bigcup_{i=1}^{\bar m} B( \tilde z_i,\frac{4}{3}R_\delta).$$  Since $\tilde z_j^n \to \tilde z_j$ as $n \to +\infty$,  
			for $n$ sufficiently large,  we have  $0\leq v_n(x) \leq \delta$ in $A_m$. Therefore,  the sequence $|v_n|_{q,A_m}$ is   bounded for all   $q\geq 1$  because the set $A_m$ is bounded in $\RN$. Finally,  by \eqref{qq1},   we know that  $v_n$ is  bounded in $H^1(A_m) $ since it  solves $\eqref{p equ}$ on the set $A_m$. 
			\medbreak 
			In the following, we  prove that  $v_n$ is   bounded in $H^1(B(\tilde{z}_i, \frac{5}{3}R_\delta))$ for all $i \in \lbr{1,\ldots,\bar m }$.   For  all fixed $j\in \lbr{1,\ldots,\bar m }$ and    $n\in \N$, we consider the positive  function $w_n\in H^1(\RN)$ such  that $\Delta w_n=0$ in $B(\tilde z_j, \frac{4}{3}R_\delta)\setminus \bar B(\tilde z_j^n, R_\delta) $  and 
			\begin{equation}
				w_n(x)= \begin{cases}
					v_n(x), &\quad \forall x\not\in    B(\tilde z_j, \frac{4}{3}R_\delta),\\
					\tilde	a_n\sbr{ v_1(x+\tilde{z}_j^1-\tilde z_j^n)}^\delta ,& \quad  \forall x \in B(\tilde z_j^n, R_\delta),
				\end{cases}
			\end{equation}
			where $\tilde a_n:= \xi_j( v_1(x+\tilde{z}_j^1-\tilde z_j^n) )\in \R^+\setminus \lbr{0} $.    By the max-min characterization of $v_n$ and the construction of $w_n$, we have 
			\begin{equation}\label{l12}
				\LR(w_n)\geq 	\LR(v_n).
			\end{equation}
			Since $\tilde z_j^n \to \tilde z_j$ as $n \to +\infty$, the sequence $\lbr{\tilde a_n}_n$  is bounded in $\R$.  Moreover, arguing as in   proof of  point (iii) of  Proposition \ref{prop 4.2},   we can prove that $v_n$ is   bounded in  $C^{1+\varepsilon}(B(  \tilde z_j,\frac{5}{3}R_\delta)\setminus B(\tilde z_j,\frac{4}{3}R_\delta))$  for all $i\in \lbr{1,\ldots,\bar m }$. Thus,  the sequence   $\nm{w_n }_{B(\tilde{z}_j, \frac{5}{3}R_\delta)}$  is also bounded.  Moreover, we note that the sequence   $|w_n^2 \log w_n^2|_{1, B(\tilde{z}_j, \frac{5}{3}R_\delta)}$  is bounded. Indeed, this fact can be  established   by applying the following Sobolev logarithmic inequality(\cite[Theorem 8.14]{lieb-loss})
			\begin{equation}
				\int_{\RN} u^2 \log u^2\mathrm{d} x \leq \frac{a}{\pi} \nt{ \nabla u}^2 +\sbr{\log \nt{u}^2-N(1+\log a)} \nt{u}^2  \ \ \text{ for } u\in H ^1(\RN) \  \text{ and } \  a>0,
			\end{equation} and  using the   inequality  $s^2\log s^2 \geq -e^{-1}$, where  $s\in  (0,+\infty)$.

			\medbreak
			The proof of  that   $v_n$ is   bounded in $H^1(B(\tilde{z}_j, \frac{5}{3}R_\delta))$ proceeds by contradiction. We assume the false statement:    there exists $R_0>0$, as defined in Corollary \ref{coro 2.2},  such that    for any $R>R_0$,  there  is some  $n\in\N$  for which   $\nm{ v_n  }_{B(\tilde{z}_j, \frac{5}{3}R_\delta)}\geq R$. 
			By the construction of $w_n$ and    \eqref{l12}, there is a constant $C_3>0$ such that  
			\begin{equation}
				0\geq  \LR(v_n)- \LR(w_n) = \frac{1}{2}\nm{ v_n  }_{B(\tilde{z}_j, \frac{5}{3}R_\delta)}^2  - \frac{1}{2} \int_{ B(\tilde{z}_j, \frac{5}{3}R_\delta)}  v_n ^2\log v_n  ^2 ~\mathrm{d}x - C_3.
			\end{equation}
			By   Corollary \ref{coro 2.2}, we have 
			\begin{equation}
				\nm{ v_n  }_{B(\tilde{z}_j, \frac{5}{3}R_\delta)}^2\leq 2C_3+ C(1+	\nm{ v_n  }_{B(\tilde{z}_j, \frac{5}{3}R_\delta)}), 
			\end{equation}
			which leads to a contradiction since $R$ can be  arbitrarily chosen.  This completes the proof. 
				\end{proof}

				\begin{proposition} \label{prop 6.2}
					Let  $\tilde{u}$ be a solution  of equation \eqref{p equ} as stated in Proposition \ref{prop 6.1}, which emerges around an unbounded sequence $\lbr{\tilde z_n}_n$ in $\RN$. Then 
					\begin{equation} \label{e111}
						\lim_{n\to+\infty} \min\lbr{ |\tilde{z}_n-\tilde{z}_l|: l  \in \N, ~ n\neq l}=+\infty.
					\end{equation}
				\end{proposition}
				\begin{proof}
					We proceed  by contradiction and assume that  the statement is false. Then there are two subsequences  $\lbr{y_n}_n$ and $\lbr{\tilde{y}_n}_n$  of  $\lbr{\tilde z_n}_n$ such that for all $n\in \N$,   $y_n\neq \tilde{y}_n$  and $\lbr{|y_n-\tilde{y}_n|}_n$  is bounded. 
					\medbreak
					For any fixed $R\geq R_\delta$ and for all $n\in \N$,  we can choose a solution $\tilde u_{k_n}$ of \eqref{p equ}, which emerges around     $\sbr{\tilde{z}_1^n , \ldots, \tilde{z}_{k_n}^n}$, such that 
					\begin{equation}
						\LR(\tilde u_{k_n}) = \alpha\sbr{\tilde{z}_1^n , \ldots, \tilde{z}_{k_n}^n} =\varLambda_{k_n}, \quad  \sup_{B(y_n,R)\cup B(\tilde y_n,R)} |\tilde u_{k_n} - \tilde{u}| \leq \frac{1}{n} .
					\end{equation}
					Therefore, without loss of generality, we may assume that there exist two points $\tilde{z}_1^n$ and $\tilde{z}_2^n $ around which $\tilde u_{k_n}$  is emerging,   such that 
					\begin{equation}\label{s11}
						\lim_{n\to+\infty} |\tilde{z}_1^n-y_n | =0, \quad 	\lim_{n\to+\infty} |\tilde{z}_2^n-\tilde y_n | =0.
					\end{equation}
					which implies that  $\lbr{|\tilde{z}_1^n-\tilde{z}_2^n|}_n$  is bounded.  Then using a similar argument as in the proof of  \eqref{e23}, we have 	
					\begin{equation}
						\limsup_{n\to+\infty}  \mbr{ \alpha(\tilde z_1^n,\tilde z_2^n, \ldots,\tilde  z_{k_n}^n ) -\alpha(\tilde z_2^n, \ldots, \tilde z_{k_n}^n )} <\CR^\infty.
					\end{equation}	  
					Moreover, by Remark \ref{Remark 4.3}, there exists $p_n\in\RN$ such that  $( p_n,\tilde z_2^n,\ldots,\tilde z_{k_n}^n) \in\ER_{k_n}$
					\begin{equation}
						\alpha(\tilde z_2^n,\ldots,\tilde z_{k_n}^n) +\CR^\infty <\alpha( p_n,\tilde z_2^n,\ldots,\tilde z_{k_n}^n).
					\end{equation}
					Hence, we deduce that for $n$ large enough, 
					\begin{equation}
						\varLambda_{k_n} 	=\alpha\sbr{\tilde{z}_1^n , \ldots, \tilde{z}_{k_n}^n} \leq  \alpha(\tilde z_2^n, \ldots, \tilde z_{k_n}^n)+\CR^\infty< \alpha( p_n,\tilde z_2^n,\ldots,\tilde z_{k_n}^n) \leq 	 \varLambda_{k_n}, 
					\end{equation}
					which is impossible. The proof is completed.
				\end{proof}

				\begin{proposition} \label{prop 6.3}
					Let  $\tilde{u}$ be a solution  of equation \eqref{p equ} as stated in Proposition \ref{prop 6.1}, which emerges around an unbounded sequence $\lbr{\tilde z_n:n\in \N}$ of $\RN$. Then 
					\begin{equation}
						\lim_{n\to+\infty} \tilde{u}(x+\tilde z_n)= U(x)
					\end{equation}
					uniformly on any compact subset $K\subset \RN$.
					
				\end{proposition}
				\begin{proof}
					For any fixed $R\geq R_\delta$ and for all $n\in \N$,  by the definition of $\tilde{u}$, we can choose a solution $\tilde u_{k_n}$ of equation \eqref{p equ}, which emerges around   $k_n$-tuples  $\sbr{\tilde{z}_1^n , \ldots, \tilde{z}_{k_n}^n}$, such that 
					\begin{equation}
						\sup_{B( \tilde{z}_n,R) } |\tilde u_{k_n} - \tilde{u}| \leq \frac{1}{n} .
					\end{equation}
					Therefore, without loss of generality, we may assume that there exists   a point $\tilde{z}_1^n$  around which $\tilde u_{k_n}$ is emerging,   such that 
					\begin{equation} \label{e22}
						\lim_{n\to+\infty} |\tilde{z}_1^n-\tilde{z}_n | =0. 
					\end{equation}
					Then from \eqref{e111}, we have 
					\begin{equation}
						\lim_{n\to+\infty} |\tilde{z}_1^n | =+\infty, \quad \lim_{n\to+\infty} \min \lbr{ |\tilde{z}_1^n -\tilde{z}_l^n |: l=2,\ldots,k_n } =+\infty .
					\end{equation}
					In the following, we only need to prove that for any fixed $R>0$,
					\begin{equation}\label{eq1}
						\tilde u_{k_n}(\cdot+ \tilde z_1^n)  \xrightarrow{n \rightarrow  +\infty}  U  ~~\text{  uniformly in }   K , ~\forall K\subset \subset B(0,R)  \text{ compact}.
					\end{equation}
					Indeed, once we proved this statement,  in view of \eqref{e22}, we obtain the desired conclusion.   The argument for proving \eqref{eq1}   is analogous to that of Proposition  \ref{prop 4.2}, nevertheless  we give the details for the sake of clarity and	completeness.
					\medbreak
					Set  $d_n= \min \lbr{|\tilde z_1^n-\tilde z_l^n|:   l=2,\ldots, k_n }$ and    $\tau_n(x)=\tau\sbr{\left| x- \tilde z_1^n \right|-\frac{d_n}{2}}$,  where    $\tau(x)\in C^\infty(\R , [0,1])$  such that \begin{equation} 
						\tau(t)= 0 ~~\text{ for } ~  |t|<\frac{1}{2} , ~~\text{ and } ~~ \tau(t)=1~~  \text{ for} ~ | t|>1 .
					\end{equation}
					Define $  f_n(x):=	\tau_n(x) \tilde u_{k_n}(x)$. 
					Since $d_n \to +\infty$, for large $n$,    we can decompose $f_n(x)$ into $$f_n(x)=f_{n,1}(x)+f_{n,2}(x)$$ with $f_{n,1}(x) \in \SR_{\tilde z_1^n}$, $f_{n,2}(x)\in \SR_{\tilde  z_2^n, \ldots, \tilde z_{k_n}^n}$,     $\supp{f_{n,1}}\cap \supp{f_{n,2}}=\emptyset$ and 
					\begin{equation}
						\supp{f_{n,1}} \subset	B(\tilde  z_1^n, R_\delta) \subset \lbr{x\in \RN:  \left| x-\tilde z_1^n \right|  <\frac{d_n}{2}-\frac{1}{2}},
					\end{equation}
					\begin{equation}
						\supp{f_{n,2}} \subset \bigcup_{i=2}^{k_n}	B( \tilde  z_i^n, R_\delta) \subset \lbr{x\in \RN:   \left| x-\tilde z_1^n \right| >\frac{d_n}{2}+\frac{1}{2}}.
					\end{equation}
					Arguing analogous  for proving \eqref{e12}, we have 
					\begin{equation}
						\LR(f_{n,1})+\LR(f_{n,2})	=\LR(f_n)   \leq \LR(\tilde  u_{k_n}) +O(e^{-\bar \zeta d_n^2})  .
					\end{equation}
					For any $    v_n \in \MR_{\tilde z_1^n} $, we define    $\bar v_n(x):=\varsigma\sbr{\left| x-\tilde z_1^n \right|-\frac{d_n}{2}} {v}_n(x)$, where    $\varsigma(x)\in C^\infty(\R , [0,1])$ such that 
					\begin{center}
						$	\varsigma(t)= 1$ ~~for ~~$|t| \leq \frac{1}{4}$, \quad and \quad  $\varsigma(t)=0$ ~~for~~  $|t| \geq \frac{1}{2}$,
					\end{center}
					then we have  $	\bar 	v_n \in \SR_{\tilde z_1^n}$, $ \bar v_n+f_{n,2} \in \SR_{\tilde z_1^n,\tilde z_2^n, \ldots, \tilde z_{k_n}^n}$,
					and
					\begin{equation}
						\LR (\tilde u_{k_n})= \alpha (\tilde z_1^n, \ldots,\tilde z_{k_n}^n)\leq \LR  (\bar v_n+f_{n,2})= \LR  (\bar v_n )+ \LR  (  f_{n,2}).
					\end{equation}
					Arguing analogous for proving   \eqref{e12} and using Remark \ref{Remark 3.4}, we obtain that
					\begin{equation}
						\LR(\bar v_n)\leq 	\LR(  v_n)+O(e^{-\bar \zeta d_n^2})= \alpha(\tilde z_1^n)+o_n(1)\leq \CR^\infty+o_n(1).
					\end{equation}
					Then 
					\begin{equation}\label{t2}
						\begin{aligned}
							\LR (f_{n,1})& \leq \LR (\tilde u_{k_n})-\LR (f_{n,2}) +o_n(1) \leq  \LR  (  {v}_n)+o_n(1) =\CR^\infty+o_n(1).
						\end{aligned}
					\end{equation}
					On the other hand,  observing  that  $  f_{n,1}^\infty:=(f_{n,1})_\delta+\xi^\infty(f_{n,1}) (f_{n,1})^\delta \in \SR_{\tilde z_1^n}^\infty$  and  applying an analogous argument to that used in proving  \eqref{e6} and \eqref{e16}, we deduce 
					\begin{equation} 
						\begin{aligned}
							\LR(f_{n,1}) \geq \LR( f_{n,1}^\infty)+\frac{1}{2} \int_{\RN} \vartheta(x) \sbr{f_{n,1}^\infty(x)}^2 ~\mathrm{d} x > \CR^\infty,
						\end{aligned}
					\end{equation}
					which, together with \eqref{t2}, gives that $  \LR(f_{n,1})= \CR^\infty+o_n(1)$.

					\medbreak
					Finally,   observe that $f_{n,1}(x)=\tilde u_{k_n}(x)$ in $B(\tilde z_1^n, \frac{d_n}{2}-1)$, then   $f_{n,1}$ solves the following equation   
					\begin{equation} \label{q22}
						-\Delta u+ V(x)u= u\log |u|^2  \quad   \text{in }~B(\tilde z_1^n, \frac{d_n}{2}-1).
					\end{equation}
					Note that $d_n\to +\infty$,   then  for any fixed $R>0$ and  sufficiently large $n$,  we have $ B(\tilde z_1^n,R) \subset B(\tilde z_1^n, \frac{d_n}{2}-1) $.  Arguing  analogous as in the proof of     point (ii) of  Proposition \ref{prop 4.2},   we  have   \begin{equation}
						f_{n,1}(\cdot+ \tilde z_1^n) \xrightarrow{n \rightarrow  +\infty} U \quad \text{ strongly in }~ H^1(B(0,R)).  
					\end{equation}Moreover, for any $\theta>0$,  there  exists $C_\theta>0$ such that 
					\begin{equation}
						-\Delta u \leq\tilde  l(u), \quad \tilde  l(s)=   -V_0s+C_\theta s^{1+\theta}.
					\end{equation}
					Observe that     $|\tilde l(s)|\leq C(1+|s|^{1+\theta})$ for all $s\in\R$ and some $C>0$,      by repeating the bootstrap argument of the proof of \cite[Lemma B.3]{Struwe-book}, one can show that $f_{n,1}(\cdot+ \tilde z_1^n)  \in L^p( B(0,R))$ for any $p>2$. Then,  by standard regularity arguments,  we know that  $f_{n,1}(\cdot+ \tilde z_1^n) \in C^{1+\varepsilon}(K)$ for any  compact set $  K \subset \subset B(0, R)$. Moreover,  the sequence ${f_{n,1}(\cdot+ \tilde z_1^n)}$ is uniformly bounded in    $C^{1+\varepsilon}(K)$ for some $\varepsilon>0$. Then it is uniformly convergence to $U$  in all compact set $\forall K \subset \subset B(0, R)$.   This completes the proof.
				\end{proof}

				{\small \noindent \textbf{Acknowledgements:}
					
					The research of J. Wei is partially supported
					by GRF grant of HK RGC entitled "New frontiers in singular limits of elliptic and parabolic equations".
					
						 W. Zou is supported by National Key R\&D Program of China (Grant 2023YFA1010001) and NSFC (12571123).   T. Liu is supported by the   National Funded Postdoctoral Researcher Program    (GZB20240945) and China Postdoctoral Science
					Foundation  (2025M784442).
				}
				
				\medskip
				
				{\small \noindent \textbf{Statements and Declarations:} The authors have no relevant financial or non-financial interests to disclose.}
				
				\medskip
				
				{\small \noindent \textbf{Data availability:} Data sharing is not applicable to this article as no datasets were generated or analyzed during the current study.}
				\medskip
				
				{\small \noindent \textbf{Conflict of interest statement:} 	On behalf of all authors, the corresponding author states that there is no conflict of interest.}

			\end{document}